\documentclass[10pt, a4paper]{amsart}
\usepackage[margin=2cm]{geometry}
\usepackage{graphicx} 
\usepackage{amsmath}
\usepackage{amssymb}
\usepackage{stix} 
\usepackage{amsthm}
\allowdisplaybreaks
\usepackage{mathtools}
\usepackage{mathrsfs}
\usepackage{dsfont} 
\usepackage[english]{babel}
\usepackage{csquotes}
\usepackage{hyperref}
\usepackage{xcolor} 
\usepackage{verbatim} 
\mathtoolsset{showonlyrefs}

\newtheorem{thm}{Theorem}[section]
\newtheorem{cor}[thm]{Corollary}
\newtheorem{lem}[thm]{Lemma}
\newtheorem{prop}[thm]{Proposition}

\newtheorem*{question*}{Question}

\theoremstyle{definition}
\newtheorem{defn}[thm]{Definition}

\newtheorem{remark}[thm]{Remark}

\newtheorem*{remark*}{Remark}

\numberwithin{equation}{section}

\newcommand{\p}{\mathcal P}
\newcommand{\E}{\mathcal{E}}
\newcommand{\F}{\mathcal{F}}

\newcommand{\ve}{\varepsilon}
\renewcommand{\H}{\mathscr{H}}
\newcommand{\G}{\mathcal{G}} 
\newcommand{\D}{\Delta} 
\newcommand{\Zp}{\mathbb{N}}
\newcommand{\R}{\mathbb{R}}        
\renewcommand{\S}{\mathbb{S}}        

\newcommand{\RNp}{\mathbb{R}^{m+1}_{+}}

\newcommand{\eps}{\varepsilon}

\renewcommand{\d}{\mathrm{d}}

\newcommand{\abs}[1]{\left| #1 \right|}
\newcommand{\spt}{\textnormal{spt}}

\usepackage{crossreftools}

\begin{document}

\title[Defect measures and free boundaries]{Parabolic defect measures for Harmonic Maps with Free Boundary}
\author{Anirban Das}
\address{Tata Institute of Fundamental Research,
Centre for Applicable Mathematics, Bangalore, 560065, India}
\email{anirban21@tifrbng.res.in}
\author{Ali Hyder}
\address{Tata Institute of Fundamental Research,
Centre for Applicable Mathematics, Bangalore, 560065, India}
\email{hyder@tifrbng.res.in}
\author{Yannick Sire}
\address{Johns Hopkins University, Krieger Hall, 3400 N. Charles St., Baltimore, MD, 21218, USA}
\email{ysire1@jhu.edu}

	\begin{abstract}
	We introduce defect measures for the heat flow of a class of harmonic maps with free boundary and initiate their analysis. Our flow is motivated by singular perturbations on the boundary of the domain, emphasizing structural differences between harmonic maps with free boundary and their classical counterpart. The harmonic maps we consider are a reformulation of maps introduced by Da Lio and Rivi\`ere and arise in several problems in the spectral geometry of Steklov eigenvalues. Contrary to the classical case, the flow we consider exhibits features of global nature and the blow-up analysis turns out to be much more challenging. 
	\end{abstract}

\Large

\maketitle

\tableofcontents

\section{Introduction}

Suppose $M$ and $N$ are closed (i.e. compact without boundary) Riemannian manifolds. In their well known work \cite{ES}, Eells and Sampson studied the so called homotopy problem for harmonic maps which investigates that starting with a given smooth map $\phi : M \to N$, whether there exists a harmonic map $v : M \to N$ that is homotopic to $\phi$. Their idea was to consider the following heat flow of harmonic maps
\begin{align}\label{classical heat flow}
\begin{cases}
&\partial_t u = \tau(u), \text{ on } M \times (0, \infty),\\
&u(\cdot, 0) = \phi, \text{ on } M,
\end{cases}
\end{align}
where for any $t > 0$, $\tau (u(\cdot, t))$ denotes the tension field of $u(\cdot, t)$. They proved that when the sectional curvature of $N$ is non-positive then \eqref{classical heat flow} admits a global smooth solution and hence the homotopy problem is solved in that case with an affirmative answer. Later Struwe in \cite{S} tried to remove the curvature assumption from the target manifold $N$ by considering weak solutions of \eqref{classical heat flow}. However in this work the domain was considered to be two dimensional. In order to show existence and partial regularity of solution of \eqref{classical heat flow} from higher dimensional domains, Chen and Struwe used Ginzburg-Landau approximation in \cite{CS}. In \cite{LW-1} using Ginzburg-Landau approximations Lin and Wang recovered the result by Eells and Sampson in the case when the target manifold is negatively curved.  

Another important geometric variational problem is the so called harmonic map system with free boundary. More precisely, let $(M, g)$ be an $m$-dimensional Riemannian manifold with non empty boundary $\partial M$ and $Z$ be an $\ell$-dimensional Riemannian manifold without boundary. Suppose $N$ is a $k$-dimensional submanifold in $Z$ without boundary. Any continuous map $u_0: M\to Z$ satisfying $u_0(\partial M)\subset N$ defines a relative homotopy class in maps from $(M, \partial M)$ to $(Z, N)$. A map $u: M\to Z$ with $u(\partial M)\subset N$ is called homotopic to $u_0$ if there exists a continuous homotopy $h:[0, 1]\times M \to Z$ satisfying $h([0,1]\times \partial M)\subset N$, $h(0) = u_0$ and $h(1) = u$. An interesting problem is that whether or not each {\sl relative} homotopy class of maps has a representation by harmonic maps, which is a stationary solution  to the following parabolic PDE,
\begin{align}\label{FB heat flow}
\begin{cases}
&\partial_t u = \tau(u), \text{ on } M \times (0, \infty),\\
&u(\partial M \times (0, \infty))\subset N,\\
&\frac{\partial u}{\partial \nu}\perp T_u N,\\
&u(\cdot, 0) = \phi, \text{ on } M.
\end{cases}
\end{align}
Here $\nu$ is the unit normal vector of $M$ along the boundary $\partial M$ and $T_pN$ is the tangent space of $N$ at $p$ and $\perp$ means orthogonal in $T_p Z$. The stationary system ($\partial_t u = 0$) constitutes the Euler-Lagrange equations for critical points of the Dirichlet energy functional
\begin{equation*}
E(u) = \frac 1 2 \int_{M}|\nabla u|^2dM
\end{equation*}
defined on the space of maps  
\begin{equation*}
H^1_N(M, Z) = \{u\in H^1(M, Z): u(\partial M)\subset N\}.
\end{equation*}

Existence and partial regularity of energy minimizing harmonic maps in $H^1_N(M, Z)$ have been established (for example, in \cite{BaldesMM1982}, \cite{DuzaarSteffen1989JRAM}, \cite{DuzaarSteffenAA1989}, \cite{GulliverJostJRAM1987}, \cite{HardtLinCPAM1989}). As for  the parabolic problem, problem \eqref{FB heat flow} was first studied by Ma \cite{MaLiCMH1991} in the case $m = \textnormal{dim} M = 2$, where a global existence and uniqueness result for finite energy weak solutions was obtained under {suitable} geometrical hypotheses on $Z$ and $N$. Global existence for weak solutions of (\ref{FB heat flow}) was established by Struwe in \cite{StruweManMath1991} for $m\ge 3$. 
In \cite{Hamilton1975LNM}, Hamilton considered the case when $\partial Z = N$ is totally geodesic and the sectional curvature $K_Z\leq 0$.
He proved the existence of a unique global smooth solution for (\ref{FB heat flow}). When $Z$ is a Euclidean space, the first equation in (\ref{FB heat flow}) is the standard heat equation
\begin{equation}
u_t - \Delta u = 0 \text{ on }M \times [0, \infty).
\end{equation}
As pointed out in \cite{chen-lin} and \cite{StruweManMath1991}, estimates near the boundary for (\ref{FB heat flow}) are difficult because of the highly nonlinear boundary conditions. Struwe in \cite{StruweManMath1991}  used a Ginzburg-Landau approximation in the interior to investigate the free boundary heat  flow. 

In \cite{DR}, \cite{DR2} Da Lio and Rivi\`ere introduced the notion of half-harmonic maps. These are also connected to harmonic maps with free boundary as has been shown by Millot and Sire in \cite{MS} as we now explain. Half-harmonic maps are critical points of a suitable energy which is the 
Dirichlet form induced by the fractional Laplacian $(-\Delta)^\frac12$, the Fourier multiplier of symbol
$|\xi|$.  More generally, given bounded open interval $\omega\subset \mathbb{R}$, the nonlocal (or fractional) $1/2$-energy in $\omega$ of a measurable function $u:\mathbb{R}\to\R^L$ is defined as 
$$\mathcal{E}(u,\omega):=\frac{\gamma}{2}\iint_{\omega\times\omega}\frac{|u(x)-u(y)|^2}{|x-y|^{2}}\,dxdy +\gamma \iint_{\omega\times(\R\setminus\omega)}\frac{|u(x)-u(y)|^2}{|x-y|^{2}}\,dxdy\,.$$
The normalization constant  $\gamma$ is chosen in such a way that  
$$\mathcal{E}(u,\omega)=\frac{1}{2}\int_{\R}|(-\Delta)^{\frac{1}{4}} u|^2\,dx, \quad \forall u\in C^\infty_c(\omega;\R^L)\,.$$
We denote by $\widehat H^\frac12(\omega;\R^L)$ 
the Hilbert space made of $L^2_{\rm loc}(\R)$-functions $u$ such that $\mathcal{E}(u,\omega)<\infty$, and we set 
$$\widehat H^\frac12(\omega;N):=\big\{u\in\widehat H^\frac12(\omega;\R^L): u(x)\in N\text{ a.e. on }\R \big\}\,.$$ 

A half-harmonic map satisfies the following Euler-Lagrange equation 
$$ (-\Delta)^{1/2}u\perp T_u\,N\quad \text{in $\mathscr{D}^\prime(\omega)$}\,.$$
In terms of scaling, this equation turns out to be critical. In the papers \cite{DR,DR2}, a regularity theory is developed for such maps. The link between such maps and harmonic maps with free boundary goes as follows. Let us consider, in a simplified setting, 
$M=\mathbb R^{n+1}_+,\,\,\,Z=\mathbb R^L,\,\,\,N=\mathbb S^{L-1}
$  and the following harmonic map system 
\begin{equation}\label{HMspheres}
\begin{cases}
-\Delta u = 0\,\,\,\text{in}\,\,\mathbb R^{n+1}_+,\\
u(\mathbb R^n)\subset \mathbb S^{L-1},\\
\frac{\partial u}{\partial \nu}\perp T_u \mathbb S^{L-1}.
  \end{cases}
\end{equation}
It can be shown (see e.g. \cite{MS}) that the trace of $u$ (denoted $v$) on $\mathbb R^n=\partial \mathbb R^{n+1}_+$, satisfies distributionally the half-harmonic map equation 
$(-\Delta)^{\frac12}\, v\perp T_v \, \mathbb S^{L-1}.$ 

On the other hand, consider a weak solution of the equation, $(-\Delta)^{\frac12}\, v\perp T_v \, \mathbb S^{L-1}$ on $\R^n$, its harmonic extension $u:\R^{n+1}_+ \to \R^L$ (obtained by convolving component-wise with the Poisson kernel) satisfies weakly \eqref{HMspheres}. 

The previous computation establishes a natural relation between half-harmonic maps and a class of harmonic maps with free boundary. Such a link can be also exhibited at the parabolic level, though some technicalities arise. The following flow was considered in \cite{HSSW} in this context
\begin{align}\label{half-harmonic heat flow}
\begin{cases}
&(\partial_t - \Delta)^{\frac 12} u \perp T_u N, \text{ on } \R^m \times (0, \infty),\\
&u(\cdot, t) = u_0, \text{ on } \R^m, \forall t \leq 0,
\end{cases}
\end{align}
for some integer $m \geq 1$ and a map $u_0 \in \dot{H}^{\frac 1 2} (\R^m; N)$. 
It is known (see \cite{banerjeeGaro}) that, suitably formulated, the flow \eqref{half-harmonic heat flow} enjoys a Struwe-type monotonicity formula.  This is due to the existence of a suitable (caloric) extension to the upper-half space (see \cite{NS} and \cite{ST}). As we will see below, though the operator $\big(\partial_t  - \Delta \big)^\frac12$ defined as a Fourier multiplier seems unnatural, its caloric extension to the upper half-space is naturally associated to extrinsic harmonic maps with free boundary. Indeed, extending the problem to $\RNp$ by adding an extra variable, one gets the following problem

\begin{align}\label{free bdry heat flow}
\left\{
\begin{array}{ll}
\partial_t U-\Delta_X U=0, &\text{in }\RNp \times(0,\infty),\\
\rule{0cm}{0.5cm} -\partial_y U \perp T_U N, &\text{on }\partial \RNp \times(0,\infty),
\end{array} \right.
\end{align}
along with the initial condition
\begin{equation}\label{eq:initial_cndn}
    U(x,y,0)=U_0(x,y), \quad   \text{on }\RNp \times\{0\},
\end{equation}

\noindent where $U_0$ satisfies
\begin{align}\label{eq:initial condition}
\begin{cases}
\Delta U_0 = 0, &\text{ on } \RNp,\\
U_0 = u_0, &\text{ on } \partial \RNp,
\end{cases}
\end{align}
and $X = (x, y)$ denotes the variable for points in $\R^m \times (0, \infty)$. In \eqref{free bdry heat flow} the map $U$ is considered to take values in $\R^L$ which is a Euclidean space that embeds the manifold $N$ isometrically. We notice that for the second condition in \eqref{free bdry heat flow} we need $U \big|_{\partial \RNp \times (0, \infty)} \in N$. \eqref{free bdry heat flow}-\eqref{eq:initial_cndn} is actually the free boundary heat flow from the space $\overline{\RNp}$. The existence and partial regularity for this flow was established in \cite{HSSW} using a suitable ``boundary Ginzburg-Landau approximation", which we introduce now. 

\subsection{The Set Up}\label{subsec-setup}

The aim of the present paper is to start the blow-up analysis of a class of flows of harmonic mappings with free boundary, by means of suitable parabolic defect measures in the spirit of the Lin-Wang theory. 
For simplicity we consider the domain manifold to be the upper half of a Euclidean space, say $\RNp := \R^m \times (0, \infty)$ for some integer $m \geq 1$. Let $N$ be a compact Riemannian manifold without boundary. By Nash's embedding theorem we may assume that $N$ is isometrically embedded in $\R^L$ for some positive integer $L$. For any $\delta > 0$ let $$B_{\delta} (N) := \{ p \in \R^L | \: \textnormal{dist}(p, N) < \delta \}.$$
It is well known that there exists $\delta_N > 0$ and a smooth map (called the nearest point projection map) $\Pi : B_{2 \delta_N} (N) \to N$ such that for all $p \in B_{2 \delta_N} (N)$, $\Pi(p)$ is the unique point in $N$ satisfying $\textnormal{dist}(p, N) = |p - \Pi(p)|$. In particular the map $p \mapsto \textnormal{dist}^2(p, N)$ is smooth on $B_{2 \delta_N} (N)$. Let us consider $\chi \in C^{\infty} (\R)$ to be an increasing function satisfying $$\chi(s) =
\begin{cases}
s, &\text{ for } s \leq \delta_N^2,\\
4 \delta_N^2, &\text{ for } s \geq 2 \delta_N^2.
\end{cases}$$
Now define $ F : \R^L \to \R$ as 
\begin{equation}\label{eq:defn of perturbation}
F(p) := \chi \left( \textnormal{dist}^2(p, N) \right).
\end{equation}

\noindent Consider the space 
\begin{equation}\label{fnspace1}
    \mathcal{V} := \left\{ V: \RNp \to \R^L | \: \int_{\RNp} |\nabla V|^2 dX \: < \infty \: \text{and} \: F(v) \in L^1(\R^m) \right\}
\end{equation}
where $v := \textup{Trace}(V)$.
For any $\eps>0$ define the following energy functional
\begin{equation}
\label{eq:energy}
E^{\eps}(V) := \frac 12\int_{\R^{m+1}_+} \abs{\nabla V}^2 \d X +
\frac{1}{\eps^2}\int_{\partial \RNp} F(v)\: dx, \ \textrm{for} \ V\in \mathcal{V}.
\end{equation}

Now for $\eps > 0$, let us consider the following boundary value problem as an approximation to \eqref{free bdry heat flow}:
\begin{equation}
\label{eq:approx_ext_+ general target}
\left\{
\begin{array}{ll}
 \partial_t U_{\eps} - \D_X U_{\eps} = 0,
&\text{ in } \mathbb{R}_+^{m+1}\times (0, \infty),\\
\rule{0cm}{0.5cm} \partial_y U_{\eps} = \frac{1}{\eps^2} (\nabla F) (U_{\eps}), &\text{ on } \partial \RNp \times (0, \infty),
\end{array} \right.
\end{equation}
together with the initial condition
\begin{equation}\label{eq:approx_ext_+ initial_cndn}
    U_\eps(\cdot ,0) = U_0, \quad\text{on }\RNp,
\end{equation}
where we require $U_{\eps}$ to take values in $\R^L$.  In the special case when $N = \S^{L-1}$, we will always consider $F(p) := \frac{1}{4} (1 - |p|^2)^2$.  The system \eqref{eq:approx_ext_+ general target} is invariant under  parabolic rescalings. More explicitly, for $(X_0, t_0) \in \partial \RNp \times (0, \infty)$ and $r > 0$, the rescaled  map $$V_{\eps}(X, t) := U_{\eps} (X_0 + rX, t_0 + r^2 t),$$   satisfies the system \eqref{eq:approx_ext_+ general target} in its domain,  with parameter  $\eps$  replaced by $\frac{\eps}{\sqrt r}$. However, this is not the case with the initial condition \eqref{eq:approx_ext_+ initial_cndn}. As the above mentioned rescalings  play a central role in  our analysis,  it is natural to look for scaling invariant conditions that can replace  \eqref{eq:approx_ext_+ initial_cndn} while retaining the essential estimates. The following constraints turn out to be suitable substitute for this:

\begin{equation}\label{eq:scale_inv_bds}
\left\{
\begin{array}{ll}
\rule{0cm}{0.5cm} \| U_{\eps} \|_{L^{\infty} (\RNp \times (0, \infty))} \leq C, &\forall \eps > 0,\\
\rule{0cm}{0.5cm} \E_{\eps} (U_{\eps}, (X, t), r) \leq C, & \forall \eps > 0, \, \forall (X, t) \in \mathcal K, \, \forall r \in (0, \frac{\sqrt t}{2}),
\end{array} \right.
\end{equation}
where $\mathcal{K}$ is any compact subset of $\partial \RNp \times (0, \infty)$ and the constant $C$ in the last inequality depends only on $m, N$ and $\mathcal K$. Here the energy $\E_{\eps}$ is as defined in Lemma \ref{lem-mono} below. The system \eqref{eq:scale_inv_bds} is invariant under scaling, in the sense that $V_{\eps}$ satisfies a similar system when $\eps$ in the second inequality of \eqref{eq:scale_inv_bds} is replaced by $\frac{\eps}{\sqrt r}$. We will discuss this in more details in Section \ref{blow up analysis section}. It can also be seen using the machinery provided in Section \ref{preliminary results section} that any solution of \eqref{eq:approx_ext_+ general target}-\eqref{eq:approx_ext_+ initial_cndn} also obeys the bounds given in \eqref{eq:scale_inv_bds}.

Existence of smooth solutions $U_{\eps}$ for the above system was established in \cite{HSSW}. Building on \cite{HSSW}, it is natural to introduce the following family of  energy densities in the space of Radon measures
\begin{equation}\label{eq:energy density radon measure}
e(U_{\eps}) \: dX dt := \frac12 |\nabla_X U_{\eps}|^2 \: dX dt + \frac{1}{ \eps^2} F(U_{\eps}) \: d\H^m \lfloor_{\partial \RNp} dt, \: \: \text{ on} \:\: \overline{\RNp} \times (0, \infty),
\end{equation}
and investigate the structure of their weak limits. The object of the present paper is precisely to derive several structure theorems on those limiting measures. Following the seminal work of Lin in \cite{Lin-annals} and Lin-Wang in \cite{LW-1,LW-2,LW-3}, {\sl defect measures} (both for the elliptic or the parabolic case) are natural objects to investigate for the appearance of singularities, which in our context are a consequence of concentration of energy. The properties of the limiting measures then encapsulates the occurrence and lack thereof singularities. In particular the theory of Lin-Wang provides the following two obstructions to the global existence and smoothness of the heat flow of harmonic maps:
\begin{itemize}
\item Harmonic spheres, i.e. non trivial harmonic maps from $\mathbb S^m$ into $N$,
\item Quasi-Harmonic spheres, i.e. non trivial maps of the form $\Phi(\frac{x}{\sqrt{-t}})$ solving \eqref{classical heat flow}. 
\end{itemize}

The system we consider, \eqref{eq:approx_ext_+ general target} has a different structure compared to the classical heat flow. The nonlinearity takes place on the boundary; this means that no concentration can hold inside the domain since the map in this case is caloric. Our purpose is indeed to have a better understanding of the formation of boundary singularities.

\subsection{Main Results}\label{subsec-main-result}

We now describe below our main results. We would like first to notice that it is important to identify the  critical dimension where  the problem is invariant under conformal transformations. For our flow, this critical dimension is $m=1$. In this case, several arguments can be simplified and  stronger results can be obtained. The problem becomes supercritical for $m \geq 2$. In contrast, the critical dimension for the classical heat flow is $m=2$. 

In harmonic mapping theory, singularity analysis revolves around bubbling phenomena. Although the situation for our free boundary flow is of a much more global nature, the core principle of geometric bubbles still prevails. As in Lin-Wang theory, the main obstructions to regularity are half-harmonic spheres, that is, global maps satisfying the boundary value problem: 
\begin{equation}\label{bubble}
\begin{cases}
-\Delta u =  0,\text{in}\,\,\, \mathbb R^{m+1}_+,\\
u(\partial \mathbb R^{m+1}_+)\subset N,\\
\frac{\partial u}{\partial \nu}\perp T_u N.
  \end{cases}
\end{equation}
Taking the trace of $u$ (denoted $u_0$) on $\partial \mathbb R^{m+1}_+$, it satisfies the equation 
$$
(-\Delta)^\frac12 u_0 \perp T_{u_0} N,
$$
where $(-\Delta)^\frac12$ is the square root of the Laplacian in $\mathbb R^m$ (the multiplier of symbol $|\xi|$). This is precisely the distributional formulation of half-harmonic maps from $\mathbb R^m$ into $N$. An important result in the original theory developed by Eells and Sampson \cite{ES} is that if the target is negatively curved, then the flow is globally defined, smooth and converges to a harmonic map in infinite time. This result has been generalized by Ding and Lin \cite{Ding-Lin}, allowing a slight amount of positive curvature in the target. For the harmonic map with free boundary, such result is available for the full system \eqref{FB heat flow}, under a mean convexity assumption of the boundary of the domain (see \cite{chen-lin}). 

Our setup is different since the map $u$ is valued in a flat space and only the boundary of the domain is constrained. This lack of ambient curvature allows more flexibility and our class of maps with free boundary are more related to half-harmonic maps as discussed above. In this context, we are not aware of a satisfying result in analogy with the Eells-Sampson theorem. More precisely, as we investigate in the current paper, non trivial half-harmonic loops into $N$ are the elements contributing to singularity formation for the flow. Recent results (see e.g. \cite{MMS}) suggest that such maps always exist for closed targets; additionally the Cartan-Hadamard theorem tells us that negative curvature for complete targets with trivial fundamental group would imply that the manifold is non-compact. This suggests that an Eells-Sampson type theorem in our setting amounts to require some structural properties of the defect measures, instead of a geometric condition. We would like to remark that even in the classical case, the singularity analysis is not completely settled since there is no existence result for quasi-harmonic spheres. This difficulty in the case of system \eqref{free bdry heat flow} amounts to the global nature of the problem and our contribution is a first step towards the understanding of singularity formation for such flows. We would like also to point out that in a recent preprint \cite{Changyou-defects}, the analysis of the defect measures for stationary half-harmonic maps has been worked out using quantitative stratification \`a la Naber-Valtorta \cite{Naber-V}. Curvature-related issues we just described also appear in this context. At a more technical level, curvature conditions arise when deriving Bochner type formulas. This is the original proof by Eells-Sampson but also hidden in all the other arguments like in Ding-Lin. Since our map $u$ is not fully constrained, but only its boundary, it is not clear at all how to get a useful (integrated) Bochner inequality. In view of this, we will take another viewpoint, reminiscent somehow of the blow-up analysis of Sacks and Uhlenbeck \cite{sacks-U}, using the flow as a perturbative template. Conditionally to some assumptions on the space-time singular set of weak solutions, we can extract a global half-harmonic map obtained as a non trivial limit of suitable parabolic solutions.

We now describe our main results in  detail. Building on the $\eps$-regularity theory of \cite{HSSW},  we fix a sufficiently small $\eps_0 > 0$ (characterized by Lemma \ref{eps_gradient_est} below) and introduce the singular set
\begin{equation}\label{sing_set_intro}
 \Sigma:=\Sigma_{\eps_0} := \bigcap_{R>0} \bigcup_{r \in (0, R)} \Big\{ Z_0 = (X_0, t_0) \in \partial \RNp \times (0,\infty)| \: \liminf_{\eps \to 0} \mathcal{E}_{\eps} (U_\eps, Z_0, r) \geq \eps_0^2 \Big\}
\end{equation}
where $\mathcal{E}_{\eps} (U_\eps, Z_0, r)$ denotes the  scale-invariant weighted   energy defined in Lemma \ref{lem-mono} below. The set $\Sigma$ contains the singular set of the weak limit of suitable solutions $U_\varepsilon$ of our systems, and coincides with the set of points where energy concentrates, which is the underlying mechanism for the formation of singularities.

Since the threshold parameter $\eps_0 > 0$ arises from the $\eps$-regularity theory in \cite{HSSW} and remains fixed throughout, we suppress the dependence on $\eps_0$ in the notation for $\Sigma$.

We first present our results in the critical dimension $m = 1$. Our main blow-up theorem is as follows:

\begin{thm}\label{main thm_intro} 
Let $m = 1$. Consider a sequence $(U_{\eps})_{\eps}$ of smooth solutions of \eqref{eq:approx_ext_+ general target} satisfying \eqref{eq:scale_inv_bds}. If the energy concentration set $\Sigma$, as defined in \eqref{sing_set_intro}, is non empty then there exists a sequence $( W_n )$ obtained through suitable parabolic rescalings of $U_\eps$ and a map $W$ such that
\begin{itemize}
\item $W_n \to W$ in $H^1_{loc} (\overline{\R^2_+} \times \R)$,
\item $W$ is independent of the time variable $t$, and
\item the spatial profile of $W$ (still denoted by $W$) is a non-trivial, smooth harmonic map satisfying
\begin{equation} \label{eq-bubble}
\begin{cases}
-\Delta W = 0, \quad &\text{in}\,\,\mathbb R^{2}_+,\\
W(\partial \R^{2}_+)\subset N,\\
\frac{\partial W}{\partial \nu}\perp T_W N, & \text{on } \partial \R^{2}_+. 
  \end{cases}
\end{equation} 
\end{itemize}    
\end{thm}

As a consequence of Theorem \ref{main thm_intro},  if the target manifold $N$ does not contain any non-constant half-harmonic $\S^1$, namely there does not exist any non-constant half-harmonic map from $\S^1 \to N$, then the convergence of $U_{\eps} \to U$ is ``strong"  and the map $U$ is smooth everywhere.  An analogous result  for the classical Ginzburg-Landau flow was proved in \cite{LW-1}.

\begin{thm}\label{energy identity_intro}
    Let $(U_{\eps})_{\eps}$ be as in Theorem \ref{main thm_intro}, $U$ be the weak limit of $U_{\eps}$ in $H^1_{loc} \left( \overline{\R^2_+} \times (0, \infty) \right)$ and $\Sigma$ be as in \eqref{sing_set_intro}. Suppose $Z_0 = (X_0, t_0) \in \partial \R^2_+ \times (0, \infty)$ is an isolated point in $\Sigma$. Then there exists a positive integer $b$, real numbers $a^1, \ldots, a^b \in (- \infty, 0]$ and for each $j \in \{1, \ldots, b\}$ a harmonic map $\omega^j : \R^2_{\geq a^j} \cup \{ \infty \} \to \R^L$ with $\omega^j \in \dot{H}^1 \left( \R^2_{ \geq a^j}; \R^L \right)$, $\omega^j (\partial \R^2_{\geq a^j}) \subseteq N$ and $\partial_y \omega^j \perp T_{\omega^j} N$ on $\partial \R^2_{\geq a^j}$ such that $$\lim_{t \nearrow t_0} E ( U(t); B_R^+ (X_0)) = E ( U(t_0); B_R^+ (X_0)) + \sum_{j = 1}^b E(\omega^j),$$ where $E$ denotes the usual Dirichlet energy and $R > 0$ is such that $P_{2R}(Z_0) \cap \Sigma = \{ Z_0 \}$. Furthermore, there exist sequences $(X^1_i), \ldots, (X^b_i)$ in $\overline{B_R^+ (X_0)}$ that converge to $X_0$ and sequences of positive real numbers $(\lambda^1_i), \ldots, (\lambda^b_i)$ converging to 0 and a sequence of time $(s_i)$ increasing  to $t_0$ such that $$ \lim_{i \to \infty} U(\cdot, s_i) - \sum_{j = 1}^b \left[ \omega^j \left( \frac{\cdot - X^j_i}{\lambda^j_i} \right) - \omega^j (\infty) \right] = U(\cdot, t_0), \text{ strongly in } H^1 (B_R^+ (X_0)).$$
\end{thm}

An energy identity for the free boundary heat flow on compact manifolds with boundary was proved in \cite{JLZ} by Jost, Liu and Zhu. Building on their  work and refining the analysis leads to Theorem \ref{energy identity_intro}. Concerning  reverse bubbling, we prove:   
\begin{thm}\label{reverse bubbling}
Suppose $Z_0 = (X_0, t_0) \in \partial \R^2_+ \times (0, \infty)$ is an isolated point in $\Sigma$, and let $R>0$ be such that $P_{2R}(Z_0) \cap \Sigma = \{ Z_0 \}$. Then there exist a non-negative integer $b$,   real numbers $a^1, \ldots, a^b \in (- \infty, 0]$ and for each $j \in \{1, \ldots, b\}$ a harmonic map $\omega^j : \R^2_{\geq a^j} \cup \{ \infty \} \to \R^L$ with $\omega^j \in \dot{H}^1 \left( \R^2_{ \geq a^j}; \R^L \right)$, $\omega^j  (\partial \R^2_{\geq a^j})\subseteq  N$ and $\partial_y \omega^j \perp T_{\omega^j} N$ on $ \partial \R^2_{\geq a^j}$ such that $$\lim_{t \searrow t_0} E ( U(t); B_R^+ (X_0)) = E ( U(t_0); B_R^+ (X_0)) + \sum_{j = 1}^b E(\omega^j).$$  Moreover,  there exist sequences $(X^1_i), \ldots, (X^b_i)$ in $\overline{B_R^+ (X_0)}$  converging to $X_0$,  sequences of positive real numbers $(\lambda^1_i), \ldots, (\lambda^b_i)$ converging to $0$,  and a sequence of times $(s_i)$  decreasing  to $t_0$ such that $$ \lim_{i \to \infty} U(\cdot, s_i) - \sum_{j = 1}^b \left[ \omega^j \left( \frac{\cdot - X^j_i}{\lambda^j_i} \right) - \omega^j (\infty) \right] = U(\cdot, t_0), \text{ strongly in } H^1 (B_R^+ (X_0)).$$
\end{thm}

To state our  further results, we introduce the limiting defect measures $\mu, \nu$ and $\eta$ as the weak* limits of sequence of the energy densities,
\begin{align}\label{eq:defect measure definition_intro}
&e(U_{\eps}) \: dX dt \xrightarrow{\eps \to 0} \frac12 |\nabla_X U|^2 dX dt + \nu =: \mu,\\
&| \partial_t U_{\eps}|^2 \: dX dt \xrightarrow{\eps \to 0} |\partial_t U|^2 \: dX dt + \eta,
\end{align}
where we recall the definition of the Radon measure $e(U_{\eps}) \: dX dt$ from \eqref{eq:energy density radon measure}. For an isolated point $Z_0 \in \Sigma$, forward bubbling necessarily occurs at $Z_0$ (Theorem \ref{energy identity_intro}), while reverse bubbling may additionally occur (Theorem \ref{reverse bubbling}). The total forward and reverse bubble energies, $A^-$ and $A^+$, are related  to the atomic mass of $\eta$ at $Z_0$ via the   identity $A^- = A^+ + \eta(\{Z_0\})$ (see  Corollary \ref{bubbling_energy_relation}).

We now describe the results in higher dimensions.

\begin{thm}\label{no harmonic S1_intro} 
Let $m\geq2$. Consider a sequence $(U_{\eps})_{\eps}$ of smooth solutions to \eqref{eq:approx_ext_+ general target} satisfying \eqref{eq:scale_inv_bds}.
If the  defect measure $\nu$, defined by \eqref{eq:defect measure definition_intro}, is non-zero, then there exists a sequence $( W_n )$ obtained through suitable parabolic rescalings of $U_\eps$ and a map $W$ such that 

\begin{itemize}
\item $ W_n \to W$ in $H^1_{loc} (\overline{\RNp} \times \R)$,

\item $W$ is independent of both the time variable $t$ and $m - 1$ linearly independent spatial directions,  

\item with respect to the remaining two spatial variables, $W$ is a bubble, that is, a non-trivial, smooth harmonic map satisfying the free boundary system \eqref{eq-bubble}.
\end{itemize}
\end{thm}

Along the way, we will actually provide finer information on the singular set of weak solutions and initiate the analysis of its stratification. We indeed prove:

\begin{thm}\label{fine reg parabolic blow up set}
 For $m \ge 2$, consider a sequence $(U_{\eps})_{\eps}$ of smooth functions as in Theorem \ref{no harmonic S1_intro} with associated energy concentration set $\Sigma$ defined in \eqref{sing_set_intro}.  Denoting by $\mathcal{P}_{\textnormal{dim}}(\Sigma)$ the parabolic Hausdorff dimension of $\Sigma$, at least one of the following holds:
\begin{enumerate}
\item $\mathcal{P}_{\textnormal{dim}}(\Sigma) \le m - 1$;
\item There exists a non-trivial, smooth, free boundary harmonic map $W : \R^2_+ \to \R^L$ with free boundary on $N$, satisfying \eqref{eq-bubble}, obtained as a limit of some parabolic rescalings and suitable conformal transformations of $(U_{\eps})$.
\end{enumerate}
\end{thm}

The present paper is the first installment of the blow-up analysis of such flows and many issues are still under investigation. The nonlocal/global nature of the problem is a challenging aspect and recent applications in various areas of geometry and topology motivate to get a better understanding of the deformation theory and the characteristics of harmonic mappings with free boundary.

\subsection{Notations and Conventions}
We now introduce some notations and conventions which will be used throughout what follows.
\begin{itemize}
\item $\RNp := \R^m \times (0, \infty)$. A general point in $\overline{\RNp}$ will be usually denoted by $X, \overline X, X_0$ etc. We will further decompose them as $X = (x, y), \overline X = (\overline x, \overline y), \ldots \in \R^m \times [0, \infty)$. 

\item $\delta_N > 0$ is such that the nearest point projection map $\Pi$ onto $N$ is defined in the $2 \delta_N$ neighborhood of $N$ in $\R^L$.

\item for any $a \in \R$, $\R^2_{\geq a}$ stands for the set $\{ (x, y) \in \R^2 | \: y \geq a \}$ and by saying that a function $w$ has domain $\R^2_{\geq a} \cup \{ \infty \}$ we mean that $w$ is defined on $\R^2_{\geq a}$ and $w(x, y)$ has a limit as $|(x, y)| \to \infty$.

\item unless specified otherwise, $U_{\eps}$ will denote a smooth solution of \eqref{eq:approx_ext_+ general target} satisfying \eqref{eq:scale_inv_bds}, and $U$ will denote an $H^1_{loc}$ weak limit of $U_{\eps}$ as $\eps$ tends to $0$, that solves
\begin{align}
\left\{
\begin{array}{ll}
\partial_t U-\Delta_X U=0, &\text{in }\RNp \times(0,\infty),\\
\rule{0cm}{0.5cm} -\partial_y U \perp T_U N, &\text{on }\partial \RNp \times(0,\infty).
\end{array} \right.
\end{align}

\item for any function $V : \RNp \to \R^L, \int_{\partial \RNp} V dx$ would denote $ \int_{\R^m} Trace(V) dx$.

\item for any $\eps > 0$, a function $V \in \mathcal{V}$ (see \eqref{fnspace1} for the definition of $\mathcal{V}$) and a bounded real valued function $\phi$ we define $$E^{\eps} (V) := \frac 12\int_{\R^{m+1}_+} \abs{\nabla_X V}^2 \: dX +
\frac{1}{\eps^2} \int_{\partial \RNp} F(V) \: dx, \text{  and}$$ $$E^{\eps}_{\phi} (V) := \frac 12\int_{\R^{m+1}_+} \abs{\nabla_X V}^2 \phi \: dX +
\frac{1}{\eps^2}\int_{\partial \RNp} F(V) \: \phi \: dx.$$

\item for a function $V : \RNp \to \R^L,\: E(V)$ will denote the usual Dirichlet energy of $V$, i.e., $$E(V) := \frac 1 2 \int_{\RNp} | \nabla_X V|^2 dX.$$
As before, for any bounded function $\phi$ we define $$E_{\phi} (V) := \frac 1 2 \int_{\RNp} |\nabla_X V|^2 \phi \: dX.$$

\item for any $\Omega \subseteq \RNp$, $\eps > 0$ and a suitable map $V : \RNp \to \R^L$, define $$ E^{\eps} (V; \Omega) := \frac 1 2 \int_{\Omega} |\nabla_X V|^2 \: dX +
\frac{1}{\eps^2} \int_{\partial \Omega \cap \partial \RNp} F(v) \: dx,$$
and similarly $$E(V; \Omega) := \frac 1 2 \int_{\Omega} |\nabla_X V|^2 dX.$$

\item for $X_0 \in \overline{\RNp}, \: t_0 \in \R$  and $ R > 0$, we  set 
\begin{align*}
T_R^+ (t_0) &:= \{ (X, t) \in \RNp \times \R | \:  t_0 - 4 R^2 <  t <  t_0 - R^2 \}, \\
\partial^0 T_R^+ (t_0) &:= \{ (x, 0, t) \in \partial \RNp \times \R | \: t_0 - 4 R^2 < t < t_0 - R^2 \},\\
B_R (X_0) &:= \{ X \in \R^{m+1} | \: |X - X_0| < R\}, B_R := B_R (0),\\
B_R^+ (X_0) &:= B_R (X_0) \cap \RNp, B_R^+ := B_R^+ (0), \\
\partial^0 B_R^+ (X_0) &:= \partial B_R^+ (X_0) \cap \partial \RNp, \partial^0 B_R^+ := \partial^0 B_R^+ (0),\\
\partial^+ B_R^+ (X_0) &:= \partial B_R^+ (X_0) \cap \RNp, \partial^+ B_R^+ := \partial^+ B_R^+ (0).
\end{align*}

\item for $Z_0 = (X_0, t_0) \in \overline{\RNp} \times \R$ and $R > 0$, we define
\begin{align*}
P_R (Z_0) &:= B_R (X_0) \times (t_0 - R^2, t_0 + R^2), P_R := P_R (0, 0)\\
P_R^+ (Z_0) &:= B_R^+ (X_0) \times (t_0 - R^2, t_0 + R^2), P_R^+ := P_R^+ (0, 0)\\
\partial^0 P_R^+ (Z_0) &:= \partial P_R^+ (Z_0) \cap \left( \partial \RNp \times \R \right), \partial^0 P_R^+ := \partial^0 P_R^+ (0, 0)\\
\partial^+ P_R^+ (Z_0) &:= \partial P_R^+ (Z_0) \cap \left( \RNp \times \R \right), \partial^+ P_R^+ := \partial^+ P_R^+ (0, 0).
\end{align*}
The dimension of the Euclidean space in which we consider a ball or a parabolic cylinder is of course not fixed and won't be specified every time. It will be clear from the context. 

\item for any non-negative integer $k$, $\H^k$ and $\p^k$ denotes respectively the $k$-dimension Hausdorff and parabolic Hausdorff measures. 

\item for $A \subseteq \R^{ m + 1} \times \R$ and $r > 0$, $$ P_r(A) := \bigcup_{Z \in A} P_r (Z), \text{ and } P_r^+ (A) := \bigcup_{Z \in A} P_r^+ (Z). $$

\item We will obey the usual convention with constants: $C$ will always denote a generic constant and the same symbol $C$ might have two different values when it appears in two different expressions. The parameters on which $C$ depends might sometimes be mentioned in parentheses beside $C$. $C$ will usually be dependent on either or both of $m$ and $N$ and this dependence will not be mentioned explicitly. 

\item  As our proofs are local in nature, the dependence of the constant $C$ on compact sets in \eqref{eq:scale_inv_bds} will be omitted henceforth. 
\end{itemize}

\section{Preliminary Results and Regularity Estimates}\label{preliminary results section}

The  first equation in the system under consideration 
is the standard heat equation, which possesses well-known, robust interior estimates. Consequently, the primary analytical challenge stems from the highly nonlinear free boundary condition. When passing to the limit in   \eqref{eq:approx_ext_+ general target} as $\varepsilon \to 0$, the analysis remains relatively straightforward in the interior  $\mathbb{R}^{m+1}_+ \times (0, \infty)$, compared to the boundary $\partial\mathbb{R}^{m+1}_+ \times (0, \infty)$. Accordingly, our primary focus is to examine these boundary effects, establish boundary estimates that are uniform with respect to $\varepsilon$, and identify and investigate where these estimates break down. In this section, we review relevant results in this direction established in \cite{HSSW}.

For $Z_0 = (X_0, t_0)\in \partial \RNp \times \R$, let 
$$\G_{Z_0} (X,t) := \frac{1}{\Gamma (\frac1 2) (4 \pi)^{\frac m 2} (t_0-t)^{ \frac{m + 1}{2}}} e^{ - \frac{| X - X_0|^2 }{ 4 (t_0 - t)}}, \quad \forall t < t_0, X \in \overline{\RNp}.$$

\begin{lem}[Monotonicity formula]\label{lem-mono}
Let $U_{\eps}$ be a smooth solution of \eqref{eq:approx_ext_+ general target} satisfying \eqref{eq:scale_inv_bds}. For any $Z_0 = (X_0, t_0) \in \partial \RNp \times (0, \infty)$ and $ 0 < R < \frac{\sqrt{t_0}}{2} $, the following two renormalized energies
\begin{align*}
\mathcal{D}_{\eps} (U_{\eps}, Z_0, R) := R^2 \Bigg( \frac 1 2 \int_{\RNp \times \{t_0 - R^2\}} |\nabla U_{\eps} |^2 \G_{Z_0} dX 
 + \frac{1}{ \eps^2} \int_{\partial \RNp \times \{t_0-R^2\} } F(U_{\eps}) \G_{Z_0} dx \Bigg)
\end{align*}
and
\begin{align*}
\E_{\eps} (U_{\eps}, Z_0, R) :=  \frac 1 2 \int_{T_R^+ (t_0)} |\nabla_X U_{\eps}|^2 \G_{Z_0} dXdt
+ \frac{1}{ \eps^2} \int_{\partial^0 T_R^+ (t_0)} F(U_{\eps}) \G_{Z_0} dx dt
\end{align*}
are nondecreasing with respect to $R$. Namely,
\begin{align*}
&\mathcal{D}_{\eps} (U_{\eps}, Z_0, r) \leq \mathcal{D}_{\eps} (U_{\eps}, Z_0, R),\\
&\E_{\eps} (U_\ve, Z_0, r)\leq \E_{\eps} (U_\ve, Z_0, R),
\end{align*}
for any $0 < r \leq R < \frac{\sqrt{t_0}}{2}$. More precisely, we have 
\begin{align*}
&\E_{\eps} (U_{\eps}, Z_0, R) - \E_{\eps} (U_{\eps}, Z_0, r)\\
=& \int_r^R \frac{1}{s} \left[ \int_{T_s^+(t_0)} \frac{ \left| \left( (X - X_0) \cdot \nabla_X \right) U_{\eps} + 2 (t - t_0) \partial_t U_{\eps} \right|^2}{2 (t_0 - t)} \, \G_{Z_0} \, dX dt + \int_{\partial^0 T_s^+ (t_0)} \frac{F(U_{\eps})}{\eps^2} \, \G_{Z_0} \, dx \, dt \right] ds.
\end{align*}
\end{lem}

\begin{lem}[Local energy inequality]\label{local_energy_ineq1}
Suppose $U_{\eps}$ is a smooth solution of \eqref{eq:approx_ext_+ general target}. For any $\phi \in C_c^1 (\R^{m+1})$, $\eps > 0$ and $0 < s < t < \infty$, it holds that
$$E^{\eps}_{\phi} (U_{\eps} (t)) - E^{\eps}_{\phi} (U_{\eps} (s)) = - \int_s^t \int_{\RNp} |\partial_t U_{\eps}|^2 \phi \: dX dt - \int_s^t \int_{\RNp} \langle \partial_t U_{\eps}, (\nabla \phi \cdot \nabla_X) U_{\eps} \rangle \: dX dt.$$
In particular, we have
$$E^{\eps}_{\phi^2} (U_{\eps} (t)) - E^{\eps}_{\phi^2} (U_{\eps} (s)) \leq - \frac 1 2 \int_s^t \int_{\RNp} |\partial_t U_{\eps}|^2 \phi^2 \: dX dt + 2 \int_s^t \int_{\RNp} |\nabla_X U_{\eps}|^2 |\nabla \phi|^2 \: dX dt,$$
and for any $Z_0 = (X_0, t_0) \in \overline{\RNp} \times (0, \infty)$ and $0 < R < \frac{\sqrt{t_0}}{2}$, $$ \int_{P_R (Z_0) \cap \RNp} \left| \partial_t U_{\eps} \right|^2 dX dt \leq \frac{C}{R^2} \left( \int_{P_{2R} (Z_0) \cap \RNp} \left| \nabla_X U_{\eps} \right|^2 dX dt + \int_{P_{2R} (Z_0) \cap \partial \RNp} \frac{F(U_{\eps})}{\eps^2} \: dx \: dt \right). $$
\end{lem}

\begin{lem}[Small energy regularity]\label{eps_gradient_est}
There exists $\eps_0 > 0$, depending only on $m$ and $N$, such that for any smooth function $U_{\eps}$ satisfying \eqref{eq:approx_ext_+ general target} and \eqref{eq:scale_inv_bds}, $Z_0=(X_0, t_0) \in \partial \RNp \times (0, \infty)$ and $0 < R < \frac{\sqrt{t_0}}{2}$, if
\begin{equation}\label{small_cond1}
\mathcal{E}_{\eps} (U_\eps, Z_0, R)<\eps_0^2,
\end{equation}
then $$\textnormal{dist}(U_{\eps}, N) \leq \delta_N,\: \text{ and } \: |\nabla_X U_{\eps}|^2 + |\partial_t U_{\eps}| \leq C (\delta_0 R)^{- 2}, \text{ on} \: P_{\delta_0 R}^+ (Z_0),$$
where $0 < \delta_0 < 1$ and $C$ are independent of $\eps, R$ and $Z_0$, and $\delta_N>0$ is as introduced in Subsection \ref{subsec-setup}. 
Moreover, under the same hypothesis as above it holds that $$\| U_{\eps} \|_{C^{1, \alpha} \left( \overline{P_{\delta_0 R}^+ (Z_0)} \right)} \leq C(\alpha, R),$$ 
for any $\alpha \in (0, 1)$, where the above constant $C (\alpha, R)$ is independent of $\eps$ and $Z_0$.
\end{lem}

We now briefly recall the essentials of the parabolic Hausdorff measure. Note that our definitions may differ from other standard versions by a constant factor depending only on the dimensions. This choice, however, does not alter the properties we are concerned with, such as sets of measure zero, sets of locally finite measure, or the parabolic Hausdorff dimension.

For any $d, \alpha \in (0, \infty)$, and any subset $S \subseteq \mathbb{R}^{m+2}$, we define$$\mathcal{P}^d_{\alpha} (S) := \inf \left\{ \sum_{i = 1}^{\infty} r_i^d \ \bigg|\  S \subseteq \bigcup_{i = 1}^{\infty} \overline{P_{r_i} (Z_i) } \text{ for some } Z_i \in \mathbb{R}^{m+2} \text{ and } r_i \leq \alpha \right\}.$$The quantity $\mathcal{P}^d_{\alpha} (S)$ takes values in the extended real numbers $[0, \infty]$. The $d$-dimensional parabolic Hausdorff measure of $S$ is then defined as$$\mathcal{P}^d(S) := \lim_{\alpha \searrow 0} \mathcal{P}^d_{\alpha} (S).$$Since $\mathcal{P}^d_{\alpha} (S)$ increases as $\alpha$ decreases to $0$, this limit always exists in $[0, \infty]$. For $d = 0$, we define $\mathcal{P}^0$ to be the counting measure. Next, for any $S \subseteq \mathbb{R}^{m + 2}$, we define its parabolic Hausdorff dimension as$$\mathcal{P}_{\textnormal{dim}} (S) := \inf \left\{ d \in (0, \infty) \mid \mathcal{P}^d (S) = 0 \right\}.$$

\begin{lem}\label{partial regularity lemma}
    For smooth solutions $U_{\eps}$ of \eqref{eq:approx_ext_+ general target} satisfying \eqref{eq:scale_inv_bds}, the energy concentration set $\Sigma$ as defined in \eqref{sing_set_intro}, has locally finite $(m + 1)$-dimensional parabolic Hausdorff measure, that is,  for any compact set $K \subseteq \partial \RNp \times (0, \infty)$, $$\p^{m + 1} (\Sigma \cap K) < \infty.$$
    Furthermore, for any $t \in (0, \infty)$, the time slice $\Sigma^t := \{ X \in \partial \RNp| \, (X, t) \in \Sigma \}$ has locally finite $(m - 1)$-dimensional Hausdorff measure.
\end{lem}

We conclude this section with a brief observation regarding the rescaling of $U_{\varepsilon}$. In what follows, we will frequently consider rescaled functions of the form$$(X, t) \mapsto U_{\varepsilon} (X_0 + r X, t_0 + r^2 t)$$for some positive parameter $r$ (which, in general, depends on $\varepsilon$). These rescaled functions satisfy systems similar to \eqref{eq:approx_ext_+ general target} and \eqref{eq:scale_inv_bds} with $\varepsilon$ replaced by $\varepsilon/\sqrt{r}$ (\eqref{eq:rescaledfunction} to be precise). Consequently, many of the previously stated and upcoming results in the following section remain valid for these rescaled functions. In particular, Lemmas \ref{lem-mono},  \ref{local_energy_ineq1} and \ref{eps_gradient_est} hold true in this context. We omit the detailed proofs for these cases, as they are nearly identical to those for $U_{\varepsilon}$.

\section{Blow Up Analysis}\label{blow up analysis section}

In this section, we establish several key properties of solutions to \eqref{eq:approx_ext_+ general target} that satisfy \eqref{eq:scale_inv_bds} and rescalings of such solutions, as outlined  at the end of the previous section. Some of these preliminary results are of independent interest. Throughout the following, it suffices to consider a sequence $(\varepsilon_n)_{n \in \Zp}$ decreasing to $0$ rather than a continuum of parameters  $\varepsilon$. Accordingly, we will occasionally write $\varepsilon_n$ in place of $\varepsilon$, and denote $U_{\varepsilon}$ correspondingly by $U_{\varepsilon_n}$ (or simply $U_n$). Furthermore, whenever we pass to a subsequence of $(\varepsilon_n)$, we shall still denote it by $(\varepsilon_n)$ for simplicity. 

Since our analysis relies on   specific rescalings of $U_{\varepsilon}$, we first summarize their basic structural   properties. For any $\varepsilon > 0$, $r > 0$ (which may depend on $\varepsilon$), and a fixed point $Z_0 = (X_0, t_0) \in \partial \mathbb{R}^{m+1}_+ \times (0, \infty)$, define the rescaled function 
\begin{equation}\label{eq:defn_rescaling}
V_{\varepsilon} (X, t) := U_{\varepsilon} (X_0 + r X, t_0 + r^2 t), \quad \forall (X, t) \in \RNp \times (-\frac{t_0}{r^2}, \infty).
\end{equation}
We refer to  $r$ as the rescaling factor of $V_{\eps}$. In what follows, we exclusively consider bounded families of rescaling factors $(r_{\eps})$.

The rescaled function $V_{\varepsilon}$ satisfies the following system
\begin{equation}\label{eq:rescaledfunction}
    \left\{
    \begin{array}{ll}
        \partial_t V_{\eps} - \Delta_X V_{\eps} = 0, & \text{ in } \RNp \times (- r^{-2} t_0 , \infty), \\
        \rule{0cm}{0.5cm} \partial_y V_{\eps} = \frac{r}{\eps^2} (\nabla F) (V_{\eps}), & \text{ on } \partial \RNp \times (- r^{-2} t_0, \infty),\\
        \rule{0cm}{0.5cm} \| V_{\eps} \|_{L^{\infty}} \leq \| U_{\eps} \|_{L^{\infty}}, & \\
        \rule{0cm}{0.5cm} \E_{\frac{\eps}{\sqrt r}} \left( V_{\eps}, (X, t), R \right) \leq  C(\mathcal K), &\forall (X, t) \in \mathcal K, \text { and } t-4R^2+r^{-2} t_0>0,
    \end{array}
    \right.
\end{equation}   
for any compact subset $\mathcal{K}$ of $\partial \RNp \times (-r^{-2} t_0, \infty)$. The last inequality follows from the identity $$\E_{\frac{\eps}{\sqrt r}} \left( V_{\eps}, (X, t), R \right) = \E_{\eps} \left( U_{\eps}, (X_0 + rX, t_0 + r^2 t), r R \right), $$
Lemma \ref{lem-mono} and \eqref{eq:scale_inv_bds}.
\begin{remark}\label{remark lem-mono}
    Using the estimates from \eqref{eq:rescaledfunction} and Lemma \ref{local_energy_ineq1} we get that $(V_{\eps})_{\eps}$ is bounded in $H^1_{loc}$, provided $(r_{\eps})$ stays bounded.
    Furthermore, a direct calculation yields
    \begin{equation}
    \E_{\eps} (U_{\eps}, Z_0, r) = \int_r^{2r} \frac{2}{s} \mathcal{D}_{\eps} (U_{\eps}, Z_0, s) \, ds, \quad \text{for any } Z_0 \in \partial \RNp \times (0, \infty), \text{ and } r \in (0, \frac{\sqrt{t_0}}{2}).
    \end{equation}
    From the above identity and  the monotonicity of $\mathcal{D}_{\eps} $ (Lemma \ref{lem-mono}) we deduce 
    \begin{equation}\label{eq:equivalence of energies}
    \mathcal{D}_{\eps} (U_{\eps}, Z_0, r) \leq \E_{\eps} (U_{\eps}, Z_0, r) \leq 2 \mathcal{D}_{\eps} (U_{\eps}, Z_0, 2r).
    \end{equation}
    Therefore, if $(r)_{\eps}$ is bounded then combining \eqref{eq:equivalence of energies} and \eqref{eq:rescaledfunction} one obtains that $(V_{\eps})_{\eps}$ is locally uniformly bounded in $L^{\infty}_t H^1_X$.
\end{remark}

For any rescaled $V_{\eps}$ as above, we now define the Radon measures $$e(V_{\eps})(\cdot, t) dX := \frac 1 2 |\nabla_X V_{\eps}|^2(\cdot, t) \: dX + \frac{r}{\eps^2} F(V_{\eps}) (\cdot, t) \: d\H^m \lfloor_{\partial \RNp}, \: \text{on } \: \overline{\RNp}, \: \forall t \in ( - \frac{t_0}{r^2}, \infty) ,$$ and 
\begin{equation}\label{eq:energy density}
e(V_{\eps}) \: dX dt := \frac12 |\nabla_X V_{\eps}|^2 \: dX dt + \frac{r}{ \eps^2} F(V_{\eps}) \: d\H^m \lfloor_{\partial \RNp} dt, \: \: \text{ on} \:\: \overline{\RNp} \times (- \frac{t_0}{r^2}, \infty),
\end{equation}
where $\mathscr{H}^m$ denotes the $m$-dimensional Hausdorff measure on  $  \R^{m + 1}$. To emphasize that $\frac{\varepsilon}{\sqrt{r}}$ replaces $\varepsilon$ in \eqref{eq:rescaledfunction}, we will occasionally write $e_{\varepsilon/\sqrt{r}} (V_{\varepsilon}) \, dX \, dt$ instead of $e(V_{\varepsilon}) \, dX \, dt$.

In the following we will assume that for some non-negative Radon measures $\nu$ and $\eta$ on $\overline{\RNp} \times (0, \infty)$ we have
\begin{align}\label{eq:Radonmeasureintro}
&e(U_{\eps}) \: dX dt \xrightarrow{\eps \to 0} \frac12 |\nabla_X U|^2 dX dt + \nu =: \mu,\\
&| \partial_t U_{\eps}|^2 \: dX dt \xrightarrow{\eps \to 0} |\partial_t U|^2 \: dX dt + \eta,
\end{align}
where the convergence is the weak* convergence of Radon measures. This follows  by passing to a subsequence of $(U_{\eps})$ and applying  Fatou's lemma. We emphasize that given these convergences, the limit inferior in \eqref{sing_set_intro} is, in fact, a limit.

Although the results in this section are stated and proved in terms of $U_{\varepsilon}$, their proofs rely primarily on \eqref{eq:rescaledfunction} and other structural properties satisfied by any rescaling $V_{\varepsilon}$ with bounded rescaling factors. Consequently, similar results remain valid for such rescaled functions as well, possibly with some more conditions on the rescaling factors $(r)_{\eps}$. Such conditions, if any, will be mentioned after the proofs of the results.

\begin{prop}\label{energy_tail_est}
    For any $ \delta > 0$ and $R, T > 1$ there exists $M = M( \delta, R, T) > 0$ such that for any $\eps > 0$ and $\overline Z = (\overline X, \overline t) \in \partial^0 B_R^+ \times (\frac{1}{T}, T)$ we have  $$\E_{\eps} (U_{\eps}, \overline Z, r) \leq \frac{C}{r^{m + 1}} \int_{\overline{t} - 4 r^2}^{\overline{t} - r^2} \int_{\overline{B_{Mr}^+ (\overline X)}} e(U_{\eps}) dX dt+ \delta,$$ and 
    $$\mathcal{D}_{\eps} (U_{\eps}, \overline Z, r) \leq \frac{C}{r^{m - 1}} \int_{\overline{B_{Mr}^+ (\overline X)}} e(U_{\eps})(\cdot, \overline t - r^2) dX + \delta,\quad \text{for }0<r<\frac{\sqrt{\bar t}}{2},$$ where $C>0$ is a  dimensional constant.
    Moreover, for any $Z_0 = (X_0, t_0) \in \partial \RNp \times (0, \infty)$, for the rescaled function $$V_{\eps} (X, t) := U_{\eps} (X_0 + r_{\eps} X, t_0 + r_{\eps}^2 t),$$ with the rescaling factor satisfying $r_{\eps} < \textnormal{min} \left\{ \frac{1}{R}, \sqrt{\frac{t_0}{2 T}} \right\}$ and $r r_{\eps} < \frac{\sqrt{t_0}}{10}$, we have for every $\overline Z = (\overline X, \overline t) \in \partial^0 B_R^+ \times (- T, T)$ $$\E_{\frac{\eps}{\sqrt{r_{\eps}}}} (V_{\eps}, \overline Z, r) \leq \frac{C}{r^{m + 1}} \int_{\overline t - 4 r^2}^{\overline t - r^2} \int_{\overline{B_{Mr}^+ (\overline X)}} e(V_{\eps}) dX dt + \delta, $$
for some constant $M = M(\delta, Z_0) > 0$.
\end{prop}

\begin{proof}
Using the notations as in the statement of the proposition, we compute for any $M > 0$
\begin{align*}
    \int_{\overline{t} - 4 r^2}^{\overline{t} - r^2} \int_{\overline{\mathbb{R}}^{m+1}_+ \setminus B_{Mr}^+ (\overline X)} \mathcal{G}_{\overline{Z}} \, e(U_{\varepsilon}) dX dt&\leq \frac{C e^{- \frac{M^2}{32}}}{r^{m + 1}} \int_{\overline{t} - 4 r^2}^{\overline{t} - r^2} \int_{\overline{\mathbb{R}}^{m+1}_+ \setminus B_{Mr}^+ (\overline X)} e^{- \frac{|X - \overline{X}|^2}{4 (2 \sqrt{2} r)^2}} \, e(U_{\varepsilon}) dX dt \\
    &\leq C e^{- \frac{M^2}{32}} \mathcal{E}_{\varepsilon} \left( U_{\varepsilon}, (\overline{X}, \overline{t} + 28 r^2),\, 2 \sqrt{2} r \right) \\
    &\leq C(R, T) e^{- \frac{M^2}{32}},
\end{align*}
where we have used \eqref{eq:scale_inv_bds} for the last inequality. Therefore the first desired  estimate for the energy $\E_{\eps}$ follows by choosing $M$ sufficiently large  so that $C(R, T) e^{- \frac{M^2}{32}} < \delta$. A similar arguments yields  the corresponding  estimate for $\mathcal D_{\eps}$ when combined with the inequalities   in Remark \ref{remark lem-mono}. Finally, the   last assertion of the proposition follows from this  estimate and the rescaling identity  $$\E_{\frac{\eps}{\sqrt{r_{\eps}}}} (V_{\eps}, (X, t), s) = \E_{\eps} (U_{\eps}, (X_0 + r_{\eps} X, t_0 + r_{\eps}^2 t), r_{\eps} s).$$
\end{proof}

\begin{prop}\label{time density of mu} Let $\eps_n \to 0$ be such that \eqref{eq:Radonmeasureintro} holds. Then there exists a subsequence $(\eps_{n'})$ of $(\eps_n)$ and a family of non-negative Radon measures $\{ \mu^t \}_{t > 0}$ on $\overline{\RNp}$ such that $$ e(U_{\eps_{n'}})(\cdot, t) dX \xrightarrow{n' \to \infty} \mu^t, \: \: \forall t > 0,$$ where the convergence is the weak* convergence of Radon measures.
\end{prop}

\begin{proof}
    It  suffices to prove the proposition for $t \in I :=(T_0, T_1)$  with  $0 < T_0 < T_1 < \infty$ arbitrary. Let $Q$ be any countable dense subset of $I$ with the property $$ \eta \left( \overline{\RNp} \times \{ t\} \right) = 0, \quad \forall t \in Q.$$
    Then by the  weak* compactness of Radon measures and a diagonal argument, we can extract  a subsequence of $(n)$ (still denoted by the same notation) and find a family of non-negative Radon measures $\{ \mu^t \}_{t \in Q}$ such that $$e(U_{\eps_n})(\cdot, t) dX \xrightarrow{n \to \infty} \mu^t$$ as Radon measures, for each $t \in Q$. 
     
     For any non-negative $\phi \in C_c^2(\R^{m+1} )$, differentiating the first identity  in Lemma \ref{local_energy_ineq1} (with $\phi$ replaced by $\phi^2$) with respect to $t$,  and  together with the Young's inequality  we obtain 
\begin{align*}
	&\frac{d}{dt} \Bigg\{ \int_{\RNp} \frac{|\nabla_X U_{\eps_n} |^2 (\cdot, t)}{2} \phi^2 dX + \int_{\partial \RNp} \frac{F(U_{\eps_n}) (\cdot, t)}{\eps_n^2} \phi^2 dx \Bigg\} \leq C \int_{\RNp} |\nabla U_{\eps_n}|^2 (\cdot, t) |\nabla \phi|^2 dX. 
\end{align*}  Applying  the local uniform $L^{\infty}_t H^1_X$ bound on $(U_\eps)$ (see Remark \ref{remark lem-mono})  to the right hand side  yields   $$ \frac{d}{dt} \left( E^{\eps_n}_{\phi^2} (U_{\eps_n}(\cdot, t)) \right) \leq C(\phi).$$
Therefore, the function  $t \mapsto E^{\eps_n}_{\phi^2} (U_{\eps_n}(\cdot, t)) - C(\phi) t$ is decreasing on $I$. 

Consider a countable subset $\{ \phi_i \}_{i \in \Zp}$ of $C_c^2(\R^{m+1}; \R_{\geq 0})$ such that $\{ \phi_i^2 \}_{i \in \Zp}$ is dense (with respect to $L^{\infty}$ norm) in $C_c^2( \R^{m+1}; \R_{\geq 0})$. For each $i \in \Zp$, since $t \mapsto E^{\eps_n}_{\phi_i^2} (U_{\eps_n}(\cdot, t)) - C(\phi_i) t$ is monotone on $I$,  passing to the limit as $n \to \infty$, we see that $s \mapsto \int_{\overline{\RNp}} \phi_i^2 d \mu^s - C(\phi_i) s$ is also monotone for $s \in Q$. Hence we can extend $s \mapsto \int_{\overline{\RNp}} \phi_i^2 d \mu^s - C(\phi_i) s$ to a monotone function on $I$ which will have at most countably many points of discontinuities. Therefore, there exists   a cocountable subset $\mathcal C$ of $I$ such that for each $i \in \Zp$,  the function $s \mapsto \int_{\overline{\RNp}} \phi_i^2 d \mu^s$ can be extended continuously to the whole of $\mathcal C$.
By somewhat abusing the notation, for any $t \in \mathcal C$, let us keep on denoting the value of this extended function at $t$ by $\int_{\overline{\RNp}} \phi_i^2 d \mu^t$ (even though $\mu^t$ has not been defined yet for $t \in \mathcal C \backslash Q$). Deleting at most countably many points from $\mathcal C$ if necessary, we may ensure that $$\eta \left( \overline{\RNp} \times \{ t \} \right) = 0, \quad \forall t \in \mathcal C.$$

For any $t_0 \in \mathcal C \backslash Q$, we know that there exists a subsequence $\left( e(U_{\eps_{n_j}})(\cdot, t_0) dX \right)_j$ of the sequence $\left( e(U_{\eps_n})(\cdot, t_0) dX \right)_n$ that converges to some non-negative Radon measure $\tilde{\mu}^{t_0}$ on $\overline{\RNp}$. Then for any $i \in \Zp$ and $t \in Q$, using the choices of the sets $Q$ and $\mathcal C$ we compute
\begin{align*}
	&\int_{\overline{\RNp}} \phi_i^2 d \tilde{\mu}^{t_0} = \lim_{j \to \infty} \int_{\overline{\RNp}} \phi_i^2 \: e(U_{\eps_{n_j}})(\cdot, t_0) dX = \lim_{j \to \infty} E^{\eps_{n_j}}_{\phi_i^2} (U_{\eps_{n_j}} (t_0))\\
=& \lim_{j \to \infty} [E^{\eps_{n_j}}_{\phi_i^2} (U_{\eps_{n_j}} (t_0)) - E^{\eps_{n_j}}_{\phi_i^2} (U_{\eps_{n_j}} (t))] + \lim_{j \to \infty} E^{\eps_{n_j}}_{\phi_i^2} (U_{\eps_{n_j}} (t))\\
=& \lim_{j \to \infty} \int_{t_0}^t \int_{\RNp} |\partial_t U_{\eps_{n_j}}|^2 \phi_i^2 dX dt + 2 \lim_{j \to \infty} \int_{t_0}^t \int_{\RNp} \langle \phi_i \partial_t U_{\eps_{n_j}}, (\nabla \phi_i \cdot \nabla_X) U_{\eps_{n_j}} \rangle dX dt + \int_{\overline{\RNp}} \phi_i^2 d\mu^t\\
=& \int_{\overline{\RNp} \times [t_0, t]} \phi_i^2 d\eta + \lim_{j \to \infty} \int_{t_0}^t \int_{\RNp} \left( |\partial_t U|^2 \phi_i^2 + 2 \langle \phi_i \partial_t U_{\eps_{n_j}}, (\nabla \phi_i \cdot \nabla_X) U_{\eps_{n_j}} \rangle \right) dX dt + \int_{\overline{\RNp}} \phi_i^2 d\mu^t.
\end{align*}
As $t\to t_0$, the space-time integral on the right hand side vanish. Hence, sending $t \to t_0$ we see that $\tilde{\mu}^{t_0}$ does not depend on the choice of the subsequence $(n_j)$. Thus,  denoting  $\tilde{\mu}^{t_0}$ by $\mu^{t_0}$  we  conclude
$$e(U_{\eps_n})(\cdot, t_0) dx \to \mu^{t_0}, \quad  \forall t_0 \in \mathcal C.$$
Since $I \backslash \mathcal C$ is countable,  we can again use a diagonal process to pass to a further subsequence $(n')$ such that $$e(U_{\eps_{n'}})(\cdot, t) dx \xrightarrow{n' \to \infty} \mu^{t}, \quad \forall t \in I.$$
This proves the proposition.
\end{proof}

\begin{prop}\label{convergence of approximation outside singular set general target}
The followings hold:
	\begin{itemize}
		\item[(a)] as $\eps \to 0, \textnormal{dist}(U_{\eps}, N) \to 0$ in $L^{\infty}_{loc} \left( \partial \RNp \times (0, \infty) \setminus \Sigma \right)$,

\item[(b)] as $\eps \to 0, \frac{1}{\eps^2} F(U_{\eps}) \to 0$ in $L^{\infty}_{loc} \left( \partial \RNp \times (0, \infty) \setminus \Sigma \right)$.
\end{itemize}
\end{prop}

\begin{proof}
It suffices to prove the above results locally within a small neighborhood of any point  $\tilde{Z} \in \left( \partial \mathbb{R}^{m+1}_+ \times (0, \infty) \right) \setminus \Sigma$. 
By Lemma \ref{eps_gradient_est}  there exists  a small ball $B$ centered at  $\tilde{Z}$ such that, on $\tilde{B} := B \cap \left( \partial \mathbb{R}^{m+1}_+ \times (0, \infty) \right)$, we have:
\begin{itemize}
	\item $d(U_{\eps}):=\textnormal{dist}(U_\eps, N)  \leq \delta_N$, and in particular  $\chi \left( d^2(U_{\eps}) \right) = d^2(U_{\eps})$,   

\item $|\partial_t U_{\eps}| + |\nabla_X U_{\eps}| \leq C$, for some $C$ independent of $\eps$,

\item $\chi' \left( d^2(U_{\eps}) \right) = 1$,

\item $\left| (\nabla d)( U_{\eps}) \right| = 1$, whenever it exists.
\end{itemize}

It then follows from 
	\begin{align*}
		 \left| \frac{1}{\eps^2} \chi' \left( d^2(U_{\eps}) \right) (\nabla d^2)( U_{\eps}) \right| = \left| \frac{1}{\eps^2} (\nabla F)(U_{\eps}) \right| =\left| \frac{\partial U_{\eps}}{\partial y}\right| \leq C,
	\end{align*} that \begin{align} d(U_{\eps}) \leq  C \eps^2,\end{align} and  consequently,   $$\frac{1}{\eps^2} F(U_{\eps})  =\frac{1}{\eps^2} d^2(U_{\eps})\leq C\eps^2.$$
\end{proof}

\begin{remark*} Note that the analogue of Proposition \ref{convergence of approximation outside singular set general target} need not hold in general for a sequence of rescaled functions $(V_{\varepsilon})$, as defined in \eqref{eq:defn_rescaling}. It remains valid, however, provided that $\varepsilon / \sqrt{r} \to 0$ as $\varepsilon \to 0$, where $r=r(\eps) > 0$ denotes the rescaling factor associated with $V_{\varepsilon}$. 
\end{remark*}

\begin{prop}\label{convergence of approximation}
	$\frac{1}{\eps^2} F(U_{\eps}) \to 0$ in $L^1_{loc} (\partial \RNp \times (0, \infty))$ as $\eps \to 0$.
\end{prop}

 \begin{proof} Let $0 < T_1 < T_2 < \infty$ and $R > 0$, and define the compact  set $K := \partial^0 B_R^+ \times [T_1, T_2]$. Our goal is to show that$$\lim_{\varepsilon \to 0} \frac{1}{\varepsilon^2}\int_K  F(U_{\varepsilon})\, dx \, dt = 0.$$ To this end, let $\delta \in (0, \delta_N)$ be small. It suffices to show  that for all sufficiently small $\varepsilon > 0$,$$ \frac{1}{\varepsilon^2} \int_K F(U_{\varepsilon}) \, dx \, dt < \delta.$$
Let us set $d(\cdot) := \textnormal{dist}(\cdot, N).$
For each $\varepsilon > 0$, we decompose $K$ into a bad part and a good part by defining$$K_b^{\varepsilon} := \left\{ Z \in K \;\middle|\; d\left( U_{\varepsilon} (Z)\right) \geq \delta \right\} \quad \text{and} \quad K_g^{\varepsilon} := K \setminus K_b^{\varepsilon}.$$We will then estimate the integral separately over $K_b^{\varepsilon}$ and $K_g^{\varepsilon}$.
 
 \medskip 

\noindent\underline{\textbf{Estimate on $K_b^{\eps}$ :}} To obtain the desired estimate on $K_b^{\eps}$, the core idea relies on the smallness of $K_b^{\eps}$ with respect to certain parabolic Hausdorff measures (although this connection will not be made explicit in the proof). To implement this strategy, we further partition $K_b^{\eps}$ into two subsets and estimate the integral over each separately. We proceed in several steps.

\medskip

\noindent\underline{Step 1:} There exists $\ve_1=\ve_1(K, \delta) > 0$ such that for every $\ve\in(0,\ve_1)$ and $Z\in K_b^\eps$ we have $$\E_{\eps} (U_{\eps}, Z ,r ) = \frac{\eps_0^2}{2} \quad\text{for some }0<r<\frac14\sqrt{T_1}. $$  

In order to prove Step 1, we first note that   $r \mapsto \E_{\eps} (U_{\eps}, Z, r)$ is continuous and  increasing. Moreover,  from the  smoothness of $U_{\eps}$ we have $$\lim_{r \searrow 0} \E_{\eps} (U_{\eps}, Z , r) = 0.$$ Thus, it suffices to show that $\E_{\eps} (U_{\eps}, Z , \sqrt{T_1}/4)>\frac{\ve_0^2}{2}$ for $Z\in K_b^\eps$ and $\ve\in(0,\ve_1)$. On the contrary,  $\E_{\eps} (U_{\eps}, Z , \sqrt{T_1}/4)\leq\frac{\ve_0^2}{2}$ for some $Z\in K_b^\ve$ would imply that $$ C(K) \geq |\partial_y U_{\eps} (Z)| = \frac{2}{\eps^2} d(U_{\eps} (Z)) |(\nabla d) (U_{\eps} (Z))| = \frac{2 d (U_{\eps} (Z))}{\eps^2} \geq \frac{2 \delta}{\eps^2},$$ thanks to Lemma \ref{eps_gradient_est}. This leads to a contradiction if $\ve_1>0$ is small enough, completing this step.

\medskip

For convenience, we define $$\F_{\eps} := \left\{ (Z, r ) \in K_b^{\eps} \times (0, \frac{\sqrt{T_1}}{4}) \Big| \: \E_{\eps} (U_{\eps}, Z , r ) = \frac{\eps_0^2}{2} \right\}.$$ Notice that if $(Z_{\eps}, r_{\eps}) \in \F_{\eps}$, then by the small energy regularity result, we have $$ C \geq r_{\eps} \left| \partial_y U_{\eps} (Z_{\eps}) \right| = \frac{2 r_{\eps}}{\eps^2} d(U_{\eps} (Z_{\eps})) \left| (\nabla d) (U_{\eps} (Z_{\eps})) \right| \geq \frac{2 \delta r_{\eps}}{\eps^2},$$ that is,   \begin{align}\label{est-step-1}r_{\eps} \leq C( \delta) \: \eps^2.\end{align}

\medskip 
\noindent\underline{Step 2 :} Fix $\overline{\delta} \in (0, 1)$ small enough so that $(m + 1) \overline{\delta} < 1$.   Then for any $\delta\in (0,\delta_N)$ small, there exist $\tau = \tau ( K, \delta) > 0$ and $\overline{\eps} = \overline{\eps} ( K, \delta) \in (0, \eps_1)$ such that for any $0 < \eps \leq \overline{\eps}$ and $(Z_{\eps}, r_{\eps}) \in \F_{\eps}$ at least one of the following holds: \begin{itemize} \item[(i)] 
\begin{align*}
\frac{1}{r_{\eps}^{m - 1}} \int_{P_{r_{\eps}^{1 - \overline{\delta}}}^+ (Z_{\eps})} |\partial_t U_{\eps}|^2 \: dX dt \geq \tau,\end{align*}
\item[(ii)] \begin{align*}\int_{\partial^0 P_{5 R r_{\eps}}^+ (Z_{\eps}) } \frac{1}{\eps^2} F(U_{\eps}) \: dx \: dt \leq \delta \int_{P_{R r_{\eps}}^+ (Z_{\eps}) } |\nabla_X U_{\eps}|^2 \: dX dt, \text{  for some } R \in \left[1, \: \frac{1}{\sqrt{r_{\eps}}} \right].
\end{align*}
\end{itemize}

We argue by contradiction. Suppose  Step 2 is false. Then, for each $k \in \Zp$ there exist $\eps_k \in (0, \frac{\eps_1}{k})$ and $(Z_k, r_k) = (X_k, t_k, r_k) \in \F_{\eps_k}$ such that
\begin{align*}
 \frac{1}{r_k^{m - 1}} \int_{P_{r_k^{1 - \overline{\delta}}}^+ (Z_k)} |\partial_t U_{\eps_k}|^2 \: dX dt < \frac 1 k, \end{align*}   and \begin{align*}
 \int_{\partial^0 P_{5 R r_k}^+ (Z_k) } \frac{1}{\eps_k^2} F(U_{\eps_k}) \: dx \: dt > \delta \int_{P_{R r_k}^+ (Z_k) } |\nabla_X U_{\eps_k}|^2 \: dX dt, \text{  for all } R \in \left[1, \: \frac{1}{\sqrt{r_k}} \right].
\end{align*}

Define $V_k (X, t) := U_{\eps_k} (X_k + r_k X, t_k + r_k^2 t)$, for all $(X, t) \in \overline{\RNp} \times (- \frac{t_k}{r_k^2}, \infty)$. Then $V_k$ satisfies
\begin{equation*}
\begin{cases}
\partial_t V_k - \Delta_X V_k = 0, \text{  in } \RNp \times (- \frac{t_k}{r_k^2}, \infty),\\
\partial_y V_k = \frac{r_k}{\eps_k^2} (\nabla F) (V_k), \text{  on } \partial \RNp \times (- \frac{t_k}{r_k^2}, \infty),\\
d( V_k (0, 0)) \geq \delta, \E_{\frac{\eps_k}{\sqrt{r_k}}} (V_k, (0, 0), 1) = \frac{\eps_0^2}{2},\\
\int_{P_{r_k^{ - \overline{\delta}}}^+} |\partial_t V_k|^2 \: dX dt < \frac 1 k,\\
\frac{r_k}{\eps_k^2}\int_{\partial^0 P_{5 R}^+ }  F(V_k) \: dx \: dt > \delta \int_{P_R^+ } |\nabla_X V_k|^2 \: dX dt, \text{  for all } R \in \left[1, \: \frac{1}{\sqrt{r_k}} \right].
\end{cases}
\end{equation*}
Recall that the sequence  $\left( r_k \, \eps_k^{- 2} \right)_k$ is bounded by \eqref{est-step-1}. Hence, from the regularity theory for the heat equation, up to a subsequence $V_k \to V$ in $C^2_{loc} (\overline{\RNp} \times \R)$, and $\frac{r_k}{\eps_k^2} \to \rho^2$, for some $\rho \in [0, \infty)$. Then using the properties of $V_k$ listed above, we see that $V$ is independent of $t$, and it  satisfies
\begin{equation}\label{system 1}
\begin{cases}
- \Delta_X V = 0, \text{  in } \RNp,\\
\partial_y V = \rho^2 (\nabla F) (V), \text{  on } \partial \RNp,\\
d(V(0, 0)) \geq \delta,\\
\rho^2 \int_{\partial^0 P_{5 R}^+ } F(V) \: dx \: dt \geq \delta \int_{P_R^+ } |\nabla_X V|^2 \: dX dt, \: \forall R \geq 1.
\end{cases}
\end{equation}
Using an estimate similar to the one given in Proposition \ref{energy_tail_est}, it follows that for all $R$ sufficiently large,
\begin{align}\label{nondegeneracy}
    \frac{\eps_0^2}{2} =& \lim_{k \to \infty} \E_{ \frac{\eps_k}{\sqrt{r_k}} } (V_k, (0, 0), 1)\\
    \leq& C \left[ \frac 1 2 \int_{- 4}^{-1} \int_{\RNp \cap \{ |X| \leq R \}} |\nabla_X V|^2 + \rho^2 \int_{-4}^{-1} \int_{\partial \RNp \cap \{ |X| \leq R \}} F(V) \right] + \frac{\eps_0^2}{4}.
\end{align}

We now distinguish the following two cases, namely $\rho=0$ and $\rho>0$. In the former case, the last inequality in \eqref{system 1} implies that $V \equiv const$,  a  contradiction to  \eqref{nondegeneracy}. In the latter case, first notice that for any $R > 1, \, t \in \R$ and sufficiently large $k \in \Zp$, we have using Lemma \ref{lem-mono}, Remark \ref{remark lem-mono} and \eqref{eq:scale_inv_bds},
\begin{align}
\frac{1}{R^{m - 1}} \: E^{\frac{\eps_k}{\sqrt{r_k}}} \left( V_k (t); B_R^+ \right) =& \frac{1}{(R r_k)^{m - 1}} \: E^{\eps_k} \left( U_{\eps_k} (t_k + r_k^2 t); B_{Rr_k}^+ (X_k) \right)\\
\leq& C \: \mathcal{D}_{\eps_k} \left( U_{\eps_k}, (X_k, t_k + r_k^2 t + (R r_k)^2), \frac{\sqrt{T_1}}{2} \right)\\
\leq& C.
\end{align}

\noindent Therefore passing to the limit as $k \to \infty$, we obtain $$\sup_{R > 1} \frac{1}{R^{m - 1}} E^{\frac{1}{\rho}} (V; B_R^+) \leq C.$$
Hence,  for any $1 < r < R < \infty$, using energy monotonicity formula for  $V$  (a proof of this monotonicity formula for the case where the target manifold is the unit sphere can be found in \cite[Lemma 5.2]{MS}; the proof for a general target manifold follows by an analogous argument.) we see that $$C \geq \frac{1}{R^{m - 1}} E^{\frac{1}{\rho}} (V; B_R^+) - \frac{1}{r^{m - 1}} E^{\frac{1}{\rho}} (V; B_r^+) \geq \int_r^R \frac 1 s \left[ \frac{1}{s^{m - 1}} \int_{\partial^0 B_s^+} \rho^2 F(V) \right] \: ds.$$ 
The above inequality clearly implies that along a subsequence of $R \to \infty$ we must have
\begin{align}
\label{eq-step-2}\lim_{R \to \infty} \frac{1}{R^{m - 1}} \int_{\partial^0 B_R^+}  F(V) = 0.
\end{align}
This, the last inequality in \eqref{system 1}, and again by the above  monotonicity  formula we deduce that for any $r>1$ and large enough $R$ along the above mentioned subsequence we have $$\frac{1}{r^{m - 1}} E^{\frac{1}{\rho}} (V; B_r^+) \leq \frac{1}{(R/ 5)^{m - 1}} E^{\frac{1}{\rho}} (V; B_{R/5}^+) \leq \frac{C(\delta)}{R^{m - 1}} \int_{\partial^0 B_R^+} \rho^2 F(V) \to 0, \text{ as } R \to \infty.$$ Hence,  $V \equiv const$, and going back to \eqref{eq-step-2},  $F(V) \equiv 0$, a contradiction to  $$  d(V(0, 0)) \geq \delta > 0.$$  
This proves Step 2.

\medskip 
 
With the same notations as above, for  each $\eps \in (0, \overline{\eps})$ we  define 
\begin{align*}
&K_{b,1}^{\eps} := \left\{ Z \in K_b^{\eps} \big| \: \exists r > 0 \text{ such that } (Z, r) \in \F_{\eps} \text{ and } \frac{1}{r^{m - 1}} \int_{P_{r^{1 - \overline{\delta}}}^+ (Z)} |\partial_t U_{\eps}|^2 \: dX dt \geq \tau \right\}, \text{ and }\\
&K_{b, 2}^{\eps} := \Bigg\{ Z \in K_b^{\eps} \big| \: \exists r > 0 \text{ such that } (Z, r) \in \F_{\eps} \text{ and }\\
&\hspace{0.5in} \int_{\partial^0 P_{5 R r}^+ (Z) } \frac{1}{\eps^2} F(U_{\eps}) \: dx \: dt \leq \delta \int_{P_{R r}^+ (Z) } |\nabla_X U_{\eps}|^2 \: dX dt, \text{ for some } R \in \left[1, \: \frac{1}{\sqrt r} \right] \Bigg\}.
\end{align*}
Step 2 then implies that $$K_b^{\eps} = K_{b, 1}^{\eps} \cup K_{b, 2}^{\eps}.$$ 

\medskip 

\noindent\underline{Step 3:}  We have    $$ \lim_{\ve\to0}\frac{1}{\eps^2}\int_{K_{b, 1}^{\eps}}  F(U_{\eps}) \: dx \: dt =0.$$  

For any $Z_0 = (X_0, t_0) = (x_0, 0, t_0) \in \partial \RNp \times [T_1, \infty)$ and $0 < r < \frac{\sqrt{T_1}}{4}$, using the monotonicity formula (Lemma \ref{lem-mono}) we have
\begin{align*}
\E_{\eps} (U_{\eps}, Z_0, \frac{\sqrt{T_1}}{4}) \geq& \E_{\eps} (U_{\eps}, Z_0, r) \geq C \int_{t_0 - 4 r^2}^{t_0 - r^2} \int_{\partial \RNp} \frac{1}{\eps^2} F(U_{\eps}) \frac{e^{- \frac{|x - x_0|^2}{4 (t_0 - t)} }}{(t_0 - t)^{\frac{m + 1}{2}}} \: dx \: dt\\
\geq& C \: r^{- m - 1} \int_{t_0 - 4 r^2}^{t_0 - r^2} \int_{\partial^0 B_r^+ (X_0)} \frac{1}{\eps^2} F(U_{\eps}) \: dx \: dt.
\end{align*} 
In particular, applying the above estimate for the point $(X_0, t_0 + 2r^2)$ with $(X_0, t_0) \in K$, we obtain 
\begin{equation}\label{eq:hausdorffbound}
\int_{\partial^0 P_r^+ (Z_0)} \frac{1}{\eps^2} F(U_{\eps}) \: dx \: dt \leq C r^{m + 1} \E_{\eps} (U_{\eps}, (X_0, t_0 + 2 r^2), \frac{\sqrt{T_1}}{4}) \leq C \: r^{m + 1},
\end{equation}
thanks to \eqref{eq:scale_inv_bds}.
We recall from Step 1 and Step 2 that whenever $\eps < \overline{\eps}$, we have for every $Z\in K_{b,1}^\ve$ there exists $r_Z\in (0,\frac14\sqrt{T_1})$ such that $(Z,r_Z)\in \mathcal F_\ve$ and 
\begin{equation}\label{eq:hausdorffbound2}
r_Z^{m - 1} \leq \frac{1}{\tau} \int_{P_{r_Z^{1 - \overline{\delta}}}^+ (Z)} |\partial_t U_{\eps}|^2 \: dX dt,\quad\text{and } r_Z\leq C(\delta)\ve^2.
\end{equation}
Combining the above estimates we see that for sufficiently small $\eps$, 
\begin{align*}
&\frac{1}{\eps^2} \int_{\partial^0 P_{5 r_Z^{1 - \overline{\delta}}}^+ (Z)} F(U_{\eps}) \: dx \: dt \leq C r_Z^{(1 - \overline{\delta}) (m + 1)} \leq C(\delta) \: \eps^2 r_Z^{m - 1}\\
&\leq C( \delta) \: \frac{\eps^2}{\tau} \int_{P_{r_Z^{1 - \overline{\delta}}}^+ (Z)} |\partial_t U_{\eps}|^2 \: dX dt = C(\delta) \: \eps^2 \int_{P_{r_Z^{1 - \overline{\delta}}}^+ (Z)} |\partial_t U_{\eps}|^2 \: dX dt.
\end{align*}
Since we have $$K_{b, 1}^{\eps} \subseteq \bigcup_{Z \in K_{b, 1}^{\eps}} P_{r_Z^{1 - \overline{\delta}}} (Z),$$
by a Vitali type covering lemma we get an at most countable, disjoint subcollection of the above cover, say $\left\{ P_{r_i^{1- \overline{\delta}}} (Z_i) \right\}_i$, such that $$ K_{b, 1}^{\eps} \subseteq \bigcup_i P_{5 r_i^{1- \overline{\delta}}} (Z_i).$$
Finally we estimate, for any $\eps \in (0, \overline{\eps})$ sufficiently small,
\begin{align*}
 \frac{1}{\eps^2}\int_{K_{b, 1}^{\eps}} F(U_{\eps}) \: dx \: dt \leq&  \frac{1}{\eps^2}\sum_i \int_{\partial^0 P_{5 r_i^{1- \overline{\delta}}}^+ (Z_i)}  F(U_{\eps}) \: dx \: dt\\
\leq& C( \delta) \: \eps^2 \sum_i \int_{P_{r_i^{1 - \overline{\delta}}}^+ (Z_i)} |\partial_t U_{\eps}|^2 \: dX dt\\
\leq& C(\delta) \: \eps^2,
\end{align*}
where we have used the local uniform boundedness of $(U_{\eps})$ in $H^1_{loc} \left( \overline{\RNp} \times (0, \infty) \right)$, provided by Remark \ref{remark lem-mono}.  Step 3 follows immediately.

\medskip 

\noindent\underline{Step 4:}   We have    $$ \limsup_{\ve\to0}\frac{1}{\eps^2}\int_{K_{b, 2}^{\eps}}  F(U_{\eps}) \: dx \: dt \leq C \delta.$$

For  $\eps \in (0, \overline{\eps})$ and  $Z \in K_{b, 2}^{\eps}$, there exist $r \in (0, \frac{1}{4} \sqrt{T_1})$ and $R \in \left[ 1, \frac{1}{\sqrt r} \right] $ such that $$\frac{1}{\eps^2}\int_{\partial^0 P_{5 R r}^+ (Z) }  F(U_{\eps}) \: dx \: dt \leq \delta \int_{P_{R r}^+ (Z) } |\nabla_X U_{\eps}|^2 \: dX dt.$$
Setting $r_Z := Rr$, we see that for every $Z \in K_{b, 2}^{\eps}$ there exists $r_Z > 0$ such that $$ \frac{1}{\eps^2} \int_{\partial^0 P_{5 r_Z}^+ (Z) } F(U_{\eps}) \: dx \: dt \leq \delta \int_{P_{r_Z}^+ (Z) } |\nabla_X U_{\eps}|^2 \: dX dt.$$
Applying a Vitali type covering lemma to the inclusion  $$K_{b, 2}^{\eps} \subseteq \bigcup_{Z \in K_{b,2}^{\eps}} P_{r_Z} (Z),$$
we get that there exists an at most countable subcollection $\left\{ P_{r_i} (Z_i) \right\}_i$ of the above cover such that $P_{r_i} (Z_i)$ are mutually disjoint and $$K_{b, 2}^{\eps} \subseteq \bigcup_i P_{5 r_i} (Z_i).$$ Then we have
\begin{align*}
\frac{1}{\eps^2}\int_{K_{b, 2}^{\eps}}  F(U_{\eps}) \: dx \: dt \leq& \sum_i  \frac{1}{\eps^2}\int_{\partial^0 P_{5 r_i}^+ (Z_i)} F(U_{\eps}) \: dx \: dt\\
\leq& \delta \sum_i \int_{P_{r_i}^+ (Z_i)} |\nabla_X U_{\eps}|^2 \: dX dt \leq C \delta,
\end{align*}
where the last inequality follows from the boundedness of $(U_{\eps})$ in $H^1_{loc} \left( \overline{\RNp} \times (0, \infty) \right)$. This concludes Step 4.

\medskip
From the above steps it follows that $$\limsup_{\eps \to 0} \frac{1}{\eps^2} \int_{K_b^{\eps}} F(U_{\eps}) \, dx \, dt \leq C \delta.$$

\medskip 
\noindent\underline{\textbf{Estimate on $K_g^{\eps}$:}} Recall that   $K = \partial^0 B_R^+ \times [T_1, T_2]$. Let $ \phi \in C_c^{\infty} (B_{2R})$ be a spatial   cut-off function satisfying $$0 \leq \phi \le 1, \quad \phi \equiv 1 \text{ on } B_R,   \quad \text{and }\, |\nabla \phi| \le \frac{2}{R}.$$ 
Set $\rho := \frac{\delta}{\delta_N} $. Testing the heat equation in \eqref{eq:approx_ext_+ general target} against the smooth vector field $\phi \chi' \left(\frac{d^2(U_{\varepsilon})}{\rho^2} \right) (\nabla d^2)(U_{\varepsilon})$, integrating over $\mathbb{R}^{m+1}_+ \times [T_1, T_2]$, and applying integration by parts in space, we obtain
\begin{align}\label{eq-3-3.9}
\int_{T_1}^{T_2}& \int_{\RNp} \phi \: \chi' \left(\frac{d^2(U_{\varepsilon})}{\rho^2} \right) \: \partial_t U_{\eps} \cdot (\nabla d^2) ( U_{\eps})\\
=& \int_{T_1}^{T_2} \int_{\RNp} \phi \: \chi' \left(\frac{d^2(U_{\varepsilon})}{\rho^2} \right) \: \Delta_X U_{\eps} \cdot (\nabla d^2)( U_{\eps})\\
=& - \int_{T_1}^{T_2} \int_{\RNp} \chi' \left(\frac{d^2(U_{\varepsilon})}{\rho^2} \right) (\nabla \phi \cdot \nabla_X) U_{\eps} \cdot (\nabla d^2)( U_{\eps})\\
&- \int_{T_1}^{T_2} \int_{\RNp} \rho^{- 2} \, \phi \: \chi'' \left(\frac{d^2(U_{\varepsilon})}{\rho^2} \right) \: \left| \nabla_X \left( d^2 (U_{\eps}) \right) \right|^2\\
&- \int_{T_1}^{T_2} \int_{\RNp}\phi \: \chi^{\prime} \left(\frac{d^2(U_{\varepsilon})}{\rho^2} \right) \sum_{i, j, k} \bigg( (\partial_i \partial_k d^2) (U_{\eps}) \bigg) (\partial_j U_{\eps}^i) (\partial_j U_{\eps}^k)\\
&- \int_{T_1}^{T_2} \int_{\partial \RNp} \phi \: \chi' \left(\frac{d^2(U_{\varepsilon})}{\rho^2} \right) \: \frac{1}{\eps^2} \chi' \left( d^2 (U_{\eps}) \right) \left| (\nabla d^2) ( U_{\eps}) \right|^2 \\=&: - I - II - III - IV.
\end{align}
In order to estimate these integrals, we recall the following properties of $\chi$ and $d$: 
\begin{itemize}\label{facts_1}
\item $\chi$ is increasing with  $|\chi'| + |\chi''| \leq C$ on $[0, \infty)$,
\item $\chi' \equiv 1$ on $[0, \delta_N^2]$, and $\chi' \equiv 0$ on $[4 \delta_N^2, \infty)$,
\item $\chi'' \equiv 0$ outside $[\delta_N^2, 4 \delta_N^2$],
\item $\left| \nabla d^2 \right| = 2 d $ in $B_{2 \delta_N} (N)$,
\end{itemize} These properties, and the definition of $K_g^\eps$ imply $$\int_{K_g^{\eps}} \frac{1}{\eps^2} d^2 (U_{\eps}) \leq \int_{T_1}^{T_2} \int_{\partial \RNp} \phi \: \chi' \left(\frac{d^2(U_{\varepsilon})}{\rho^2} \right) \: \frac{1}{\eps^2} \chi' \left( d^2 (U_{\eps}) \right) \left| (\nabla d^2)( U_{\eps}) \right|^2 = IV.$$  Considering  $II$, notice that $$\rho^{-2}\left|\chi'' \left(\frac{d^2(U_{\varepsilon})}{\rho^2} \right)\right| \: \left| \nabla_X \left( d^2 (U_{\eps}) \right) \right|^2\leq 4 \left|\chi'' \left(\frac{d^2(U_{\varepsilon})}{\rho^2} \right)\right| \: \frac{d^2(U_\eps)}{\rho^2}\left| \nabla_X   U_{\eps}  \right|^2\leq C\left| \nabla_X   U_{\eps}  \right|^2 , $$ and the integrand vanishes unless   $\delta< d(U_\eps)<2\delta$. We get $$ \liminf_{\eps\to0}|II|\leq C  \liminf_{\eps\to0}\int_{T_1}^{T_2}\int_{\{\delta<d(U_\eps)<2\delta\}}\phi|\nabla_X U_\eps|^2=:|II_\delta|.$$   

 In order to obtain a lower bound on $III$, we recall that $ d^2(p)=|p-\Pi(p)|^2$. For any point $p\in B_{\delta_N}(N)$ and any vector   $v\in\R^L $, expanding $d^2(p+tv)$ for $t$ small  yields
\begin{align*}
    \sum_{i, k} \partial_i \partial_k d^2 (p) v_i v_k = \frac{d^2}{dt^2}\big|_{t=0}d^2(p+tv)=2 |v - (v \cdot \nabla) \Pi (p)|^2 - 2 (p - \Pi(p)) \cdot \textnormal{Hess} \Pi(p) [v, v].
\end{align*} Consequently, \begin{align*}
\sum_{i, j, k} \bigg( (\partial_i \partial_k d^2) (U_{\eps}) \bigg) (\partial_j U_{\eps}^i) (\partial_j U_{\eps}^k) \geq - C d(U_{\eps})  |\nabla_X U_\eps|^2,\quad\text{in }K_g^\eps,
\end{align*}
whence it holds that $$-III \leq C \int_{T_1}^{T_2} \int_{B_{2R}^+} \chi' \left(\frac{d^2(U_{\varepsilon})}{\rho^2} \right)\, d(U_{\eps}) \, |\nabla_X U_{\eps}|^2 \leq C \delta.$$
Combining these estimates, the  properties of $\chi$ and $ d$, and   the uniform $H^1_{loc}$ bound of $U_{\eps}$, we deduce from  \eqref{eq-3-3.9} that 
 \begin{align}   \int_{K_g^{\eps}} \frac{1}{\eps^2} d^2 (U_{\eps})& \leq C \int_{T_1}^{T_2} \int_{\RNp}d(U_\eps)(\phi+|\nabla\phi|)  \chi'   \left(\frac{d^2(U_{\varepsilon})}{\rho^2} \right)  (|\partial_t U_\eps|^2+|\nabla_X U_\eps|^2)dXdt+|II|\\ &\leq C\delta+|II|.
\end{align}
Therefore, together with the estimates on $K_b^\ve$,  we obtain  $$\liminf_{\ve\to0}\frac{1}{\eps^2} \int_K F(U_{\eps}) \: dx \: dt \leq C \: \delta+|II_\delta|.$$ Finally, choosing $\delta\downarrow0$ suitably,  we obtain  $$\liminf_{\ve\to0}\frac{1}{\eps^2} \int_K F(U_{\eps}) \: dx \: dt=0,$$ thanks to  Lemma \ref{image_closeness}. 
\end{proof}

\begin{lem}\label{image_closeness}Let $R>0$ and $\phi\in C_c^\infty(B_R)$ be a  spatial cut-off function such that $0\leq\phi\leq1$. Then for $0<T_1<T_2$ we have \begin{align}  \liminf_{\delta\to0}\liminf_{\eps\to0}\int_{T_1}^{T_2}\int_{ \{ \delta< d(U_\eps)\leq 2\delta \}}\phi |\nabla_X U_\eps|^2 dXdt=0.  \end{align}
\end{lem}
\begin{proof} Since $(U_\eps)$ is uniformly bounded in $H^1_{loc} $, we have that $$\int_{T_1}^{T_2}\int_{B_R^+}|\nabla_X U_\eps|^2dXdt\leq M,$$ for some $M\geq 1$.  For any integer $n \gg 1$, consider the sets $$A_{\eps,k}:=\left\{ (X,t)\in B_R^+\times[T_1,T_2]\,|\, 2^{-k-1}<d(U_\eps(X,t)) \leq 2^{-k}  \right\},\quad 2^n\leq k\leq 2^{2n}-1 .$$
It then follows that for each $\eps>0$, there exists $k \in \{2^n, \dots, 2^{2n}-1\} $ such that  $$  \int_{A_{\eps,k}}|\nabla_X U_\eps|^2dXdt<\frac 1n,$$ provided $n$ is sufficiently large (depending only on $M$). On the contrary, for some $\eps > 0$, as the sets $\{A_{\eps,k}\}_{2^n\leq k\leq 2^{2n}-1}$ are  mutually disjoint, $$\frac{2^n}{n}\leq \sum_{k=2^n}^{2^{2n}-1}  \int_{A_{\eps,k}}|\nabla_X U_\eps|^2dXdt\leq M,$$ a contradiction for $n$ large. In particular, for such large $n$, there exists $k_n\in \{2^n, \dots, 2^{2n}-1\}$ such that $$\liminf_{\eps\to0}  \int_{A_{\eps,k_n}}|\nabla_X U_\eps|^2dXdt\leq\frac 1n.$$
The lemma follows immediately. 
\end{proof}

\begin{remark*}
Analogous to Proposition \ref{convergence of approximation outside singular set general target}, Proposition \ref{convergence of approximation}   remains valid for a sequence of rescaled functions $(V_{\varepsilon})$ as defined in \eqref{eq:defn_rescaling}, provided that $\varepsilon / \sqrt{r} \to 0$ as $\varepsilon \to 0$, where $r =r(\eps)> 0$ denotes the rescaling factor associated with $V_{\varepsilon}$.
\end{remark*}

We now  recall a few notions: $$reg(U) := \left\{ Z \in \overline{\RNp} \times (0, \infty) | \: U \in C^{\infty} \left( \overline{P_r^+ (Z)} \right), \text{ for some } r > 0 \right\}, \text{ and }$$ $$sing(U) := \left( \overline{\RNp} \times (0, \infty) \right) \setminus reg(U).$$

For any measure $\Gamma$ on $\overline{\RNp} \times (0, \infty)$, $$zero(\Gamma) := \left\{ Z \in \overline{\RNp} \times (0, \infty) | \: \Gamma \left( \overline{P_r^+ (Z)} \right) = 0, \text{ for some } r > 0 \right\}, \text{ and } $$ $$ \textnormal{spt}(\Gamma) := \left( \overline{\RNp} \times (0, \infty) \right) \backslash zero(\Gamma).$$
Let us also recall the defect measures $\nu$ and $\eta$ from \eqref{eq:Radonmeasureintro}.

\begin{prop}\label{configuration of Sigma}
	$\Sigma = \textnormal{spt}(\nu) \cup sing(U)$ and $\textnormal{spt}(\eta) \subseteq \Sigma.$
\end{prop}

\begin{proof} It follows from the  weak convergence of $U_{\eps}$ to $U$, standard estimates for the heat equation and the uniform bounds in \eqref{eq:scale_inv_bds} that $U_{\eps} \to U$ is $C^2_{loc} \left( \RNp \times (0, \infty) \right)$. In particular, $$\textnormal{spt}(\nu) \cup sing(U) \cup \textnormal{spt}(\eta) \subseteq \partial \RNp \times (0, \infty). $$
From  Lemma \ref{eps_gradient_est} we have that for any $\alpha \in (0, 1)$, $$U_{\varepsilon} \to U \text{ in } C^{1, \alpha}_{\text{loc}} \left( \left( \overline{\mathbb{R}}^{m+1}_+ \times (0, \infty) \right) \setminus \Sigma \right).$$
The higher order regularity theory then implies that $U \in C^{\infty} \left( \left( \overline{ \mathbb{R}}^{m+1}_+ \times (0, \infty) \right) \setminus \Sigma \right)$. Furthermore, this strong convergence yields $$\nu \left( \left( \partial \mathbb{R}^{m+1}_+ \times (0, \infty) \right) \setminus \Sigma \right) = 0 = \eta \left( \left( \partial \mathbb{R}^{m+1}_+ \times (0, \infty) \right) \setminus \Sigma \right).$$ 
This establishes the inclusion $$\textnormal{spt}(\nu) \cup \text{sing}(U) \cup \textnormal{spt}(\eta) \subseteq \Sigma.$$

\noindent Next, for the reverse inclusion, suppose $\tilde{Z} = (\tilde{X}, \tilde{t}) = (\tilde{x}, \tilde{y}, \tilde{t}) \in \left( \overline{\mathbb{R}}^{m+1}_+ \times (0, \infty) \right) \setminus \left( \textnormal{spt}(\nu) \cup \text{sing}(U) \right)$. If $\tilde{Z} \in \mathbb{R}^{m+1}_+ \times (0, \infty)$, then by the definition of $\Sigma$ we immediately have $\tilde{Z} \notin \Sigma$. Now, suppose $\tilde{Z} \in \partial \mathbb{R}^{m+1}_+ \times (0, \infty)$. By assumption, there exists a radius $r_0 \in (0, \sqrt{\tilde{t}}/4)$ such that $\nu \left( \overline{P^+_{r_0} (\tilde{Z})} \right) = 0$ and $U \in C^{\infty} \left( \overline{P^+_{r_0} (\tilde{Z})} \right)$. Proposition \ref{energy_tail_est} asserts the existence of $M > 1$ such that for any $r \in (0, r_0)$,

$$\E_{\eps} (U_{\eps}, \tilde Z, r) \leq \frac{C}{r^{m + 1}} \int_{\tilde t - 4 r^2}^{\tilde t - r^2} \int_{\overline{B_{Mr}^+ (\tilde X)}} e(U_{\eps}) dX dt + \frac{\eps_0^2}{2}.$$ 
Passing to the limit as $\varepsilon \to 0$, and using $\nu \left( \overline{P^+_{ r_0} (\tilde{Z})} \right) = 0$, it follows that for all $r$ sufficiently small
\begin{align} 
\liminf_{\eps \to 0}\, \E_{\eps} (U_{\eps}, \tilde Z, r) &\leq \frac{C}{r^{m + 1}} \int_{\tilde t - 4 r^2}^{\tilde t - r^2} \int_{B_{2 M r}^+ (\tilde X)} |\nabla_X U|^2 \: dX dt + \frac{\eps_0^2}{2} \\
& \leq C \| U\|^2_{C^1 \left( P_{r_0}^+ (\tilde Z) \right)} \: r^2+ \frac{\eps_0^2}{2}.
\end{align}
Thus,  for  sufficiently small choice of $r>0$, $$\liminf_{\eps \to 0} \E_{\eps} (U_{\eps}, \tilde Z, r) < \eps_0^2,$$ which implies $\tilde{Z} \notin \Sigma$. Hence, $\Sigma \subseteq \textnormal{spt}(\nu) \cup \text{sing}(U)$, completing the proof. 
\end{proof}

For any $Z \in \partial \RNp \times (0, \infty)$, let us define 
\begin{align}\label{defn_density}
&\Theta^{m+1} (\mu, Z) := \lim_{r \searrow 0} \int_{\overline{T_r^+ (Z)}} \G_Z \: d\mu,\\
&\Theta^{m+1} (U, Z) := \lim_{r \searrow 0} \frac{1}{r^{m+1}} \int_{P_r^+(Z)} |\nabla_X U|^2 \: dX dt,\\
&\Theta^{m+1} (\nu, Z) := \lim_{r \searrow 0} \int_{\overline{T_r^+ (Z)}} \G_Z \: d\nu,
\end{align}
whenever the limits exist.

\begin{prop}\label{density_results}
	With the above notations, we have the following
\begin{enumerate}
	\item for all $Z \in \partial \RNp \times (0, \infty)$ the limit defining $\Theta^{m+1} (\mu, Z)$ exists and we have $$\Sigma = \left\{ Z \in \partial \RNp \times (0, \infty) | \: \Theta^{m+1}(\mu, Z) \geq \eps_0^2 \right\},$$
	\item for $\p^{m+1}$-almost every $Z \in \partial \RNp \times (0, \infty)$ the limit defining $\Theta^{m+1}(U, Z)$ exists and equals $0$,
	\item for $\p^{m+1}$-almost every $Z \in \partial \RNp \times (0, \infty)$ we have $\Theta^{m+1} (\mu, Z) = \Theta^{m+1} (\nu, Z)$.
\end{enumerate}
\end{prop}

\begin{proof} The first assertion of $(1)$ follows directly from the monotonicity formula (Lemma \ref{lem-mono}), which implies that the map $r \mapsto \int_{\overline{T_r^+ (Z)}} \G_Z \: d\mu$  is non-decreasing. The second assertion is an immediate consequence of the definitions of $\Sigma$ and $\mu$.

\medskip 

The proof of  $(2)$ follows from standard covering arguments in geometric measure theory and is therefore omitted here; we refer the reader to \cite[Section 2.10.19]{F} for a detailed proof.

\medskip 
The proof of $(3)$ would follow immediately from $(2)$ if we show for $\tilde Z = (\tilde X, \tilde t) \in \partial \RNp \times (0, \infty)$ that \begin{align}\label{p-3.5.3}\lim_{r \searrow 0} \int_{T_{r}^+ (\tilde Z)} | \nabla_X U|^2 \, \mathcal G_{\tilde Z} \, dX dt = 0,\end{align} whenever $\Theta^{m+1}(U,\tilde Z)=0$. Let us fix $\delta > 0$. In order to show \eqref{p-3.5.3}, we use Proposition \ref{energy_tail_est} to obtain $M \gg 1$ satisfying
\begin{align*}
\frac{1}{2}& \int_{T_r^+ (\tilde Z)} |\nabla_X U|^2 \, \G_{\tilde Z} dX dt\\
\leq& \frac{1}{2} \int_{\tilde t - 4 r^2}^{\tilde t - r^2} \int_{ B_{Mr}^+ (\tilde X)} |\nabla_X U|^2 \, \G_{\tilde Z} \, dX dt + \liminf_{\eps \to 0} \int_{\tilde t - 4 r^2}^{\tilde t - r^2} \int_{ \RNp \setminus B_{Mr}^+ (\tilde X)} \G_{\tilde Z} \, e(U_{\eps}) dX dt\\
\leq& \frac{C}{r^{m + 1}} \int_{P_{M r}^+ (\tilde Z)} |\nabla_X U|^2 \, dX dt + \delta = o_r (1) + \delta,
\end{align*}
where we have used the assumption that $\Theta^{m + 1} (U, \tilde Z) = 0$. This completes the proof.
\end{proof}

\begin{prop}\label{blowupcndn}
$\nu > 0$ if and only if $\p^{m + 1}(\Sigma)>0$.
\end{prop}

\begin{proof}
If $\p^{m+1}(\Sigma) > 0$, assertions $  (1)$ and $(3)$ of Proposition \ref{density_results} yield a point $\tilde{Z} \in \Sigma$  such that $\eps_0^2 \leq \Theta^{m+1} (\mu, \tilde Z) = \Theta^{m+1} (\nu, \tilde Z)$.  This lower bound for the density immediately implies that $\nu > 0$.

Conversely,  suppose that $\p^{m+1} \left( \Sigma \right) = 0$. Fix a compact set $K \subseteq \partial \RNp \times (0, \infty)$ and $r_0 > 0$ such  that   $r_0 < \frac{\sqrt t}{4}$ for all $(X, t) \in K$. Notice that for any $\tilde Z = (\tilde X, \tilde t) \in \Sigma \cap K$ and any $r \in (0, r_0)$ we have  
\begin{align*}
\frac{\mu \left( \overline{ P_r^+ (\tilde Z)} \right)}{r^{m+1}} &\leq \frac{C}{r^{m+1}} \int_{ \overline{B_r^+ (\tilde X)} \times [\tilde t - r^2, \tilde t + r^2] } e^{ - \frac{|X- \tilde X|^2}{4 r^2} } d\mu\\
&\leq C \int_{ \overline{\RNp} \times [\tilde t - 2 r^2, \tilde t + r^2]} \G_{\tilde Z + (0, 2 r^2)} d\mu = C \int_{ \overline{T_r^+ (\tilde Z + (0, 2 r^2))}} \G_{\tilde Z + (0, 2 r^2)} d\mu\\
&\leq C \int_{\overline{T_{r_0}^+ (\tilde Z + (0, 2 r^2))}} \G_{\tilde Z + (0, 2 r^2)} d\mu \leq C.
\end{align*}
Since $\p^{m+1} \left( \Sigma \cap K \right) = 0$,  for each $\delta > 0$ there exists a covering sequence of parabolic cylinders $\{ P_{r_i} (\tilde Z_i) \}_{i = 1}^{\infty}$ centered at $\tilde Z_i \in \Sigma \cap K$ with radii  $r_i \in (0, r_0)$ such that  $$\Sigma \cap K \subseteq \bigcup_{i=1}^{\infty} \overline{ P_{r_i} (\tilde Z_i) }\quad\text{ and }\quad\sum_{i = 1}^{\infty} r_i^{m+1} < \delta.$$
Using the local upper density bound for $\mu$, it follows that $$\mu \left( \Sigma \cap K \right) \leq \sum_{i = 1}^{\infty} \mu \left( \overline{ P_{r_i}^+ (\tilde Z_i)} \right) \leq C \sum_{i = 1}^{\infty} r_i^{m+1} < C \delta.$$
Taking the limit as $\delta \to 0$ shows that $\mu (\Sigma \cap K) = 0$. By the arbitrariness of the compact set $K$, we conclude that $\mu (\Sigma) = 0$, which immediately implies that $\nu \equiv 0$.
\end{proof}

\begin{prop}
    If $\eta > 0$ then $\p^{m - 1} (\Sigma) > 0$.
\end{prop}

\begin{proof}
    We will show that the contrapositive of the above statement holds. Therefore let us assume that $\p^{m - 1} ( \Sigma) = 0$. Consider a compact set $K \subseteq \partial \RNp \times (0, \infty)$. Fix any $\alpha, \delta > 0$ small enough so that $P_{3 \alpha} (K) \subseteq \R^{m + 1} \times (0, \infty)$. Then we can find at most countably many  parabolic cylinders $\{ P_{r_i} (Z_i) \}_i$ with $Z_i \in K$ and $r_i < \alpha$ such that $$\Sigma \cap K \subseteq \bigcup_i P_{r_i} (Z_i), \text{ and } \sum_i r_i^{m - 1} < \delta.$$
    Now using Lemma \ref{local_energy_ineq1} and Lemma \ref{lem-mono} we have
    \begin{align*}
    \eta (\Sigma \cap K) \leq& \sum_i \liminf_{\eps \to 0} \int_{P_{r_i}^+ (Z_i)} |\partial_t U_{\eps}|^2 \, dX \, dt\\
    \leq& C \sum_i r_i^{- 2} \liminf_{\eps \to 0} \left( \int_{P_{2 r_i}^+ (Z_i)} |\nabla_X U_{\eps}|^2 \, dX \, dt + \int_{\partial^0 P_{2 r_i}^+ (Z_i)} \frac{F(U_{\eps})}{\eps^2} dx \, dt \right)\\
    \leq& C \sum_i r_i^{- 2} \liminf_{\eps \to 0} r_i^{ m + 1} \E_{\eps} (U_{\eps}, Z_i + (0, 28 r_i^2), 2 \sqrt 2 r_i)\\
    \leq& C \sum_i r_i^{m - 1} < C \delta.
    \end{align*}
    Using the arbitrariness of $\delta$ and $K$ we see that $\eta( \Sigma ) = 0$. Together with Proposition \ref{configuration of Sigma} this implies that $\eta \equiv 0$.
\end{proof}

\begin{prop}\label{time_singular_set}
For $\delta \geq 0$ define $$F_\delta := \left\{ Z \in\Sigma| \limsup_{r \to 0} \liminf_{ \ve \to 0} \frac{1}{r^{m - 1}} \int_{P_r^+ (Z)} |\partial _t U_\ve|^2 \: dX dt > \delta   \right\}.$$ Then for each $\delta > 0$, $F_\delta$ has locally finite $\p^{m - 1}$ measure. In particular, when $m = 1$, $F_{\delta}$ is locally finite for each $\delta > 0$ and $F_0$ is at most countable.

\end{prop}

\begin{proof}
    Let us set $$\tilde{\eta} := |\partial_t U|^2 dX dt + \eta = \lim_{\eps \to 0} \, |\partial_t U_{\eps}|^2 dX dt.$$
    Then for any $\delta > 0$ we have $$F_{\delta} \subseteq \left\{ Z \in \Sigma| \: \limsup_{r \to 0} \frac{1}{r^{m - 1}} \tilde{\eta} \left( \overline{P_{2 r}^+ (Z)} \right) > \delta \right\}.$$
    Appealing to the density results for Radon measures (see \cite[2.10.19]{F}) we conclude the first part of the proposition. The rest follows immediately from this by taking $m = 1$.
\end{proof}

\section{The Conformal Case}

In this section we present the proofs of Theorems \ref{main thm_intro},    \ref{energy identity_intro} and  
\ref{reverse bubbling}. We  begin with Theorem \ref{energy identity_intro},  whose proof relies crucially on an energy identity proved in \cite[Theorem 5.1]{JLZ}. 
To set the framework, we briefly recall some essential notations: For any $a \in ( - \infty, 0]$ we define $\R^2_{\geq a} := \{ (x, y) \in \R^2 | \: y \geq a\}$ and for $a \in [- \infty, 0], \: \R^2_{> a} := \{ (x, y) \in \R^2 | \: y > a\}$. Needless to mention, in this section the domains for the spatial variables of any map will be taken to be subsets of $\R^2$.

\begin{lem}[\cite{JLZ}]\label{free boundary bubbling}
Let $\tilde N$ be a compact Riemannian manifold (possibly with non-empty boundary) that is isometrically embedded in $\R^L$ and $N$ be an embedded submanifold of $\tilde N$ with $N \cap \partial \tilde N = \emptyset$. Also let $v_i \in H^2 \left( B_1^+(0); \tilde N \right)$ be a sequence of maps with tension fields $\tau(v_i)$ and with free boundaries $v_i(\partial^0 B_1^+(0))$ on $N$ and satisfying
\begin{itemize}
\item[(a)] $\|v_i \|_{H^1 \left( B_1^+(0) \right)} + \| \tau(v_i) \|_{L^2 \left( B_1^+(0) \right)} \leq C$,

\item[(b)] $v_i \to v$ strongly in $H^1_{loc} \left( \overline{B_1^+(0)} \setminus \{ 0 \}; \R^L \right)$, as $i \to \infty$,

\item[(c)]  on $\partial^0 B_1^+(0)$ we have $v_i \in N$ and $\partial_y v_i \perp T_{v_i} N$.
\end{itemize}
Then there exist a subsequence of $(v_i)$ (still denoted by $(v_i)$) and a non-negative integer $b$ such that, for any $j \in \{1, \ldots, b \}$, there exists a sequence of points $(X^j_i)_i$ in $\overline{B_1^+ (0)}$, a sequence of positive numbers $(\lambda^j_i)_i$ and either a non-constant harmonic sphere $\omega^j$ or a non-positive constant $a^j$ and a non-constant harmonic disk $\omega^j$ (which we view as a map from $\R^2_{\geq a^j} \cup \{ \infty \} \to \tilde N$) with free boundary $\omega^j(\partial \R^2_{\geq a^j})$ on $N$ such that
\begin{itemize}
\item[(1)] for each $j \in \{1, \ldots, b\}$, $X^j_i \to 0, \lambda^j_i \to 0$, as $i \to \infty$,

\item[(2)] for each $j \in \{1, \ldots, b\}$, $- \frac{\textnormal{dist}(X^j_i, \partial^0 B_1^+(0))}{\lambda^j_i} \to a^j$ or $ \frac{\textnormal{dist}(X^j_i, \partial^0 B_1^+(0))}{\lambda^j_i} \to \infty$, as $i \to \infty$,

\item[(3)] for any $j, k \in \{1, \ldots, b \}$ with $j \neq k$ we have $\frac{\lambda^j_i}{\lambda^k_i} + \frac{\lambda^k_i}{\lambda^j_i} + \frac{|X^j_i - X^k_i|}{\lambda^j_i + \lambda^k_i} \to \infty$, as $i \to \infty$,

\item[(4)] for each $j \in \{1, \ldots, b\}, \, \omega^j$ is the weak limit of $v_i (X^j_i + \lambda^j_i \: \cdot)$ in $H^1_{loc} \left( \R^2 \right)$, if $ \frac{\textnormal{dist} (X^j_i, \partial^0 B_1^+(0))}{\lambda^j_i} \to \infty$ or $\omega^j$ is the weak limit of $v_i(X^j_i + \lambda^j_i \: \cdot)$ in $H^1_{loc} \left( \R^{2}_{> a^j} \right)$, if $- \frac{\textnormal{dist} (X^j_i, \partial^0 B_1^+(0))}{\lambda^j_i} \to a^j$,

\item[(5)] we have $$\lim_{i \to \infty} E(v_i; B_1^+(0)) = E(v; B_1^+(0)) + \sum_{j = 1}^b E(\omega^j).$$

\end{itemize}
\end{lem}

The following version of small energy regularity lemma for free boundary approximate harmonic maps has been proved in \cite[Lemma 4.1]{JLZ}.

\begin{lem}[\cite{JLZ}]\label{free boundary small energy regularity} Let $N$ and $\tilde N$ be as in Lemma \ref{free boundary bubbling}. 
Let $v \in H^2( B_2^+(0); \tilde N)$ be a map with tension field $\tau(v) \in L^2(B_2^+(0))$ and with free boundary $v (\partial^0 B_2^+(0))$ on $N$. There exists $\overline{\eps} = \overline{\eps} (\tilde N, N) > 0$ such that if $$\| \nabla v \|_{L^2 (B_2^+(0))} + \| \tau (v) \|_{L^2 (B_2^+(0))} \leq \overline{\eps},$$ then $$ \left\| v - \frac{1}{|B_1^+(0)|} \int_{B_1^+(0)} v \right\|_{H^2 (B_{\frac 1 2}^+(0))} \leq C(\tilde N, N) (\| \nabla v \|_{L^2 (B_1^+(0))} + \| \tau (v) \|_{L^2 (B_1^+(0))}).$$
\end{lem}

The next lemma provides a local energy monotonicity formula for the limit map $U$.

\begin{lem}\label{energy decay}
For $0 < \frac{1}{T} < s < t < T < \infty$, $0 < r < R \leq \overline R < \infty$ and $X_0 \in \partial \R^2_+$, if $U$ is smooth in $\overline{B_{\overline R}^+ (X_0)} \times [s, t]$, then we have $$ \int_{B_r^+ (X_0)} | \nabla_X U(\cdot, t)|^2 \leq \int_{B_R^+ (X_0)} | \nabla_X U(\cdot, s)|^2 + C( \overline R, T) \frac{t - s}{(R - r)^2}.$$
In particular, if $2 r \leq \overline R$, then it holds that $$ \int_{B_r^+ (X_0)} | \nabla_X U(\cdot, t)|^2 \leq \int_{B_{2r}^+ (X_0)} | \nabla_X U(\cdot, s)|^2 + C( \overline R, T) \frac{t - s}{r^2}.$$
\end{lem}

\begin{proof}
Recall that $U$ satisfies
\begin{equation}\left\{
\begin{array}{ll}
    \partial_t U - \Delta_X U = 0, &\text{ in } \R^2_+ \times (0, \infty),  \\
    \rule{0cm}{0.5cm} \partial_y U \perp T_U N, &\text{ on } \partial \R^2_+ \times (0, \infty).
\end{array}\right.
\end{equation}
Let  $\phi \in C_c^{\infty} \left( \R^2 \right)$ be a cut-off function such that $\phi \equiv 1$ on $B_r (X_0)$, $\phi \equiv 0$ outside $B_R (X_0)$, $0 \leq \phi \leq 1$ and $| \nabla \phi | \leq \frac{2}{R - r}$.
Testing the heat equation against $\phi^2 \partial_t U$, integrating over $\mathbb{R}^2_+$, and applying integration by parts together with the boundary condition $\partial_y U \perp T_U N$ (which yields $\partial_t U \cdot \partial_y U = 0$ on $\partial \mathbb{R}^2_+$ since $\partial_t U \in T_U N$), we obtain
\begin{align*}
0 =& \int_{\R^2_+} |\partial_t U|^2 \phi^2 - \int_{\R^2_+} \phi^2 \: \partial_t U \cdot \D_X U\\
=& \int_{\R^2_+} |\partial_t U|^2 \phi^2 + \frac 1 2 \frac{d}{dt} \int_{\R^2_+} \phi^2 |\nabla_X U|^2 + 2 \int_{\R^2_+} \phi \: \partial_t U \cdot (\nabla \phi \cdot \nabla_X) U + \int_{\partial \R^2_+} \phi^2 \: \partial_t U \cdot \partial_y U\\
\geq& \int_{\R^2_+} |\partial_t U|^2 \phi^2 + \frac 1 2 \frac{d}{dt} \int_{\R^2_+} \phi^2 |\nabla_X U|^2 - \int_{\R^2_+} |\partial_t U|^2 \phi^2 - C \| \nabla \phi \|_{L^{\infty}}^2 \int_{B_{\overline R}^+(X_0)} |\nabla_X U|^2 \\
\geq& \frac 1 2 \frac{d}{dt} \int_{\R^2_+} \phi^2 |\nabla_X U|^2 - \frac{C( \overline R, T)}{(R - r)^2},
\end{align*}
where we have used a local uniform $L^{\infty}_t H^1_X$ bound from the Remark \ref{remark lem-mono} for the last inequality. Integrating with respect to time from $s$ to $t$ and applying the definition of $\phi$, we arrive at the desired energy inequality. The final statement of the lemma follows immediately upon setting $R = 2r$.  
\end{proof}

\begin{remark}\label{remark:energy decay}
    It is worth noting that in the set up of Lemma \ref{energy decay} we have in fact derived the following identity $$ \int_{\R^2_+} |\partial_t U|^2 \phi^2 + \frac 1 2 \frac{d}{dt} \int_{\R^2_+} \phi^2 |\nabla_X U|^2 + 2 \int_{\R^2_+} \phi \: \partial_t U \cdot (\nabla \phi \cdot \nabla_X) U = 0.$$
    Integrating this from $s$ to $t$, we obtain the following finer version of Lemma \ref{energy decay} $$|E_{\phi^2} (U(t)) - E_{\phi^2} (U(s))| \leq C \int_s^t \int_{B_{\overline R}^+ (X_0)} |\partial_t U|^2 + C(\overline R, T) \frac{t - s}{(R - r)^2}.$$
\end{remark}

\begin{lem}\label{existence of limiting energy}
Suppose $Z_0 = (X_0, t_0)$ is an isolated point of $\Sigma$. If $\delta > 0$ is such that $P_{2 \delta}(Z_0) \cap \Sigma = \{ Z_0 \}$, then $\lim_{t \nearrow t_0} \int_{B_{\delta}^+(X_0)} \left| \nabla_X U(\cdot, t) \right|^2$ exists.
\end{lem}

We provide two proofs of this result. The first one is the shorter one, whereas certain estimates  obtained  in  the second one will be used in the proof of Theorem \ref{energy identity_intro}.  

\begin{proof}
Fix a smooth cut-off function  $\phi \in C_c^{\infty} (\R^2)$ satisfying  $\phi \equiv 1$ on $B_{\delta} (X_0), \, \phi \equiv 0$ outside $B_{\frac{3 \delta}{2}} (X_0)$ and $0 \leq \phi \leq 1$. For any $\eps > 0$ and $t_0 - \delta^2 < s < t < t_0$, using Lemma \ref{local_energy_ineq1} for $U_{\eps}$ we have 
\begin{align*}
\left| E^{\eps}_{\phi^2} (U_{\eps} (t)) - E^{\eps}_{\phi^2} (U_{\eps} (s)) \right| &\leq 2 \int_s^t \int_{\R^2_+} |\partial_t U_{\eps}|^2 \phi^2 + C \int_s^t \int_{\R^2_+} |\nabla_X U_{\eps}|^2 |\nabla \phi|^2\\
&\leq 2 \int_s^t \int_{\R^2_+} |\partial_t U_{\eps}|^2 \, \phi^2 + C(\phi) |t - s|,
\end{align*}
where the second inequality follows from the  local uniform $L^{\infty}_t H^1_X$ bound in  Remark \ref{remark lem-mono}. Since $\Sigma \cap \left( B_{2 \delta} (X_0) \times (t_0 - \delta^2, t_0) \right) = \emptyset$,  by  Proposition \ref{convergence of approximation outside singular set general target} we may pass to the limit as $\eps \to 0$ to obtain    \begin{align}\left| E_{\phi^2} (U (t)) - E_{\phi^2} (U (s)) \right| \leq 2 \int_s^t \int_{\R^2_+} |\partial_t U |^2 \, \phi^2 + C(\phi) |t - s| \xrightarrow{s,t\,\uparrow t_0} 0    .\end{align}
This and the fact that $U$ is smooth in $P_{2 \delta}^+ (Z_0) \backslash \{ Z_0 \}$ yields 
\begin{align*}
\Bigg| \int_{B_{\delta}^+(X_0)} \left| \nabla_X U(\cdot, t) \right|^2 &- \int_{B_{\delta}^+(X_0)} \left| \nabla_X U(\cdot, s) \right|^2 \Bigg| \leq \left| \int_{\R^2_+} \left| \nabla_X U(\cdot, t) \right|^2 \phi^2 - \int_{\R^2_+} \left| \nabla_X U(\cdot, s) \right|^2  \phi^2 \right|\\
&+ \left| \int_{B_{2 \delta}^+ (X_0) \setminus B_{\delta}^+ (X_0)} \left| \nabla_X U(\cdot, t) \right|^2 \phi^2 - \int_{B_{2 \delta}^+ (X_0) \setminus B_{\delta}^+ (X_0)} \left| \nabla_X U(\cdot, s) \right|^2  \phi^2 \right|\\
&  \xrightarrow{s,t\,\uparrow t_0} 0   . 
\end{align*}

\medskip

\noindent\textit{Second proof.}
This proof  is in the spirit of  of \cite[Proposition 2.1]{Q}.

Let $\delta > 0$ be as in the statement of the lemma. Suppose $t_n \nearrow t_0$ and $s_n \nearrow t_0$ are two sequences such that for some constants $A$ and $B$ we have $$\limsup_{t \nearrow t_0} \int_{B_{\delta}^+(X_0)} \left| \nabla_X U(\cdot, t) \right|^2 = \lim_{n \to \infty} \int_{B_{\delta}^+(X_0)} \left| \nabla_X U(\cdot, t_n) \right|^2 = A + \int_{B_{\delta}^+(X_0)} \left| \nabla_X U(\cdot, t_0) \right|^2,  $$ and $$\liminf_{t \nearrow t_0} \int_{B_{\delta}^+(X_0)} \left| \nabla_X U(\cdot, t) \right|^2 = \lim_{n \to \infty} \int_{B_{\delta}^+(X_0)} \left| \nabla_X U(\cdot, s_n) \right|^2 = B + \int_{B_{\delta}^+(X_0)} \left| \nabla_X U(\cdot, t_0) \right|^2.$$ It follows directly from  Fatou's lemma that $0\leq B\leq A$.   To establish the existence of the limit, it suffices   to show that $B \geq A$. The case $A = 0$ being immediate, we henceforth assume $A > 0$.

Given any $i \in \Zp$, consider $\lambda_i := \frac{\delta}{i}$. For each $i \in \Zp$, the smoothness of $U$ outside $\Sigma$ guarantees the existence of $n_i \in \Zp$ satisfying $$ \int_{B_{\delta}^+(X_0) \setminus B_{\lambda_i}^+(X_0)} \left| |\nabla_X U(\cdot, t_{n_i})|^2 - |\nabla_X U(\cdot, t_0)|^2 \right| < \frac{A}{2i},$$ and  $$ \int_{B_{\delta}^+(X_0)} |\nabla_X U(\cdot, t_{n_i})|^2 - \int_{B_{\delta}^+(X_0)} |\nabla_X U(\cdot, t_0)|^2 \geq A - \frac{A}{2i}.$$
Therefore it holds that
\begin{align*}
\int_{B_{\lambda_i}^+(X_0)} |\nabla_X U(\cdot, t_{n_i})|^2 &= \int_{B_{\delta}^+(X_0)} |\nabla_X U(\cdot, t_{n_i})|^2 - \int_{B_{\delta}^+(X_0) \setminus B_{\lambda_i}^+(X_0)} |\nabla_X U(\cdot, t_{n_i})|^2\\
&\geq A + \int_{B_{\delta}^+(X_0)} |\nabla_X U(\cdot, t_0)|^2 - \frac{A}{2i} - \int_{B_{\delta}^+(X_0) \setminus B_{\lambda_i}^+(X_0)} |\nabla_X U(\cdot, t_0)|^2 - \frac{A}{2i}\\
&\geq \left( \frac{i - 1}{i} \right) A.
\end{align*}
Without loss of generality we may assume that the sequence $(n_i)$ is strictly increasing and $s_1 \leq t_1 \leq s_2 \leq t_2 \leq \cdots$. Then for any $\alpha \in (0, \delta)$, applying the estimate from Remark \ref{remark:energy decay} we obtain
$$\int_{B_{\alpha}^+ (X_0)} |\nabla_X U(\cdot, s_{n_i})|^2 \geq \int_{B_{\lambda_i}^+(X_0)} |\nabla_X U(\cdot, t_{n_i})|^2 + o_i(1)$$ holds for all $i$ large enough so that $\lambda_i < \alpha$.
Combining the above estimates, we find
\begin{align*}
\lim_{i \to \infty} \int_{B_{\delta}^+ (X_0)} | \nabla_X U (\cdot, s_{n_i}) |^2 &= \lim_{i \to \infty} \left( \int_{B_{\delta}^+ (X_0) \setminus B_{\alpha}^+ (X_0)} + \int_{B_{\alpha}^+ (X_0)} \right) | \nabla_X U (\cdot, s_{n_i}) |^2\\
&\geq \int_{B_{\delta}^+ (X_0) \setminus B_{\alpha}^+ (X_0)} |\nabla_X U (\cdot, t_0)|^2 + \limsup_{i \to \infty} \int_{B_{\lambda_i}^+ (X_0)} |\nabla_X U (\cdot, t_{n_i}) |^2\\
&\geq \int_{B_{\delta}^+ (X_0) \setminus B_{\alpha}^+ (X_0)} |\nabla_X U (\cdot, t_0)|^2 + A.
\end{align*}
Now we are done by the arbitrariness of $\alpha$.
\end{proof}

\begin{remark}\label{remark: energy concentration}
    Using the same notations as in Lemma \ref{existence of limiting energy} it follows from the above proof that if $$ \lim_{t \nearrow t_0} \int_{B_{\delta}^+ (X_0)} |\nabla_X U(\cdot, t)|^2 = A + \int_{B_{\delta}^+ (X_0)} |\nabla_X U(\cdot, t_0)|^2,$$ for some $A > 0$, then we can find sequences $t_i \nearrow t_0$ and $\lambda_i \searrow 0$ satisfying $$ \int_{B_{\lambda_i}^+ (X_0)} |\nabla_X U(\cdot, t_i)|^2 \geq \left( \frac{i - 1}{i} \right) A.$$
\end{remark}

\begin{lem}\label{fwd_bubble_exist}
    Suppose $Z_0 = (X_0, t_0)$ is an isolated point of $\Sigma$. Let   $ \delta > 0$ be such that  $P_{2 \delta} (Z_0) \cap \Sigma = \{ Z_0 \}$. If  $$ \lim_{t \nearrow t_0} \int_{B_{\delta}^+ (X_0)} |\nabla_X U (\cdot, t)|^2 = A + \int_{B_{\delta}^+ (X_0)} |\nabla_X U (\cdot, t_0)|^2,$$  for some constant $A$,  then   $A$ must be positive. 
\end{lem}

\begin{proof}
Fatou's lemma yields that $A \geq 0$. For the sake of contradiction,  assume  that $A = 0$.
By lowering the value of $\delta$ if necessary, we may assume without loss of generality that $$\int_{B_{\delta}^+ (X_0)} |\nabla_X U|^2 (\cdot, t_0) \,dX < \frac{\eps_0^2}{C_1},$$ where $C_1$ is to be determined later. As $A = 0$, we can find $r_1 > 0 $ so that $$\int_{B_{\delta}^+ (X_0)} |\nabla_X U|^2 (\cdot, t) \,dX \leq \int_{B_{\delta}^+ (X_0)} |\nabla_X U|^2 (\cdot, t_0) \,dX + \frac{\eps_0^2}{C_1} < \frac{2 \eps_0^2}{C_1}, \quad \forall t \in (t_0 - 4 r_1^2, t_0).$$
Applying Proposition \ref{energy_tail_est} we obtain $M > 0$ such that for each $r \in (0, r_1)$,
\begin{align*}
\liminf_{\eps \to 0} \E_{\eps} (U_{\eps}, Z_0, r) \leq \frac{C}{r^2} \int_{t_0 - 4 r^2}^{t_0 - r^2} \int_{B_{Mr}^+ (X_0)} |\nabla_X U|^2 dXdt + \frac{\eps_0^2}{2},
\end{align*}
thanks to Proposition \ref{convergence of approximation}. Therefore for sufficiently small $r \in (0, r_1)$ we obtain  
\begin{align*}
    \liminf_{\eps \to 0} \E_{\eps} (U_{\eps}, Z_0, r) \leq& \frac{C}{r^2} \int_{t_0 - 4 r^2}^{t_0 - r^2} \int_{B_{\delta}^+ (X_0)} |\nabla_X U|^2 dXdt + \frac{\eps_0^2}{2}\\
    \leq& \frac{C \eps_0^2}{C_1} + \frac{\eps_0^2}{2} < \eps_0^2,
\end{align*}
provided $C_1$ is sufficiently large. But this is a contradiction as $Z_0 \in \Sigma$. This completes the proof.
\end{proof}

With these results at hand, we are   now in a position to establish  the energy identity at an isolated singular point in our context.

\begin{proof}[Proof of Theorem \ref{energy identity_intro}] Let $U$ be the weak limit of $(U_\eps)$ in $H^1_{loc}(\overline{\R^{2}_+}\times(0,\infty) )$. Assume that  $Z_0=(X_0,t_0)\in\Sigma$ is an isolated singular point, and let $R>0$ be such that $\Sigma\cap P_{2R}(Z_0)=\{Z_0\}$. It then follows from  Lemma \ref{existence of limiting energy} and Lemma \ref{fwd_bubble_exist} that there exists $A > 0$ satisfying $$\lim_{t \nearrow t_0} \int_{B_R^+(X_0)} \left| \nabla_X U(\cdot, t) \right|^2 = A + \int_{B_R^+(X_0)} \left| \nabla_X U(\cdot, t_0) \right|^2.$$
Moreover, by Remark \ref{remark: energy concentration},  there exist   $\lambda_i \searrow 0$ and $t_i \nearrow t_0$ such that
\begin{align}\label{lambdai-ti}
\int_{B_{\lambda_i}^+ (X_0)} | \nabla_X U (\cdot, t_i) |^2 \geq \left( \frac{i - 1}{i} \right) A, \quad \forall i \in \Zp.
\end{align}
We  define $$\tilde V_i (X, t) := U (X_0 + \lambda_i X, t_i + \lambda_i^2 t), \quad \forall (X, t) \in \overline{\R^2_+} \times (- \frac{t_i}{\lambda_i^2}, \infty).$$
Notice that for any fixed $T > 0$, $$ \int_{-T}^0 \int_{B_{R/\lambda_i}^+} |\partial_t \tilde V_i|^2 = \int_{t_i - \lambda_i^2 T}^{t_i} \int_{B_R^+ (X_0)} | \partial_t U|^2  \xrightarrow{i\to\infty}0  .$$
Therefore, we can choose $\overline t_i \in (- 1, 0)$ such that $\overline{t}_i \to 0$ and
\begin{equation}\label{eq: tension field L2 bound}
\lim_{i \to \infty} \int_{B_{R/\lambda_i}^+} |\partial_t \tilde V_i|^2 (\cdot, \overline t_i) = 0.
\end{equation}
Finally, we  set $$v_i(X) := \tilde V_i (X, \overline t_i) = U (X_0 + \lambda_i X, t_i + \lambda_i^2 \overline t_i), \quad \text{and } s_i := t_i + \lambda_i^2 \overline t_i.$$

\medskip

\noindent\textbf{\underline{Step 1 :}} The sequence $(v_i)$ satisfies the hypotheses of Lemma \ref{free boundary bubbling}.

\medskip 

An application of the energy bounds   in Remark \ref{remark lem-mono} yields  
\begin{equation}\label{eq:finite energy}
E(v_i; B_{R/\lambda_i}^+) \leq E(U(\cdot, s_i); B_R^+ (X_0)) \leq C.
\end{equation}
By virtue of the above estimate,  after passing to a subsequence if necessary,   we may assume that  the  Radon measures $|\nabla v_i|^2 dX$ converges to some Radon measure on $\overline{\R^2_+}$. Moreover, the heat equation satisfied by $\tilde{V}_i$ in the interior of its domain implies that  $$\D v_i = \partial_t \tilde V_i (\cdot, \overline t_i)\quad\text{in }\R^2_+. $$ Combined with  \eqref{eq: tension field L2 bound}, this  shows that the tension field of $v_i$ converges to zero in $L^2_{loc} ( \overline{\R^2_+} )$.

Consider $\tilde N$ to be a large closed ball in $\R^L$  containing $N$ in its interior. We define the energy concentration set $$ S:= \left\{ X \in \overline{\R^2_+}| \: \liminf_{i \to \infty} E \left( v_i; B_r^+ (X) \right) \geq \overline{\eps}, \text{ for each } r > 0 \right\}, $$ where $\overline{\eps}$ is the constant given by the small energy regularity estimate in Lemma \ref{free boundary small energy regularity}. From the uniform energy bound in  \eqref{eq:finite energy} we deduce  that $S$ is a finite (possibly empty) subset of $\partial \R^2_+$, 
say $S = \{ X_1, \ldots, X_k\}$. Again by   Lemma \ref{free boundary small energy regularity} and using that $\D v_i\to0$ in $L^2_{loc}(\overline{\R^2_+})$, up to a subsequence,  we obtain  $$v_i \to v\quad \text{in }H^1_{loc} (\overline{\R^2_+} \setminus S)\cap C^{\alpha}_{loc} (\overline{\R^2_+} \setminus S)\quad\text{and weakly in }H^2_{loc} (\overline{\R^2_+}),$$  for some limit map $v$. This implies  $v  \left(\partial \R^2_+\right)\subseteq  N$.

Also for any $\phi \in L^{\infty} \cap H^1 (\partial \R^2_+; \R^L)$ satisfying $\phi \in T_{v} N$, let $\phi = \phi_i^{\top} + \phi_i^{\perp}$ with $\phi_i^{\top} \in T_{v_i}N$ and $\phi_i^{\perp} \in T_{v_i}^{\perp} N$ for each $i \in \Zp$. It is straightforward to see that $\phi_i^{\top} \to \phi$ and $\phi_i^{\perp} \to 0$ strongly in $L^2_{loc} (\partial \R^2_+)$. Therefore using the weak $H^1_{loc} (\partial \R^2_+ \setminus S)$ convergence of $v_i$ to $ v$ and the fact that $\partial_y v_i \perp T_{v_i} N$, we obtain for any compact set $K \subseteq \partial \R^2_+$ $$\int_K \phi \cdot \partial_y v = \lim_{i \to \infty} \int_K \phi \cdot \partial_y v_i = \lim_{i \to \infty} \int_K \phi_i^{\perp} \cdot \partial_y v_i = 0.$$ Hence $v$ is a harmonic map with free boundary on $\partial \R^2_+$. In fact, as a corollary of the regularity theory for harmonic maps with free boundary on two-dimensional domains (see \cite{scheven}), $v$ satisfies
\begin{equation}\left\{
\begin{array}{ll}
\D v = 0, &\text{ in } \R_+^2,\\
\rule{0cm}{0.5cm} \partial_y v \perp T_v N, &\text{ on } \partial \R^2_+.
\end{array}\right.
\end{equation}

At this point if $S$ is empty then we may skip the remainder of this step and proceed to the next one. If $S$ is non-empty then going back to the sequence $(v_i)$,
we see that it satisfies the hypotheses of Lemma \ref{free boundary bubbling} locally around points of $S$. Therefore applying Lemma \ref{free boundary bubbling},    for any  $\rho > 0$ large enough so that $S \subseteq B_{\rho}$, we get   \begin{align}\label{enrg-vi}\lim_{i \to \infty} E (v_i; B_{\rho}^+) = E( v; B_{\rho}^+) + \sum_{j = 1}^k \sum_{l = 1}^{n_j} E(\omega^j_l),\end{align}
where for each  $j \in \{ 1, \ldots, k \}$, the maps  $ \omega^j_1, \ldots, \omega^j_{n_j}$ are the non-trivial bubbles obtained  in  Lemma \ref{free boundary bubbling} in a neighborhood of $X_j$. For  $j \in \{ 1, \ldots, k\}$ and $l \in \{ 1, \ldots, n_j\}$ let $(X^j_{l,i})_i$, $(\lambda^j_{l, i})_i$ be the sequences and $a^j_l$ as given in the statement of Lemma \ref{free boundary bubbling}. Then $$\lim_{i \to \infty} X^j_{l, i} = X_j\quad  \text{ and } \lim_{i \to \infty} \lambda^j_{l, i} = 0,$$ and $\omega^j_l$ is defined on $\R^2_{\geq a^j_l} \cup \{\infty\}$. Here we emphasize   that for every possible values of $j$ and $l$ we must have $a^j_l > - \infty$. Indeed, if  $a^j_l = - \infty$, then $\omega^j_l$ would be a bounded harmonic map from $\R^2$ to $\R^L$, forcing it to be constant and thus $E(\omega^j_l) = 0$, which contradicts the non-triviality of bubbles. Hence, $a^j_l > -\infty$ holds for every bubble.  

It follows that  $$v_i - v - \sum_{j = 1}^k \sum_{l = 1}^{n_j} \left[ \omega^j_l \left( \frac{\cdot - X^j_{l, i}}{\lambda^j_{l, i}} \right) - \omega^j_l (\infty) \right] \to 0, \quad \text{weakly in } H^1_{loc} (\overline{\R^2_+}).$$
In order to show that the above convergence is strong in $ H^1_{loc} (\overline{\R^2_+})$, we first note that  for any $\rho > 0$ as above and for any  $\delta > 0$ small  $$\lim_{i \to \infty} \| v_i - v \|_{H^1 ( B_{\rho}^+ \setminus B_{\delta}^+ (S))} = 0\quad \text{and } \lim_{i \to \infty} \left\| \omega^j_l \left( \frac{\cdot - X^j_{l, i}}{\lambda^j_{l, i}} \right) - \omega^j_l (\infty) \right\|_{H^1 ( B_{\rho}^+ \backslash B_{\delta}^+ (S))} = 0,$$
where $B_{\delta}^+(S) := \cup_{j = 1}^k B_{\delta}^+(X_j)$. Since the Dirichlet energies of $v_i$  and the bubbles $\omega^j_l$  are uniformly bounded, and   $\int_{B_{\delta}^+(S)} |\nabla v|^2 = o_{\delta} (1)$, we obtain 
\begin{align*}
\int_{B_{\delta}^+(S)}& \left| \nabla \left( v_i - v - \sum_{j = 1}^k \sum_{l = 1}^{n_j} \omega^j_l \left( \frac{\cdot - X^j_{l, i}}{\lambda^j_{l, i}} \right) \right) \right|^2 \\
=& \int_{B_{\delta}^+(S)} | \nabla v_i |^2 + \sum_{j = 1}^k \sum_{l = 1}^{n_j} \int_{B_{\delta}^+(S)} \left| \nabla \left( \omega^j_l \left( \frac{\cdot - X^j_{l, i}}{\lambda^j_{l, i}} \right) \right) \right|^2- 2 \sum_{j, l} \int_{B_{\delta}^+(S)} \nabla v_i \cdot \nabla \left( \omega^j_l \left( \frac{\cdot - X^j_{l, i}}{\lambda^j_{l, i}} \right) \right) \\
&\quad -\sum_{(j, l) \neq (j', l')} \int_{B_{\delta}^+(S)} \nabla \left( \omega^{j'}_{l'} \left( \frac{\cdot - X^{j'}_{l', i}}{\lambda^{j'}_{l', i}} \right) \right) \cdot \nabla \left( \omega^j_l \left( \frac{\cdot - X^j_{l, i}}{\lambda^j_{l, i}} \right) \right) + o_{\delta} (1).
\end{align*}
We  analyze each term above individually.  First, using the energy identity \eqref{enrg-vi}   for $(v_i)$  together with  strong convergence away from $S$, we obtain  
\begin{align*}
\int_{B_{\delta}^+ (S)} | \nabla v_i |^2 
&= \sum_{j, l} \int_{\R^2_{\geq a^j_l}} |\nabla \omega^j_l |^2 + o_i (1) + o_{\delta} (1). 
\end{align*}
Next, by scale invariance of the Dirichlet energy,  the rescaled bubbles satisfy $$\sum_{j, l} \int_{B_{\delta}^+ (S)} \left| \nabla \left( \omega^j_l \left( \frac{\cdot - X^j_{l, i}}{\lambda^j_{l, i}} \right) \right) \right|^2 = \sum_{j, l} \int_{\R^2_{\geq a^j_l}} |\nabla \omega^j_l|^2 + o_i (1).$$
 Using the weak convergence resulting from Lemma \ref{free boundary bubbling}, we get 
\begin{align*}
\sum_{j, l} \int_{B_{\delta}^+ (S)} \nabla v_i \cdot \nabla \left( \omega^j_l \left( \frac{\cdot - X^j_{l, i}}{\lambda^j_{l, i}} \right) \right) &= \sum_{j, l} \int_{B_{\delta/ \lambda^j_{l, i}}^+ \left( \frac{1}{ \lambda^j_{l, i}} S \right) - \frac{X^j_{l, i}}{\lambda^j_{l, i}}} \nabla \left( v_i (X^j_{l, i} + \lambda^j_{l, i} \: \cdot) \right) \cdot \nabla \omega^j_l\\
&= \sum_{j, l} \int | \nabla \omega^j_l |^2 + o_i (1).
\end{align*}
  Finally, consider distinct pairs $(j, l) \neq (j', l')$. We first suppose $j \neq j'$, so that the two bubbles concentrate at distinct singular points $X_j \neq X_{j'}$. It is then easy to see that 
\begin{align*}
&\int_{B_{\delta}^+ (S)} \nabla \left( \omega^{j'}_{l'} \left( \frac{\cdot - X^{j'}_{l', i}}{\lambda^{j'}_{l', i}} \right) \right) \cdot \nabla \left( \omega^j_l \left( \frac{\cdot - X^j_{l, i}}{\lambda^j_{l, i}} \right) \right)=o_i (1).
\end{align*}
We now consider  the second case where  $j = j'$ and $l \neq l'$. From  Lemma \ref{free boundary bubbling}, the bubble parameters satisfy   $$\lim_{i \to \infty} \left[ \frac{\lambda^j_{l,i}}{\lambda^{j'}_{l',i}} + \frac{\lambda^{j'}_{l',i}}{\lambda^j_{l,i}} + \frac{|X^j_{l,i} - X^{j'}_{l',i}|}{\lambda^j_{l,i} + \lambda^{j'}_{l',i}} \right] = \infty.$$
If the spatial separation term (the third term  above) diverges to infinity, we proceed exactly   as   in the case $j \neq j'$.   Otherwise,  interchanging the roles of $l$ and $l'$ if necessary, we may assume without loss of generality that $$\lim_{i \to \infty} \frac{\lambda^j_{l, i}}{\lambda^{j'}_{l', i}} = 0.$$
For $M \gg 1$ we decompose  $$B_\delta^+(S)=\left(B_\delta^+(S)\cap B_{M\lambda_{l,i}^j}(X_{l,i}^j)\right)\cup\left(B_\delta^+(S) \setminus B_{M\lambda_{l,i}^j}(X_{l,i}^j)\right)=:\Omega_{l,i}\cup \Omega_{l',i}. $$ 
Then we have $$ \int_{\Omega_{l',i}} \left|\nabla \left( \omega^j_l \left( \frac{\cdot - X^j_{l, i}}{\lambda^j_{l, i}} \right) \right)\right|^2=o_i(1) +o_M(1),\quad   \int_{\Omega_{l,i}} \left|\nabla \left( \omega^{j'}_{l'} \left( \frac{\cdot - X^{j'}_{l', i}}{\lambda^{j'}_{\Omega{l', i}}} \right) \right)\right|^2=o_i(1),$$
which leads to  $$\int_{B_{\delta}^+ (S)} \nabla \left( \omega^{j'}_{l'} \left( \frac{\cdot - X^{j'}_{l', i}}{\lambda^{j'}_{l', i}} \right) \right) \cdot \nabla \left( \omega^j_l \left( \frac{\cdot - X^j_{l, i}}{\lambda^j_{l, i}} \right) \right) \to 0\quad  \text{ as } i \to \infty.$$
Combining the above estimates we conclude that  \begin{align}\label{strong-vi}v_i - v - \sum_{j = 1}^k \sum_{l = 1}^{n_j} \left[ \omega^j_l \left( \frac{\cdot - X^j_{l, i}}{\lambda^j_{l, i}} \right) - \omega^j_l (\infty) \right] \xrightarrow{i\to\infty }0  \quad  \text{  in } H^1_{loc} (\overline{\R^2_+}).\end{align}

\medskip

\noindent\textbf{\underline{Step 2 :}} Unraveling the scaling and completing the proof.

\medskip 
We now return to the original heat flow by inverting the rescaling defining $v_i$. Set  $b := n_1 + \cdots + n_k + 1$, and reindex the bubble collections  $\omega^1_1, \ldots, \omega^1_{n_1}, \ldots \: \ldots, \omega^k_1, \ldots, \omega^k_{n_k}$ as $\omega^1, \ldots, \omega^{b-1}$. For each  $j \in \{ 1, \ldots, b - 1 \}$, let  $ j' \in \{1, \ldots, k\}$ and $l \in \{1, \ldots, n_{j'} \}$ denote the indices such  that $\omega^{j} = \omega^{j'}_l$. We set  $X^{j}_i := X_0 + \lambda_i X^{j'}_{l, i}$, $\lambda^{j}_i := \lambda_i \lambda^{j'}_{l, i}$ and $a^j := a^{j'}_l$ for $1\leq j\leq b-1$, and for $j=b$  set  $\omega^b := v$, $X^b_i := X_0$,   $\lambda^b_i := \lambda_i$, $X^{j'}_{l', i} := 0, \lambda^{j'}_{l', i} := 1$ and $a^b := 0$.  Under this unified indexing, by a change of variable,  we deduce for any fix $\rho \gg 1$,
\begin{flalign*}
\int_{B_{R}^+ (X_0)} \left| \nabla \left(  U (\cdot, s_i) - \sum_{j = 1}^b   \omega^j \left( \frac{\cdot - X^j_i} {\lambda^j_i} \right)      \right)\right|^2
&= \int_{B_{R/ \lambda_i}^+} \left| \nabla \left( v_i - \sum_{j = 1}^b \omega^j \left( \frac{\cdot - X^{j'}_{l, i}}{\lambda^{j'}_{l, i}} \right) \right)\right|^2  \\
=& \left( \int_{B_{\rho}^+} + \int_{B_{R/ \lambda_i}^+ \backslash B_{\rho}^+} \right) \left| \nabla \left( v_i - \sum_{j = 1}^b \omega^j \left( \frac{\cdot - X^{j'}_{l, i}}{\lambda^{j'}_{l, i}} \right) \right)\right|^2&&\\
=& \int_{B_{R/ \lambda_i}^+ \backslash B_{\rho}^+}  | \nabla   v_i   |^2 + o_{i} (1)+o_\rho(1)\\
=& \int_{B_R^+ (X_0) \backslash B_{\rho \lambda_i}^+ (X_0)} |  \nabla  U |^2 ( \cdot, s_i) + o_i (1)+o_\rho(1)\\
\geq & \int_{B_R^+ (X_0)  } |  \nabla  U |^2 ( \cdot, t_0) + o_i (1)+o_\rho(1),
\end{flalign*} where the above last inequality follows from Fatou's lemma.  Recalling that $s_i=t_i+\lambda_i^2 \bar t_i$, $\bar t_i\uparrow0$, \eqref{lambdai-ti} and the estimate from Remark \ref{remark:energy decay}, we infer
\begin{align*}
\liminf_{i\to\infty}  \int_{ B_{2 \lambda_i}^+ (X_0)} |  \nabla  U |^2 ( \cdot, s_i)\geq \liminf_{i\to\infty}  \int_{ B_{ \lambda_i}^+ (X_0)} |  \nabla  U |^2 ( \cdot, t_i)-C\bar t_i\geq A .
\end{align*} 
Since $\rho \gg 1$, \begin{align*}
\int_{B_R^+ (X_0) \backslash B_{\rho \lambda_i}^+ (X_0)} |  \nabla  U |^2 ( \cdot, s_i)=\int_{B_R^+ (X_0)  } |  \nabla  U |^2 ( \cdot, s_i) -\int_{  B_{\rho \lambda_i}^+ (X_0)} |  \nabla  U |^2 ( \cdot, s_i) \leq \int_{B_R^+ (X_0)  } |  \nabla  U |^2 ( \cdot, t_0)+o_i(1) . \end{align*}
Combining the above estimates and together with \eqref{strong-vi} we conclude that 
\begin{equation}\label{eq:convergenceofnorm}
\lim_{i \to \infty} \int_{B_{R}^+ (X_0)} \left| \nabla \left(  U (\cdot, s_i) - \sum_{j = 1}^b \left[ \omega^j \left( \frac{\cdot - X^j_i}{\lambda^j_i} \right) - \omega^j (\infty) \right] \right) \right|^2 = \int_{B_R^+ (X_0)} |  \nabla_X U |^2 ( \cdot, t_0).
\end{equation}
Since $$  U (\cdot, s_i) - \sum_{j = 1}^b \left[ \omega^j \left( \frac{\cdot - X^j_i}{\lambda^j_i} \right) - \omega^j (\infty) \right] \to U (\cdot, t_0), \text{ almost everywhere in } B_R^+ (X_0),$$   the convergence of the norms in \eqref{eq:convergenceofnorm} implies the above convergence holds strongly in $H^1(B_R^+(X_0))$. This completes the proof.
\end{proof}

We now prove the reverse bubbling phenomena stated in Theorem \ref{reverse bubbling}.

\begin{proof}[Proof of Theorem \ref{reverse bubbling}]
The argument is analogous to the proof of Theorem \ref{energy identity_intro} and  hence, we only outline the main steps. Following the reasoning in Lemma \ref{existence of limiting energy}, there exists a constant $A \ge 0$ such that for any $R > 0$ as in the statement of Theorem \ref{reverse bubbling}, $$\lim_{t \searrow t_0} E(U(t); B_R^+(X_0)) = A + E(U(t_0); B_R^+(X_0)).$$
Moreover, we can extract a sequence of times $t_i \searrow t_0$ and a scaling sequence $\lambda_i \searrow 0$ such that $$\int_{B_{\lambda_i}^+(X_0)} \vert{}\nabla_X U(\cdot, t_i)\vert{}^2 \, dX \ge A + o_i(1).$$
As shown earlier, define the rescaled functions $\tilde V_i (X, t) := U(X_0 + \lambda_i X, t_i + \lambda_i^2 t)$ and choose a sequence $\overline t_i \downarrow 0$ with the property $$\lim_{i \to \infty} \int_{B_{R/ \lambda_i}^+} |\partial_t \tilde V_i|^2 (\cdot, \overline{t}_i) = 0.$$
Set $v_i := \tilde{V}_i (X, \overline{t}_i)$ and $s_i := t_i + \lambda_i^2 \overline{t}_i$.
Performing the blow-up analysis on the sequence $(v_i)$ and unwinding the rescaling as in the proof of Theorem~\ref{energy identity_intro} then yields the desired bubble decomposition and strong $H^1(B_R^+(X_0))$ convergence.
\end{proof}

\begin{remark*}
    It is worth noting that the bubbling analysis in Theorem \ref{energy identity_intro} remains valid for non-isolated points $Z_0 \in \Sigma$ admitting some $R > 0$ that satisfies $$ \left( B_R (X_0) \times [t_0 - R^2, t_0] \right) \, \bigcap \, \Sigma = \{ Z_0\}.$$
    Similarly, Theorem \ref{reverse bubbling} applies to points $Z_0 \in \Sigma$ for which there exists an $R > 0$  satisfying $$ \left( B_R (X_0) \times [t_0, t_0 + R^2] \right) \, \bigcap \, \Sigma = \{ Z_0\}.$$
\end{remark*}

\begin{cor}\label{bubbling_energy_relation}
Suppose that $Z_0 = (X_0, t_0)$ is an isolated point in $\Sigma$ with $P_{2 R} (Z_0) \cap \Sigma = \{ Z_0 \}$. We know that there exist a positive real number $A^-$ and a non-negative real number $A^+$ such that 
\begin{align*}
\lim_{t \nearrow t_0} E(U(t); B_R^+( X_0 )) &= E(U(t_0); B_R^+ ( X_0 )) + A^-, \: \: \text{ and}\\
\lim_{t \searrow t_0} E(U(t); B_R^+( X_0 )) &= E(U(t_0); B_R^+ ( X_0 )) + A^+.
\end{align*}
Then for any $s \in (t_0 - R^2, t_0), t \in (t_0, t_0 + R^2)$ and $\phi \in C_c^{\infty} \left( B_R (X_0) \right)$ we have
\begin{equation}
E_{\phi} (U (t)) - E_{\phi} (U (s)) = - \int_s^t \int_{\R^2_+} |\partial_t U|^2 \phi - \int_s^t \int_{\R^2_+} \phi d\eta - \int_s^t \int_{\R^2_+} \langle \partial_t U, (\nabla \phi \cdot \nabla_X) U \rangle,
\end{equation}
where $\eta$ is the Radon measure introduced in \eqref{eq:Radonmeasureintro}. In particular, we have $$ A^- = A^+ + \eta( \{ Z_0 \}).$$
\end{cor}

\begin{proof}
Using Lemma \ref{local_energy_ineq1} we have that for any $\eps > 0$, $t_0 - R^2 < s < t_0 < t < t_0 + R^2$ and $\phi \in C_c^{\infty} (B_R ( X_0 ))$, 
\begin{equation}\label{change of energy}
E^{\eps}_{\phi} (U_{\eps} (t)) - E^{\eps}_{\phi} (U_{\eps} (s)) = - \int_s^t \int_{\R^2_+} |\partial_t U_{\eps}|^2 \phi - \int_s^t \int_{\R^2_+} \langle \partial_t U_{\eps}, (\nabla \phi \cdot \nabla_X) U_{\eps} \rangle.
\end{equation}

\noindent Due to Proposition \ref{convergence of approximation outside singular set general target} and Lemma \ref{eps_gradient_est} we can pass to the limit as $\eps \to 0$, to obtain
\begin{equation}\label{change of energy of limit}
E_{\phi} (U (t)) - E_{\phi} (U (s)) = - \int_s^t \int_{\R^2_+} |\partial_t U|^2 \phi - \int_s^t \int_{\R^2_+} \phi d\eta - \int_s^t \int_{\R^2_+} \langle \partial_t U, (\nabla \phi \cdot \nabla_X) U \rangle.
\end{equation}
Here the convergence of the last term follows from Proposition \ref{convergence of approximation outside singular set general target}, Lemma \ref{eps_gradient_est} and the local uniform $H^1$ and $L^{\infty}_t H^1_X$ bounds of $(U_{\eps})$, aided by Remark \ref{remark lem-mono}.
Taking $\phi$ such that $\phi \equiv 1$ in a small neighborhood of  $ X_0 $ and passing to the limits  $s \nearrow t_0$ and $t \searrow t_0$ in \eqref{change of energy of limit} we deduce  $$ A^+ - A^- = - \eta (\{ Z_0 \}).$$
\end{proof}

\begin{proof}[Proof of Theorem \ref{main thm_intro}]  Let $U$ be the weak limit of $(U_\eps)$ in $H^1_{loc}(\overline{\R^{2}_+} \times (0,\infty) )$. 
Let $\Sigma$ be non-empty. First of all let us consider the case when $\Sigma$ contains an isolated point, say $\overline Z = (\overline X, \overline t)$. Let $R > 0$ be such that $P_{2R} (\overline Z) \cap \Sigma = \{ \overline Z \}$. Then we know that there exists a positive real number $A$ such that $$\lim_{t \nearrow \bar t} E(U(t); B_R^+(\overline X)) = E(U(\overline t); B_R^+ (\overline X)) + A.$$
Recall from Theorem \ref{energy identity_intro} that $A$ can be written as sum of Dirichlet energies of harmonic maps $\omega : \R^2_+ \to \R^L$ with $\omega \left( \partial \R^2_+ \right) \subseteq N$ and $\partial_y \omega \big|_{\partial \R^2_+} \perp T_{\omega} N$. If any such $\omega$ is non-constant then we are done. If all such $\omega$'s are constants then we get $A = 0$, which contradicts Lemma \ref{fwd_bubble_exist}.
Hence we are done in the case when $\Sigma$ contains an isolated point.

\medskip 

We now assume that  $\Sigma$ is non-empty and  contains no isolated points. As $\Sigma$ is closed, it must also be uncountable. Hence by Proposition \ref{time_singular_set} there exists $Z_0 = (X_0, t_0) = (x_0, 0, t_0) \in \Sigma \backslash F_0$. Moreover, Lemma \ref{partial regularity lemma} shows  that the time slice  $\Sigma^t := \Sigma \cap \left( \R^2 \times \{ t \} \right)$ is locally finite for every $t > 0$. Therefore, we can choose $r_0 > 0$ such that $ B_{2r_0}(X_0) \cap \Sigma^{t_0} = \{ X_0 \}$. For any $r < r_0$ and $\eps > 0$ we define   $$Q_{\eps}(r) := \sup_{x \in \overline{B_{r_0}(x_0)}} \E_{\eps} (U_{\eps}, (x, 0, t_0), r).$$
For every $\eps > 0$, the smoothness of $U_{\eps}$ implies $$\lim_{r \to 0} Q_{\eps}(r) = 0.$$
On the other hand, since   $Z_0 \in \Sigma$,  it follows that  for any $r>0$ $$ \liminf_{\eps \to 0} Q_{\eps} (r) \geq \liminf_{\eps \to 0} \E_{\eps} (U_{\eps}, Z_0, r) > \frac{\eps_0^2}{2 C_1},$$ where 
 $C_1 = 1024 \pi e$ (a constant arising in the proof of Claim 2 below).
 By a standard diagonal process we can choose $\eps_n \searrow 0$, $r_n \searrow 0$,  $ Z_n := (X_n, t_0) := (x_n, 0, t_0)$ such that $$\E_{\eps_n} (U_{\eps_n}, Z_n, r_n) = Q_{\eps_n}(r_n) = \frac{\eps_0^2}{2 C_1}. $$

\noindent\underline{\textbf{Claim 1 :}}  We have $Z_n\to  Z_0$. 

\medskip 

 Since $(Z_n)_n$ is bounded, up to a subsequence,   $Z_n \to \tilde Z := (\tilde X, t_0) \in \partial^0 {B_{r_0}^+ (X_0)}\times\{t_0\}$. By our choice of $r_0$ we have  $B_{2r_0}(X_0) \cap \Sigma^{t_0} = \{X_0\}$, and hence it suffices to show   that $\tilde Z \in \Sigma$.

If possible assume that $\tilde Z \notin \Sigma$. Then by Lemma \ref{eps_gradient_est} there exists $\delta_0 > 0$ such that up to a subsequence, $\| U_{\eps_n} \|_{C^1 \left( P_{2 \delta_0}^+ (\tilde Z) \right)}$ is uniformly bounded (i.e. independent of $n$). Proposition \ref{energy_tail_est} implies that for some $M > 0$ and all sufficiently large $n$ we have
\begin{align*}
    \frac{\eps_0^2}{2 C_1} =& \E_{\eps_n} (U_{\eps_n}, Z_n, r_n) \leq \frac{C}{r_n^2} \int_{t_0 - 4 r_n^2}^{t_0 - r_n^2} \int_{\overline{B_{M r_n}^+ (X_n)}} e(U_{\eps_n}) dX dt + \frac{\eps_0^2}{4 C_1}\\
    \leq& C M^2 \| U_{\eps_n} \|^2_{C^1 \left( \overline{P_{2 \delta_0}^+ (\tilde Z)} \right)} r_n + \frac{\eps_0^2}{4 C_1}\\
    =& o_n(1) + \frac{\eps_0^2}{4 C_1}.
\end{align*}
Letting $n \to \infty$, we arrive at a contradiction.
Thus $\tilde{Z} \in \Sigma$, completing the proof of Claim 1.

\medskip

$$ V_n (X, t) := U_{\eps_n} (X_n + r_n X, t_0 + r_n^2 t),\quad \forall (X,t) \in \overline{\R^2_+} \times \left( - \frac{t_0}{r_n^2}, \infty \right).  $$ Then    $V_n$ satisfies
\begin{equation}\label{rescaled_eqn}
\left\{\begin{array}{ll}
 \partial_t V_n -\D_{X} V_n=0,
&\text{ in } \R_+^2 \times (- \frac{t_0}{r_n^2}, \infty),\\
\rule{0cm}{0.5cm} \partial_y V_n = \frac{1}{\tilde \eps_n^2} (\nabla F) (V_n), &\text{ on } \partial \R^2_+ \times (- \frac{t_0}{r_n^2}, \infty),
\end{array}\right.
\end{equation}
where $\tilde \eps_n := \frac{\eps_n}{\sqrt{r_n}}$. Moreover, for every $R>0$ we have 
\begin{equation}\label{eq:time_cgs}
\liminf_{n\to\infty} \int_{P_R^+}|\partial_t V_n|^2dXdt=0.
\end{equation}

\noindent To see this, recall that from the choice of $Z_0$ we have
\begin{equation}\label{time_int_est}
	\lim_{r \to 0} \liminf_{\eps \to 0} \int_{P_r^+(Z_0)} |\partial_t U_{\eps} |^2 dX dt = 0.
\end{equation}
Therefore, as $Z_n\to Z_0$, for any $R>0$,
\begin{equation}\label{time_est_eqn}
\liminf_{n\to\infty} \int_{P_R^+}|\partial_t V_n|^2dXdt = \liminf_{n\to\infty} \int_{P_{Rr_n}^+(Z_n)}|\partial_t U_{\eps_n}|^2dXdt = 0,
\end{equation}
which proves \eqref{eq:time_cgs}.

\medskip

\noindent\underline{\textbf{Claim 2 :}}  Up to a subsequence, $V_n \to V$   in $C^1_{loc} \left( \overline{\R^2_+} \times \R \right)$, where the limit map $V$ is independent of the time variable $t$.

\medskip

The $C^1_{loc}$ convergence of $V_n$ follows directly from standard parabolic regularity theory, provided $\tilde{\varepsilon}_n \not\to 0$. Therefore, it suffices to consider the case $\tilde{\varepsilon}_n \to 0$. 

For each  $\overline X \in \partial \R^2_+$,   Claim 1 implies that for   $n$ sufficiently large, 
\begin{equation}\label{eq:small_bdry_energy}
\E_{\tilde \eps_n } \left( V_n, (\overline X, 0), 1 \right) = \E_{\eps_n} \left(U_{\eps_n}, ( X_n + r_n \overline X, t_0), r_n \right) \leq \E_{\eps_n} (U_{\eps_n}, Z_n, r_n ) =\frac{\eps_0^2}{2C_1}.
\end{equation}
Extending the small-energy estimate to arbitrary centers in $\partial \mathbb{R}^2_+ \times \mathbb{R}$ allows us to invoke Lemma~\ref{eps_gradient_est}, yielding uniform $C^{1, \alpha}_{\text{loc}}$ estimates for $V_n$ in a small neighborhood of $\overline{\R^2_+} \times \{0 \}$.  
To this end,  fix any $ \overline X \in \partial \R^2_+ $.
The choice of $C_1$,  estimate  \eqref{eq:small_bdry_energy} and the mean value property for integrals yield a time slice  $t_n \in (-4, -1)$ such that for $n$ sufficiently large 
$$ E^{\tilde{\eps}_n}(V_n(t_n); B_2^+ (\overline X)) \leq \frac{\eps_0^2}{8} . $$
Fix $\phi \in C_c^{\infty} (B_2(\overline X))$ such that $\phi \equiv 1$ on $B_1(\overline X)$ and  $0 \leq \phi \leq 1$. From Lemma \ref{local_energy_ineq1} we have for any $s \in \R$,
\begin{equation}\label{energy_identity_eqn}
	E^{\tilde{\eps}_n}_{\phi} \left( V_n(s) \right) = - \int_{t_n}^s \int_{\R^2_+} |\partial_t V_n|^2 \: \phi - \int_{t_n}^s \int_{\R^2_+} \langle \partial_t V_n, (\nabla \phi \cdot \nabla_X) V_n \rangle + E^{\tilde{\eps}_n}_{\phi} \left( V_n(t_n) \right).
\end{equation}
Using  the local uniform $L^{\infty}_t H^1_X$ bound in Remark \ref{remark lem-mono}, \eqref{eq:time_cgs} and together with Holder's inequality we conclude that the first two terms appearing on the right hand side converge to zero as $n \to \infty$. 
In particular, for any $T>0$,    we obtain for   $n$  sufficiently large  (depending on $\overline X$ and $T$) $$ E^{\tilde{\eps}_n} (V_n(s); B_1^+ (\overline X)) < \frac{\eps_0^2}{C_2}, \quad\forall s \in [ - T, T],$$ where the constant $C_2$ is to be determined shortly.
We now consider any $\overline t \in \R$ and let $\overline Z := (\overline X, \overline t)$. Combining the above estimate with Proposition \ref{energy_tail_est} we get an $M > 0$ such that for $r \in (0, 1)$ sufficiently small, it holds that
\begin{align*}
    \E_{\tilde \eps_n} (V_n, \overline Z, r) \leq& \frac{C}{r^2} \int_{\overline t - 4 r^2}^{\overline{t} - r^2} \int_{\overline{B_{M r}^+ (\overline X)}} e(V_n) dX dt + \frac{\eps_0^2}{2}\\
    \leq& \frac{C}{r^2} \int_{\overline t - 4 r^2}^{\overline{t} - r^2} E^{\tilde{\eps}_n} (V_n(s); B_1^+ (\overline X)) ds + \frac{\eps_0^2}{2}\\
    \leq& \frac{C \eps_0^2}{C_2} + \frac{\eps_0^2}{2} < \eps_0^2,
\end{align*}
provided $C_2$ is chosen sufficiently large.

\noindent Hence by Lemma \ref{eps_gradient_est} and together with standard parabolic regularity theory  we conclude that  $(V_n)$ is bounded $C^{1,\alpha}_{loc}(  \overline{\R^2_+} \times \R)$. Therefore, up to a subsequence, $V_n \to V$   in $C^1_{loc} \left( \overline{\R^2_+} \times \R \right)$, and the limit map  $V$ is independent of $t$ follows immediately from \eqref{eq:time_cgs}.

\medskip

\noindent\underline{\textbf{Claim 3 :}} $V$ is the desired non-trivial, smooth harmonic map with free boundary on $N$.

\medskip

We first assume the case that $\tilde{\eps}_n \to 0$. Passing to the limit in the system \eqref{rescaled_eqn}, we find that $V$ satisfies
\begin{equation}\left\{
    \begin{array}{lll}
        \Delta V &= 0, & \text{ in } \R^2_+,  \\
        \rule{0cm}{0.5cm} \partial_y V &\perp T_{V} N, & \text{ on } \partial \R^2_+. 
    \end{array}\right.
\end{equation}
The smoothness of $V$ follows from the regularity theory for free boundary harmonic maps from two-dimensional domains (see \cite{scheven}). For the non-constancy of $V$, we compute using Proposition \ref{energy_tail_est} 
\begin{align}\label{eq:non-trivial}
	\frac{\eps_0^2}{2 C_1} &= \E_{\tilde \eps_n} (V_n, (0, 0), 1)\\
&\leq C \left[ \int_{-4}^{-1} \int_{B_M^+} |\nabla_X V_n|^2 \, dX \, dt + \frac{1}{\tilde{\eps}_n^2} \int_{-4}^{-1} \int_{\partial^0 B_M^+} F(V_n) \, dx \, dt \right] + \frac{\eps_0^2}{4 C_1},
\end{align}
for some $M \gg 1$ and for all $n \in \Zp$. Therefore using Claim 2 and Proposition \ref{convergence of approximation} we obtain$$\int_{-4}^{-1} \int_{B_R^+} |\nabla_X V|^2 \: dX dt = \lim_{n \to \infty} \int_{-4}^{-1} \int_{B_R^+} |\nabla_X V_n|^2 \: dX dt > 0.$$

Therefore it is enough to establish that $\tilde{\eps}_n \to 0$. We prove this by assuming the contrary, which brings us to the following two cases:
\medskip

\noindent\underline{\textbf{Case 1 :}} Up to a subsequence, $ \tilde \eps_n \to \infty$ as $n \to \infty$.

\medskip

From \eqref{rescaled_eqn} and \eqref{eq:time_cgs} we see that $V$ satisfies
\begin{equation}\left\{
\begin{array}{ll}
\D V = 0, &\text{ in } \R_+^{2},\\
\rule{0cm}{0.5cm} \frac{\partial V}{\partial y} = 0, &\text{ on } \partial \R^2_+.
\end{array}\right.
\end{equation}
Therefore an even extension of $V$ to the whole of $\R^{2}$ is a weakly harmonic function and hence is smooth. The $L^{\infty}$ boundedness of $V$ (which follows from that of $U_{\eps}$) then gives that $V$ is constant throughout. But this brings us to a contradiction when we pass limit to \eqref{eq:non-trivial}, as $n \to \infty$.
\medskip

\noindent\underline{\textbf{Case 2 :}} Up to a subsequence, $ \tilde \eps_n \to c \in (0, \infty)$, as $n \to \infty$.

\medskip

Similar to the previous case, from \eqref{rescaled_eqn} and \eqref{eq:time_cgs} we see that $V$ satisfies
\begin{equation}\left\{
\begin{array}{ll}
\D V = 0, &\text{ in } \R_+^{2},\\
\rule{0cm}{0.5cm} \partial_y V = \frac{1}{c^2} (\nabla F) (V), &\text{ on } \partial \R^2_+.
\end{array}\right.
\end{equation}

\noindent Moreover, from  Remark \ref{remark lem-mono} and the second inequality in \eqref{eq:scale_inv_bds} we deduce that that for any $r > 0$, $$E^c (V; B_r^+) \leq C \liminf_{n \to \infty} \mathcal{D}_{\eps_n} (V_n, (0, r^2), r) \leq C,$$ where 
the above constant $C$ is independent of $r$. Letting $r \to \infty$, we get $$\int_{\R^2_+} |\nabla V|^2 dX + \int_{\partial \R^2_+} F(V) \, dx < \infty.$$
Appealing to Proposition \ref{Proposition A.2} we see that $V$ must be constant and $F(V) \equiv 0$. This is a contradiction to \eqref{eq:non-trivial}.

Thus, we must have $\tilde{\eps}_n \to 0$, finally establishing the claim and concluding the proof of the theorem.
\end{proof}

\section{The Problem in Higher Dimensions}

This section is devoted to the proof of Theorems \ref{no harmonic S1_intro} and \ref{fine reg parabolic blow up set}. 

\begin{proof}[Proof of Theorem \ref{no harmonic S1_intro}] Before proceeding with the technical details, we outline the strategy of the proof, which relies on an iterated blow-up analysis and bubble extraction. We start by selecting  a  suitable   point $Z_0 = (X_0, t_0) \in \Sigma$  along with  suitable  sequences $r_k \searrow 0$ and $\eps_k \searrow 0$ to define the rescaled functions $\tilde U_k (X, t) := U_{\eps_k} (X_0 + r_k X, t_0 + r_k^2 t)$ in such a way that their   $H^1_{loc}$ weak limit is a constant function. In this setting,  the corresponding energy concentration set $\Sigma_*$ (analogous to $\Sigma$) coincides with the support of the related defect measure $\mu_*$ (analogous to $\nu$), as established in Proposition \ref{configuration of Sigma}. The key properties of the new energy concentration set $\Sigma_*$ and the measure $\mu_*$ are established  in  Claim 1 and Claim 2. 

Next we choose a suitable point $Z_1 \in \Sigma_*$ and rescaling factors $s_j \searrow 0$ and further rescale $\tilde U_k$ around $Z_1$ to obtain the rescaled functions $V_l$. 
As a result of the particular choice of rescaling, at least $m - 1$ spatial directional derivatives and the time derivative of $V_l$ converge strongly to zero in $L^2_{loc}$ (Claim 3 to 6), reducing the problem to an effective two dimensional spatial setting. However, because  $V_l$ itself converges weakly in $H^1_{loc}$ to a constant by construction, the weak limit alone does not yield a non-trivial map. 

As before, we consider $\Sigma_{**}$ and $\mu_{**}$ to be the energy concentration set and the defect measure corresponding to the sequence $(V_l)$, respectively.
In order to separate the above mentioned $m - 1$ spatial directions from the rest of the two spatial directions, using Claim 7 and Claim 8, we decompose $\mu_{**}$ into the product of an $(m - 1)$-dimensional Hausdorff measure and  the restriction of the parabolic Hausdorff measure $\p^2$ to a self-similar set. 
Finally, we select a suitable energy concentration point in $\Sigma_{**}$ and perform  a standard bubble extraction. Because the sequence effectively depends on only two spatial variables, the analysis reduces to the critical conformal dimension, ensuring that the bubble extraction succeeds and yields a non-trivial bubble.

\medskip

We start the proof by assuming that  $\nu \not\equiv 0$. Choose a point $Z_0 \in \Sigma$ such that  
\begin{equation}\label{eq:zerodensity}
    \lim_{r \searrow 0} r^{- (m + 1)} \int_{P_r^+ (Z_0)} |\nabla_X U|^2 \: dX dt = \lim_{r \searrow 0} r^{- (m + 1)} \int_{P_r^+ (Z_0)} |\partial_t U|^2 \: dX dt = 0.
\end{equation}
Since these equalities are satisfied $\p^{m + 1}$-almost everywhere in $\partial \RNp \times (0, \infty)$, and    Proposition \ref{blowupcndn} ensures that  $\p^{m + 1}(\Sigma) > 0$, such a point $Z_0 = (X_0, t_0) \in \Sigma$ satisfying \eqref{eq:zerodensity} can indeed be chosen. Furthermore, we fix a sequence   $(r_i)$ of positive real numbers decreasing  to $0$.

For any $r \in (0, \infty)$, we define the parabolic dilation $D_r : \R^{m + 2} \to \R^{m + 2}$ by setting 
\begin{equation}\label{eq:parabolic dilation}
D_r (X, t) := (r X, r^2 t),\quad \: \forall \,(X, t) \in \R^{ m + 2 }, 
\end{equation}
and  for any $i \in \Zp$, define the rescaled measure $\mu_i$ as $$\mu_i (A) := r_i^{- (m + 1)} \mu \left( Z_0 + D_{r_i} (A) \right).$$
For  convenience,   any Radon measure on $\overline{\RNp} \times \R$ (or on a subset) is extended to a Radon measure on the entire space  $\R^{m + 1} \times \R$ via the zero extension. For instance, the measure $e(U_{\eps}) dXdt$ is  originally defined on $\overline{\RNp} \times (0, \infty)$. We treat  it as a Radon measure on $\R^{m + 1} \times \R$  by defining  the measure of any Borel set $A$ to be $$ \int_{A \cap \left( \overline{\RNp} \times (0, \infty) \right) } e(U_{\eps}) dX dt. $$

For any $R>0$ and $i \in \Zp$ sufficiently large, we have
\begin{align}\label{eq:measure growth rate}
\mu_i \left( P_R \right) =& r_i^{- (m + 1)} \mu \left( P_{R r_i} (Z_0) \right) \leq \liminf_{\eps \to 0} r_i^{- (m + 1)} \int_{\overline{P_{R r_i}^+ (Z_0)}} e \left( U_{\eps} \right) dX dt\\
\leq& C \liminf_{\eps \to 0} r_i^{- (m + 1)} \int_{t_0 - R^2 r_i^2}^{t_0 + R^2 r_i^2} \int_{ \overline{B_{R r_i}^+ (X_0)}} e^{- \frac{|X - X_0|^2}{4 (t_0 + 2 R^2 r_i^2 - t)}} e(U_{\eps}) dX dt\\
\leq& C R^{m + 1} \liminf_{\eps \to 0} \E_{\eps} (U_{\eps}, (X_0, t_0 + 2 R^2 r_i^2), R r_i) \\ \leq & C R^{m + 1},
\end{align}
thanks to \eqref{eq:scale_inv_bds}.
By a diagonal argument, we obtain a Radon measure $\mu_*$ on $\R^{m + 1} \times \R$ such that, up to a subsequence, $\mu_i \to \mu_*$ as Radon measures. Since $Z_0 \in \Sigma$,    Proposition \ref{density_results} yields  $\Theta^{m + 1} (\mu, Z_0) \geq \eps_0^2$. This fact,  combined with the monotonicity formula (Lemma \ref{lem-mono}) and Proposition \ref{energy_tail_est}, implies that the measure $\mu_*$ is not the zero measure, and in fact $(0, 0) \in \textnormal{spt}(\mu_*)$.

For any $\eps > 0$ and $r > 0$, we set $$U_{\eps, r} (X, t) := U_{\eps} (X_0 + r X, t_0 + r^2 t), \quad \forall (X, t) \in \RNp \times (- \frac{t_0}{r^2}, \infty).$$ 
Let $\{ \phi_j \}$ be a   countable dense subset of $C_c^{\infty} (\R^{m + 2})$ with respect to the $L^{\infty}$ norm. Then for every $j \in \Zp$,  weak* convergence implies  $$\lim_{i \to \infty} \int \phi_j \, d\mu_i = \int \phi_j \, d\mu_*.$$
Thus, for every $k \in \Zp$, we can  find an index  $i(k) \gg 1$  such that 
 $$\left| \int \phi_j \, d\mu_{i(k)} - \int \phi_j \, d\mu_* \right| < \frac{1}{2k},\quad \forall j \in \{ 1, \ldots, k \}.$$
Since for every $j, k \in \Zp$ we have $$\lim_{\eps \to 0} \int \phi_j e_{\eps/ \sqrt{r_{i(k)}} } (U_{\eps, r_{i(k)}}) \: dX dt = \int \phi_j d \mu_{i(k)},$$
 we can choose $\eps_k \searrow 0$ such that for every  $\eps \in (0, \eps_k]$   $$\left| \int \phi_j e_{\eps/ \sqrt{r_{i(k)}} } (U_{\eps, r_{i(k)}}) \: dX dt - \mu_{i(k)} (\phi_j) \right| < \frac{1}{2k},\quad  \forall j \in \{1, \ldots, k\}.$$
 Combining these two estimates via the triangle inequality, for any $j \in \{1, \dots, k\}$ and $\eps \in (0, \eps_k]$, we obtain
$$\left| \int \phi_j e_{\eps/ \sqrt{r_{i(k)}} } (U_{\eps, r_{i(k)}}) \: dX dt - \mu_* (\phi_j) \right| < \frac{1}{k}.$$
Since all the inequalities above are preserved upon decreasing $\eps_k$, we may lower the value of each $\eps_k > 0$ further to satisfy the asymptotic condition
 $$\lim_{k \to \infty} \frac{\eps_k}{\sqrt{r_{i(k)}}} = 0.$$ Now, let  $\{ \psi_j \}\subset C_c^{\infty} \left( \R^{m + 1} \times \R \right)$  be a   countable dense subset of $L^2 (\RNp \times \R; \R^{L \times (m + 2)})$.  By choosing $\eps_k$ even smaller if necessary, we additionally ensure that for every $j \in \{1, \dots, k\}$,
 $$\left| \int \nabla_X \left( U_{\eps_k} - U \right) : \psi_j \left( \frac{X - X_0}{r_{i(k)}}, \frac{t - t_0}{r_{i(k)}^2} \right) \right| < \frac{(r_{i(k)})^{m + 2}}{k},$$ $$\left| \int \partial_t \left( U_{\eps_k} - U \right) : \psi_j \left( \frac{X - X_0}{r_{i(k)}}, \frac{t - t_0}{r_{i(k)}^2} \right) \right| < \frac{(r_{i(k)})^{m + 1}}{k}.$$

Define $\tilde U_k := U_{\eps_k, r_{i(k)}}$. Then the above choices ensure that $\tilde U_k$ converges to a constant function weakly in $H^1_{loc} ( \overline{\RNp} \times \R; \R^L)$. Moreover, setting  $\tilde \eps_k := \frac{\eps_k}{\sqrt{r_{i(k)}}}$ we have   $$e_{\tilde \eps_k} (\tilde U_k) \: dX dt \overset{*}{\rightharpoonup} \mu_*.$$

Let   $\Sigma_* := \textnormal{spt}(\mu_*)$ denote the support of $\mu_*$. As established before,  $(0, 0) \in \Sigma_*$. In light of    Proposition \ref{configuration of Sigma},    $\Sigma_*$ is the energy concentration set for $(\tilde U_k)$. More precisely,  $\Sigma_*$ plays the same role for the sequence $(\tilde U_k)$ as $\Sigma$ plays for $(U_{\eps})$. Since $\mu_* \not\equiv 0$,   Proposition \ref{blowupcndn} implies  that $\p^{m + 1} (\Sigma_*) > 0$. For any $t \in \R$, we define the time slice of $\Sigma_*$ as follows
\begin{equation}\label{eq:time slice of singular set}
\Sigma_*^t := \{X \in \overline{\RNp} | \: (X, t) \in \Sigma_*\}.
\end{equation}
Let us also recall the Radon measure $\mu_*^t$ on $\overline{\RNp}$, obtained by applying Proposition \ref{time density of mu} to $\mu_*$. It follows from Lemma \ref{eps_gradient_est} and Proposition \ref{time density of mu} that $\spt( \mu_*^t) \subseteq \Sigma_*^t$.

\medskip 

\noindent \underline{\textbf{Claim 1:}} There exists $\overline{\delta} > 0$ such that $\mu_*^t (B_1) > 0 $ for all $t \in (- \overline{\delta}, 0)$. In particular $\Sigma_*^t \cap B_1 \neq \emptyset$ for all such $t$.

Suppose, for the sake of contradiction, that Claim 1 is false. Then there exists a strictly increasing sequence $t_n=-r_n^2 \nearrow 0$ such that 
$\mu_*^{- r_n^2} (B_1) = 0$. This, Proposition \ref{time density of mu}  and Proposition \ref{energy_tail_est} with $\delta=\frac{\eps_0^2}{4}$ would imply  that $$\lim_{k \to \infty} \mathcal{D}_{\tilde \eps_k} \left( \tilde U_k, (0, 0), r_n \right) \leq \frac{\eps_0^2}{2}, \quad \forall n \gg 1.$$
Consequently,  $(0, 0)$ can not be a point of energy concentration for $(\tilde U_k)_k$, which contradicts  $(0, 0) \in \Sigma_*$. The proof of Claim 1 follows immediately from $\spt(\mu_*^t) \subseteq \Sigma_*^t$.

\medskip 

\noindent\underline{\textbf{Claim 2:}} For any $R, T > 0$ there exists a constant $C = C(R, T)>0$ such that for each $\overline X \in B_R$ and $\overline t \in (- T, T)$ it holds that $$\mu_*^{\overline t} \left( \overline{B_r(\overline X)} \right) \leq C r^{m - 1}, \quad \text{for all } r < R.$$
In particular we have $$\mu_*^{ \overline t} \left( A \right) \leq C \H^{m - 1} \left( \Sigma_*^{\overline t} \cap A \right), \quad \forall \overline t \in (- T, T), \: \text{and any Borel set } A \subseteq B_R.$$

Let $(\overline X, \overline t) \in B_R \times (- T, T)$ and $r < R$. Applying Proposition \ref{time density of mu}, the estimate from Remark \ref{remark lem-mono} and \eqref{eq:scale_inv_bds} we have
\begin{align*}
\mu_*^{\overline t} (B_r (\overline X)) \leq& \liminf_{k \to \infty} \int_{B_r (\overline X)} e_{\tilde \eps_k} (\tilde U_k) (\cdot, \overline t) \, dX\\
=& r^{m - 1} \liminf_{k \to \infty} (r r_{i(k)})^{1 - m} \int_{B_{r r_{i(k)}} (X_0 + r_{i(k)} \overline X)} e(U_{\eps_k}) (\cdot, t_0 + r_{i(k)}^2 \overline t) \: dX\\
\leq& C r^{m - 1} \liminf_{k \to \infty} \mathcal{D}_{\eps_k} \left( U_{\eps_k}, (X_0 + r_{i(k)} \overline X, t_0 + r_{i(k)}^2 \overline t + r^2 r_{i(k)}^2, r r_{i(k)} \right)
\leq C r^{m - 1}.
\end{align*}
This proves the first part of the claim. The rest of the proof follows by standard arguments (see \cite[2.10.19]{F}).

\medskip 

Now appealing to Proposition \ref{time_singular_set} we have that for $\p^{m + 1}$-almost every point $\overline Z = (\overline X, \overline t) \in \Sigma_* \cap \left( B_1 \times (- \overline{\delta}, 0) \right)$,
\begin{equation}\label{zero time density equation}
    \limsup_{r \to 0} \liminf_{k \to \infty} \frac{1}{r^{m - 1}} \int_{P_r^+ (\overline Z)} |\partial_t \tilde U_k|^2 \: dX dt = 0.
\end{equation}
Recall from \eqref{defn_density} that $$\Theta^{m + 1} (\mu_*, \cdot) := \lim_{r \to 0} \lim_{k \to \infty} \E_{\tilde \eps_k} \left( \tilde U_k, \cdot, r \right).$$
Applying the last inequality in \eqref{eq:rescaledfunction} in the context of $(\tilde{U}_k)$, we get local boundedness of $\Theta^{m + 1} (\mu_*, \cdot)$. This in turn implies the local boundedness for the $(m + 1)$-dimensional upper density of $\mu_*$. The density results in \cite[2.10.19]{F} then shows that \eqref{zero time density equation} holds for $\mu_*$-almost every point in $\Sigma_* \cap \left( B_1 \times (- \overline{\delta}, 0) \right)$.
Combining this with Claim 1, Claim 2 and Proposition \ref{appx_cty_iff} we see that for $\mu_*$-almost every point $\overline Z = (\overline X, \overline t) \in \Sigma_* \cap \left( B_1 \times (- \overline{\delta}, 0) \right)$, the following properties hold:
\begin{itemize}
    \item[$(a)$] $\H^{m - 1} (\Sigma_*^{\overline t} \cap B_1) > 0$,
    \item[$(b)$] $\Theta^{m + 1} (\mu_*, \cdot)$ is $\mu_*$- approximately continuous at $\overline Z$, 
    \item[$(c)$] $\overline Z$ satisfies \eqref{zero time density equation}.
\end{itemize}
Let us recall from Proposition \ref{time density of mu} that $\mu_* = \mu_*^t \, dt$. So we can find $t_1 \in (- \overline{\delta}, 0)$ such that
\begin{equation}\label{eq:slice_measure_positive}
\mu_*^{t_1}(A)>0,\quad A:= \left \{ X \in \Sigma_*^{t_1} \cap B_1| \: (X, t_1) \text{ satisfies the properties $(a)\text{-}(c)$} \right\} .
\end{equation}
Lemma \ref{partial regularity lemma} shows us that $\Sigma_*^{t_1}$ has locally finite $\H^{m - 1}$ measure.
Combining this with Claim 2 we can conclude that $\mu_*^{t_1}$ is absolutely continuous with respect to the Radon measure $\H^{m - 1} \lfloor_{\Sigma_*^{t_1}}$. 
This fact together with Proposition \ref{appx_cty_iff} and  standard results in geometric measure theory yield that the following properties hold at $\mu_*^{t_1}$-almost every $X_1 \in A$:
\begin{itemize}
    \item[$(d)$] $\overline{\Theta}^{m - 1} (\mu_*^{t_1}, \cdot) $ is $\H^{m - 1}\lfloor_{\Sigma_*^{t_1}}$-approximately continuous at $X_1$, where $$\overline{\Theta}^{m - 1} (\mu_*^{t_1}, X) := \limsup_{r \to 0} r^{1 - m} \mu_*^{t_1} (\overline{B_r (X)}), \text{ for } X \in \Sigma_*^{t_1} \cap B_1,$$ 
    \item[$(e)$] $\overline{\Theta}^{m - 1} (\mu_*^{t_1}, X_1) > 0$, and
    \item[$(f)$] $\limsup_{r \searrow 0} r^{1 - m} \H^{m - 1} \left( \Sigma_*^{t_1} \cap \overline{B_r (X_1)} \right) < \infty$.
\end{itemize}
In particular, since  $\mu_*^{t_1} (A) > 0$,  we can choose a point $X_1 \in A$   satisfying  the properties $(d)$-$(f)$. We set $Z_1 := (X_1, t_1)$.

With this we now have the following claim which follows from Lemma \ref{geometric lemma}.
\medskip

\noindent\underline{\textbf{Claim 3:}} There exists a sequence of positive real numbers $s_j \searrow 0$ and $m - 1$ sequences of points $(X_1^j)_j, \ldots, (X_{m - 1}^j)_j$ in $\Sigma_*^{t_1}$ and a constant $\delta \in (0, \frac 12)$ such that
\begin{enumerate}
\item $\delta s_j \leq \left| X_k^j - X_1 \right| \leq s_j, \forall j \in \Zp, \forall k \in \{ 1, \ldots, m - 1\}$,
\item dist$\left( X_k^j - X_1, \text{span} \{ X_1^j - X_1, \ldots, X_{k - 1}^j - X_1 \} \right) \geq \delta s_j, \forall j \in \Zp, \: \forall k \in \{ 2, \ldots, m - 1\}$, 
\item $\overline{\Theta}^{m - 1} (\mu_*^{t_1}, (X_k^j, t_1)) \geq \overline{\Theta}^{m - 1} (\mu_*^{t_1}, (X_1, t_1)) + o_j(1), \, \forall j \in \Zp, \, \forall k \in \{ 1, \ldots, m - 1\}$.
\end{enumerate}
\medskip
The above mentioned constant $\delta$ will not be used extensively  in what follows. Hence we use a generic symbol for the constant in order to reduce some notational jargon.

From \eqref{zero time density equation}, for each $l \in \Zp $, there exists $\overline{r}_l > 0$ such that
\begin{equation}\label{eq: liminf-bound}
\liminf_{k \to \infty} \frac{1}{r^{m - 1}} \int_{P_r^+ (Z_1)} |\partial_t \tilde{U}_k|^2 \, dX dt < l^{-m}, \quad \forall r \in (0, \overline{r}_l].
\end{equation}
We choose a strictly decreasing sequence $(s_j)  $ from Claim 3 satisfying $j s_j < \overline{r}_j$ for all $j \in \Zp$. For $j, k \in \Zp$, set $$\mu_{*, j}(A) := s_j^{-(m+1)} \mu_*(Z_1 + D_{s_j}(A)) \quad \text{and} \quad V_{k, j}(X, t) := \tilde{U}_k(X_1 + s_j X, \, t_1 + s_j^2 t).$$
As before, we can show that the measures $\mu_{*, j}$ are locally uniformly bounded, and hence up to a subsequence, $\mu_{*, j} \overset{*}{\rightharpoonup} \mu_{**}$ for some Radon measure $\mu_{**}$. Therefore, by a diagonal process, we may choose subsequences $(k_l)$ and $(j_l)$ of $(k)$ and $(j)$ respectively, so that
\begin{itemize}
\item  $V_l := V_{k_l, j_l}$ converges to a constant function weakly in $W_{loc}^{1, 2} (\overline{\RNp} \times \R; \R^L)$,
\item   $\hat \eps_l := \frac{\tilde \eps_{k_l}}{\sqrt{s_{j_l}}}\to0$,
\item  $e_{\hat \eps_l} (V_l) \: dX dt \overset{*}{\rightharpoonup} \mu_{**}$
\item $l s_{j_l} < \overline{r}_l$.
\end{itemize}
Then by  \eqref{eq: liminf-bound}
\begin{equation}\label{eq: kl-choice}
\frac{1}{(l s_{j_l})^{m-1}} \int_{P_{l s_{j_l}}^+ (Z_1)} |\partial_t \tilde{U}_{k_l}|^2 \, dX dt < l^{-m},
\end{equation}
and in particular,
\begin{equation}\label{eq: time derivative cgs}
\lim_{l \to \infty} \partial_t V_l = 0 \quad \text{in } L^2_{\mathrm{loc}} \left( \overline{\RNp} \times \mathbb{R} \right).
\end{equation}

Let us denote the support of $\mu_{**}$ by $\Sigma_{**}$. It follows from the first and second conclusions of Claim 3 that there exists a set of $m-1$ linearly independent points $\{\xi_1, \ldots, \xi_{m - 1}\} \subset \overline{B_1 \backslash B_{\delta}}$ such that, up to a subsequence, 
\begin{equation}\label{eq:defofxi}
    \xi_k = \lim_{l \to \infty} \frac{X_k^{j_l} - X_1}{s_{j_l}}, \quad  \text{ for all } k \in \{ 1, \ldots, m - 1 \}.
\end{equation}

\medskip 

\noindent\underline{\textbf{Claim 4:}} For any $r > 0$ and $(Y, s) \in \Sigma_{**}$ we have $\int_{\overline{T_r^+ (s)}} \G_{Y, s} d\mu_{**} = \Theta^{m + 1} (\mu_*, Z_1)$. 

\medskip 

We fix  $r > 0$ and $(Y, s) \in \Sigma_{**}$. Then for any $\rho > 0$ we have
\begin{align*}
\int_{\overline{T_r^+ (s)}} \G_{Y, s} \, d\mu_{**} =& \lim_{j \to \infty} \int_{\overline{T_r^+ (s)}} \G_{Y, s} \, d\mu_{*, j}\\
=& \lim_{j \to \infty} s_j^{- (m + 1)} \int_{\overline{T_{r s_j}^+( t_1 + s_j^2 s)}} \G_{Y, s} \left( \frac{X - X_1}{s_j}, \frac{t - t_1}{s_j^2} \right) \, d\mu_* (X, t)\\
=& \lim_{j \to \infty} \int_{\overline{T_{r s_j}^+ ( t_1 + s_j^2 s)}} \G_{X_1 + s_j Y, t_1 + s_j^2 s} \, d\mu_*\\
\leq& \lim_{j \to \infty} \int_{\overline{T_{\rho}^+ ( t_1 + s_j^2 s)}} \G_{X_1 + s_j Y, t_1 + s_j^2 s} \, d\mu_*\\
=& \int_{\overline{ T_{\rho}^+ ( t_1 )}} \G_{X_1, t_1} \, d\mu_*,
\end{align*}
where we have used Lemma \ref{lem-mono} and the dominated convergence theorem. The arbitrariness of $\rho > 0$ gives us $$\int_{\overline{T_r^+ (s)}} \G_{Y, s} d\mu_{**} \leq \Theta^{m + 1} (\mu_*, Z_1).$$

In order to prove the reverse inequality, we fix  $\alpha > 0$ small. We first show that for any $R > 0$
\begin{equation}\label{eq:appx cty}
\lim_{j \to \infty} \mu_{*, j} \left( \left\{ (X, t) \in \overline{ \partial^0 P_R^+} | \: \Theta^{m + 1} (\mu_{*, j}, (X, t)) \leq \Theta^{m + 1} (\mu_*, Z_1) - \alpha \right\}\right) = 0.
\end{equation}
To see this, we notice that $\Theta^{m + 1} (\mu_{*, j}, (X, t)) = \Theta^{m + 1} (\mu_*, (X_1 + s_j X, t_1 + s_j^2 t))$. So we have 
\begin{align*}
&\mu_{*, j} \left( \{ (X, t) \in \overline{ \partial^0 P_R^+} | \: \Theta^{m + 1} (\mu_{*, j}, (X, t)) \leq \Theta^{m + 1} (\mu_*, Z_1) - \alpha \}\right)\\
&= \frac{ \mu_* \left( \{ (X, t) \in \overline{ \partial^0 P_{R s_j}^+ (Z_1)} | \: \Theta^{m + 1} (\mu_*, (X, t)) \leq \Theta^{m + 1} (\mu_*, Z_1) - \alpha \}\right) }{\mu_* \left( \overline{ P_{R s_j}^+ (Z_1)} \right)} \frac{\mu_* \left( \overline{ P_{R s_j}^+ (Z_1)} \right)}{(R s_j)^{m + 1}} R^{m + 1}\\&\xrightarrow{j\to\infty}0,
\end{align*}
since $\Theta^{m + 1} (\mu_*, \cdot)$ is   $\mu_*$-approximate continuous  at $Z_1$, and   the second factor stays bounded. The latter fact can be shown using arguments similar to the ones given in \eqref{eq:measure growth rate}. This proves \eqref{eq:appx cty}.

Let $(Y, s) \in \textnormal{spt} (\mu_{**})$, and let $R>0$ be such that  $(Y, s) \in P_{\frac{R}{2}}$. Using \eqref{eq:appx cty} we can  choose a sequence $\left( (\tilde Y_j, \tilde s_j ) \right)_j$ that converges to $(Y, s)$ and satisfies
$$\Theta^{m + 1} (\mu_{*, j}, (\tilde Y_j, \tilde s_j )) > \Theta^{m + 1} (\mu_*, Z_1) - \alpha.$$
Let us fix some $\rho \in (0, r)$, $M \in \Zp$ and $\delta > 0$. If $j \in \Zp$ is taken sufficiently large, then
\begin{align*}
\Theta^{m + 1}( \mu_*, Z_1) <& \Theta^{m + 1} \left( \mu_{*, j}, (\tilde Y_j, \tilde s_j) \right) + \alpha \leq \int_{\overline{T_{\rho}^+ (\tilde s_j)}} \G_{\tilde Y_j, \tilde s_j} d\mu_{*, j} + \alpha\\
\leq& \int_{s - 4 \rho^2 - \delta}^{s - \rho^2 + \delta}  \int_{\overline{B_M^+ (Y)}} \G_{\tilde Y_j, \tilde s_j} d\mu_{*, j} + \int_{s - 4 \rho^2 - \delta}^{s - \rho^2 + \delta} \int_{\overline{\RNp} \setminus \overline{B_M^+ (Y)}} \G_{\tilde Y_j, \tilde s_j} d\mu_{*, j} + \alpha.
\end{align*}
The second integral in the right hand side is bounded above by $\alpha$, provided we choose $M$ large enough (see e.g.  proof of Proposition \ref{energy_tail_est}). Therefore, we have
\begin{align*}
\Theta^{m + 1}( \mu_*, Z_1) <& \int_{s - 4 \rho^2 - \delta}^{s - \rho^2 + \delta} \int_{\overline{B_M^+ (Y)}} \G_{\tilde Y_j, \tilde s_j} d\mu_{*, j} + 2 \alpha\\
\leq& \int_{s - 4 \rho^2 - \delta}^{s - \rho^2 + \delta} \int_{\overline{B_M^+ (Y)}} \G_{Y, s} d\mu_{*, j} + 3 \alpha\\
\leq& \int_{s - 4 \rho^2 - \delta}^{s - \rho^2 + \delta} \int_{\overline{\RNp}} \G_{Y, s} d\mu_{**} + 4 \alpha\\
\leq& \int_{\overline{T_{\rho}^+ (Y, s)}} \G_{Y, s} d\mu_{**} + 4 \alpha \leq \int_{\overline{T_r^+ (Y, s)}} \G_{Y, s} d\mu_{**} + 4 \alpha,
\end{align*}
where the last inequality  follows from the arbitrariness of $\delta > 0$ and the monotonicity formula (Lemma \ref{lem-mono}).  Since $\alpha > 0$ was arbitrary, this completes the proof.

\medskip 

\noindent\underline{\textbf{Claim 5:}} We have $0, \xi_1, \ldots, \xi_{m - 1} \in \Sigma_{**}^0$, where $\xi_1, \ldots, \xi_{m - 1}$ are defined as in \eqref{eq:defofxi}, and    $\Sigma_{**}^{\overline t}  := \{ X \in \overline{\RNp} | \: (X, \overline t) \in \Sigma_{**} \}$ denotes the time slice for  $\overline t \in \R$.

\medskip 
 The proof   that $0 \in \Sigma_{**}^0$ is analogous to that of   $(0, 0) \in \Sigma_*$. For the  remaining  points, fix $k \in \{ 1, \ldots, m - 1\}$ and take any $\delta > 0$. Then
\begin{align*}
\mu_{**} \left( P_{2 \delta} (\xi_k, 0) \right) \geq& \limsup_{j \to \infty} \mu_{*, j} \left( P_{\delta} (\xi_k, 0) \right) = \limsup_{j \to \infty} s_j^{-(m +1)} \mu_* \left( P_{\delta s_j} ( X_1 + s_j \xi_k, t_1) \right)\\
\geq& \limsup_{j \to \infty} s_j^{-(m +1)} \mu_* \left( P_{\frac{\delta s_j}{2} } ( X_k^j, t_1) \right).
\end{align*}
Since $X_k^j \in \Sigma_*^{t_1}$ for every $j \in \Zp$, we get that   $\Theta^{m + 1} (\mu_*, (X_k^j, t_1)) \geq \eps_0^2$. So using Proposition \ref{energy_tail_est}, we can get some large enough $M \in \Zp$ such that for all large enough $j \in \Zp$, $0 < r < \delta s_j$ and with the help of Lemma \ref{lem-mono} we have
\begin{align*}
    \eps_0^2 \leq& \Theta^{m + 1}(\mu_*, (X^j_k, t_1)) \leq \int_{\overline{T_r^+ (t_1)}} \G_{X^j_k, t_1} \, d\mu_*
    \leq \frac{C}{r^{m + 1}} \mu_* \left( P_{2 M r} (X^j_k, t_1) \right) + \frac{\eps_0^2}{2}.
\end{align*}
Choosing $r = \frac{\delta s_j}{4 M}$, we obtain \begin{align}\mu_{**} \left( P_{\delta} (\xi_k, 0) \right) &\geq \limsup_{j \to \infty} s_j^{-(m +1)} \mu_* \left( P_{\frac{\delta s_j}{2} } ( X_k^j, t_1) \right)\\& = \limsup_{j \to \infty} s_j^{-(m +1)} \mu_* \left( P_{2 M \frac{\delta s_j}{4 M} } ( X_k^j, t_1) \right) \geq C \frac{\eps_0^2 \: \delta^{m + 1}}{M^{m + 1}}.\end{align} This completes the proof of Claim 5. 

\medskip

\noindent\underline{\textbf{Claim 6:}} For every  $k \in \{1, \ldots, m - 1\}$, we have $$\lim_{l \to \infty} (\xi_k \cdot \nabla_X) V_l = 0, \quad  \text{ in } L^2_{loc} \left( \overline{\RNp} \times (-\infty, 0) \right).$$

Suppose that $\tilde \xi = (\tilde X, \tilde t) \in \Sigma_{**}$ and $0 < R_1 < R_2$. Then using Claim 4 we know that $$\int_{\overline{T_{R_1}^+ (\tilde t)}} \G_{\tilde \xi} \: d\mu_{**} = \int_{\overline{T_{R_2}^+ (\tilde t)}} \G_{\tilde \xi} \: d\mu_{**}.$$ Using this, Lemma \ref{lem-mono} and Fatou's lemma, we get
\begin{align*}
0 =& \lim_{l \to \infty} \left[ \E_{\hat \eps_l} \left( V_l, \tilde \xi, R_2 \right) - \E_{\hat \eps_l} \left( V_l, \tilde \xi, R_1 \right) \right]\\
\geq& \lim_{l \to \infty} \int_{R_1}^{R_2} \int_{T_1^+} \frac{r}{2 |t|} \left| \partial_r \left( V_l (\tilde X + r X, \tilde t + r^2 t) \right) \right|^2 \G_0 \: dX dt \: dr\\
\geq& \int_{R_1}^{R_2} \liminf_{l \to \infty} \int_{T_1^+} \frac{r}{2 |t|} \left| \partial_r \left( V_l (\tilde X + r X, \tilde t + r^2 t) \right) \right|^2 \G_0 \: dX dt \: dr.
\end{align*}
Hence,  for almost every $r \in (0, \infty)$ we get that  $$ \liminf_{l \to \infty} \partial_r \left( V_l (\tilde X + r X, \tilde t + r^2 t) \right) = 0, \text{ in } L^2_{loc} \left( \overline{T_1^+} \right).$$ Evaluating the derivative in the integrand and using a change of variables it can be seen that $$\liminf_{l \to \infty} \int \left| \left( (X - \tilde X) \cdot \nabla_X \right) V_l + 2 (t - \tilde t) \partial_t V_l \right|^2 \: dX dt = 0,$$ where the integral is taken over any compact subset of $\overline{\RNp} \times (- \infty, 0)$. The above equality is true for any $\tilde \xi \in \Sigma_{**}$. In particular using Claim 5, we can take $\tilde \xi = (0, 0), (\xi_1, 0), \ldots, (\xi_{m - 1}, 0)$. Passing to a subsequence along $l$, we get that for every $k \in \{ 1, \ldots, m - 1\}$ we have $$ \left( (X - \xi_k) \cdot \nabla_X \right) V_l + 2t \partial_t V_l \to 0, \text{  in } L^2_{loc} \left( \overline{\RNp} \times (- \infty, 0) \right), \text{ and}$$
\begin{equation}\label{eq:monotonicity at origin}
\left( X \cdot \nabla_X \right) V_l + 2t \partial_t V_l \to 0, \text{  in } L^2_{loc} \left( \overline{\RNp} \times (- \infty, 0) \right).
\end{equation}
Now the claim follows from $$ \left| (\xi_k \cdot \nabla_X) V_l \right| \leq \left| \left( (X - \xi_k) \cdot \nabla_X \right) V_l + 2t \partial_t V_l \right| + \left| \left( X \cdot \nabla_X \right) V_l + 2t \partial_t V_l \right|.$$

\medskip 

\noindent\underline{\textbf{Claim 7:}} For any $r>0$ and  $A \subseteq \overline{\RNp} \times (-\infty, 0)$  we have $$r^{-(m + 1)} \mu_{**} \left( D_r (A) \right) = \mu_{**} (A).$$

\medskip
 It suffices  to show that for any test function  $\phi \in C_c^{\infty} \left( \R^{m + 1} \times (-\infty, 0) \right)$ $$\int_{- \infty}^0 \int_{\overline{\RNp}} \phi \, d\mu_{**} 
 =r^{-(m + 1)} \int_{- \infty}^0 \int_{\overline{\RNp}} \phi \left( \frac X r, \frac{t}{r^2} \right) \, d\mu_{**}.$$   Recalling that $e_{\hat \eps_l} (V_l) \: dX dt \overset{*}{\rightharpoonup} \mu_{**}$,   Proposition \ref{convergence of approximation},  and a  change of variables together imply, $$r^{-(m + 1)} \int_{- \infty}^0 \int_{\overline{\RNp}} \phi \left( \frac X r, \frac{t}{r^2} \right) \, d\mu_{**}=\lim_{l \to \infty} \frac 12 \int_{- \infty}^0 \int_{\RNp} \phi \, |\nabla_X V_{l, r}|^2 \, dX dt,$$ where $V_{l,r}(X, t) := V_l(r X, r^2 t)$. 
Using the fundamental theorem of calculus in $s$, integration by parts in space, and Einstein's summation convention, we obtain
\begin{align*}
    &r^{- (m + 1) } \int_{- \infty}^0 \int_{\overline{\RNp}} \phi \, d\mu_{**}(D_r(\cdot)) - \int_{- \infty}^0 \int_{\overline{\RNp}} \phi \, d\mu_{**}\\
    =& \frac 12 \lim_{l \to \infty} \int_{- \infty}^0 \int_{\RNp} \phi \, [|\nabla_X V_{l, r}|^2 - |\nabla_X V_{l, 1}|^2] \, dX dt \\
    =& \frac 12 \lim_{l \to \infty} \int_1^r \int_{- \infty}^0 \int_{\RNp} \phi \, \partial_s (|\nabla_X V_{l, s}|^2) \, dX dt \, ds\\
    =& \lim_{l \to \infty} \int_1^r \int_{- \infty}^0 \int_{\RNp} \phi \, (\partial_i V_{l, s}^j) \, (\partial_i \partial_s V_{l, s}^j) \, dX dt \, ds \\
    =& - \lim_{l \to \infty} \int_1^r \int_{- \infty}^0 \int_{\RNp} \phi \, (\Delta_X V_{l, s}^j) \, (\partial_s V_{l, s}^j) \, dX dt \, ds
    - \lim_{l \to \infty} \int_1^r \int_{- \infty}^0 \int_{\RNp} ( \partial_i \phi) \, (\partial_i V_{l, s}^j) \, (\partial_s V_{l, s}^j) \, dX dt \, ds\\
    &\quad - \lim_{l \to \infty} \int_1^r \int_{- \infty}^0 \int_{ \partial \RNp} \phi \, (\partial_y V_{l, s}^j) \, (\partial_s V_{l, s}^j) \, dx \, dt \, ds\\
    =& - \lim_{l \to \infty} \int_1^r \int_{- \infty}^0 \int_{\RNp} \phi \, (\partial_t V_{l, s}^j) \, (\partial_s V_{l, s}^j) \, dX dt \, ds - \lim_{l \to \infty} \int_1^r \int_{- \infty}^0 \int_{\RNp} ( \partial_i \phi) \, (\partial_i V_{l, s}^j) \, (\partial_s V_{l, s}^j) \, dX dt \, ds\\
    &\quad - \lim_{l \to \infty} \frac{1}{\hat{\eps}_l^2} \int_1^r \int_{- \infty}^0 \int_{ \partial \RNp} \phi \, \partial_s (F(V_{l, s})) \, dx \, dt \, ds \\ =&: I + II + III. 
\end{align*} Since $(V_l)$ is bounded in $H^1_{loc}$,  the convergence in \eqref{eq:monotonicity at origin} implies  $I=II=0$. To bound  $III$, we integrate in $s$ and apply Proposition \ref{convergence of approximation}
    $$III = \lim_{l \to \infty} \frac{1}{\hat{\eps}_l^2} \int_{- \infty}^0 \int_{ \partial \RNp} \phi \, F(V_{l, 1}) \, dx \, dt - \lim_{l \to \infty} \frac{1}{\hat{\eps}_l^2} \int_{- \infty}^0 \int_{ \partial \RNp} \phi \, F(V_{l, r}) \, dx \, dt = 0.$$
 This completes the proof of Claim 7.

\medskip

Without loss of generality, we assume (recall from    Claim 3  that     $\{\xi_1, \ldots, \xi_{m - 1}\}$ is linearly independent)
$$\textnormal{span} \left(\{ \xi_1, \ldots, \xi_{m - 1} \}\right) = \textnormal{span} \left( \{e_1, \ldots, e_{m - 1}\} \right).$$
Together with  Claim 6, this yields  $$\lim_{l \to \infty} \partial_k V_l = 0, \text{ in } L^2_{loc} \left( \overline{\RNp} \times (- \infty, 0) \right)\quad \text{for each } k \in \{ 1, \ldots, m - 1\}.$$ Let $\Sigma_{**}^- := \textnormal{spt}( \mu_{**}) \cap \left( \R^{m + 1} \times (- \infty, 0) \right)$ be the restricted singular set. The scale invariance of $\mu_{**}$ proved in Claim 7 implies that $\Sigma_{**}^-$ is self-similar $$D_r(\Sigma_{**}^-)=\Sigma_{**}^-, \quad\text{for every }r>0.$$Consequently, $\Sigma_{**}^-$ is completely parameterized by its time slice at $t = -1$ $$\Sigma_{**}^- = \left\{ ( {\zeta} \sqrt{ - t}, t) | \:  {\zeta}\in \Sigma_{**}^{- 1} \text{ and } t \in (- \infty, 0) \right\}.$$

\noindent Since $(0, 0)\in \Sigma_{**}$ is an  energy concentration point, Proposition \ref{density_results} and Claim 4 guarantees $$\eps_0^2 \leq \Theta^{m + 1} (\mu_{**}, (0,0)) = \int_{\overline{T_ 1^+}} \G_0 \, d\mu_{**},$$
which  in particular shows that $\Sigma_{**}^-$ is nonempty.

\medskip 

\noindent \underline{\textbf{Claim 8:}} We have $$\mu_{**} = \H^{m - 1}\lfloor_{span(\{ \xi_1, \ldots, \xi_{m - 1} \})} \times h \p^2 \lfloor_{\mathcal S}, \quad \text{ on } \R^{m - 1} \times (\R^2 \times (- \infty, 0)),$$ where $h$ satisfies $$\frac{\Theta^{m + 1} (\mu_*, Z_1)}{C} \leq h \leq C \Theta^{m + 1} (\mu_*, Z_1), \quad \text{ and } \mathcal{S} = \{ (\zeta_i \sqrt{ - t}, t) \in \R^2 \times (- \infty, 0) | \: i \in I \},$$ for some $\zeta_i$'s in $\R^2$ and some index set $I$ which is non-empty and at most countable. Moreover we have $\mathcal S \cap \left(\R^2 \times \{t = - 1\} \right)$ is locally finite.

\medskip

Using the local uniform boundedness of the rescaled energy $\E$ for $V_l$, provided by \eqref{eq:rescaledfunction}, and the   estimates on the time energy from Lemma \ref{local_energy_ineq1}, we see that $(V_l)$ is bounded in $H^1_{loc} \left( \overline{\RNp} \times \R \right)$. So, for any $\phi \in C_c^{\infty} \left( \R^{m + 1} \times (- \infty, 0) \right), \, l \in \Zp$ and $k \in \{1, \ldots, m - 1\}$ by integration by parts (and using Einstein's summation convention) we have
\begin{align*}
\int_{\R} \int_{\overline{\RNp}} \partial_k \phi \: e(V_l) dX dt &= \frac12 \int_{\R} \int_{\RNp} \partial_k \phi \: |\nabla_X V_l|^2\: dX dt + \frac{1}{\hat \eps_l^2} \int_{\R} \int_{\partial \RNp} \partial_k \phi \: F(V_l) \: dx \: dt\\
&= - \int_{\R} \int_{\RNp} \phi (\partial_i V_l^j) \: \partial_i (\partial_k V_l^j) \: dX dt + \frac{1}{\hat \eps_l^2} \int_{\R} \int_{\partial \RNp} \partial_k \phi \: F(V_l) \: dx \: dt\\
&= \int_{\R} \int_{\RNp} (\partial_i \phi) (\partial_i V_l^j) \: (\partial_k V_l^j) \: dX dt + \int_{\R} \int_{\RNp} \phi \: (\partial_{ii} V_l^j) \: (\partial_k V_l^j) \: dX dt\\
& + \int_{\R} \int_{\partial \RNp} \phi \: (\partial_y V_l^j) \: (\partial_k V_l^j) \: dx \: dt + \frac{1}{\hat \eps_l^2} \int_{\R} \int_{\partial \RNp} \partial_k \phi \: F(V_l) \: dx \: dt\\
&= \int_{\R} \int_{\RNp} ( \nabla_X \phi \cdot \nabla_X) V_l \cdot \partial_k V_l \: dX dt + \int_{\R} \int_{\RNp} \phi \partial_t V_l \cdot \partial_k V_l \: dX dt\\
 &\xrightarrow{l\to\infty} 0,  
\end{align*}
thanks to  Claim 6.
Thus,  for any $k \in \{ 1, \ldots, m - 1\}$, $$ \lim_{l \to \infty} \partial_k \left( e(V_l) dX dt \right) = 0,$$ as distributions on $ \R^{m + 1} \times (- \infty, 0)$. Proposition \ref{Propsition A.1} immediately implies the existence of a non-negative Radon measure $\nu_{**}$ on $\R^2 \times (- \infty, 0)$ such that $$ \mu_{**} = \H^{m - 1} \times \nu_{**},\quad  \text{ on } \R^{m - 1} \times ( \R^2 \times (- \infty, 0)).$$
This in particular gives us
\begin{equation}\label{eq:measure_decomposition}
    \textnormal{spt}(\mu_{**}) \cap (\R^{m + 1} \times (-\infty, 0)) = (\R^{m - 1} \times \textnormal{spt}(\nu_{**})) \cap (\R^{m + 1} \times (-\infty, 0)).
\end{equation}

Now in order to determine $\nu_{**}$ as mentioned in the statement of the above claim, we will first show that on the support of $\nu_{**}$ the 2-dimensional upper and lower densities of $\nu_{**}$ is bounded above and below by $C \Theta^{m + 1} (\mu_*, Z_1)$ and $\frac{\Theta^{m + 1} (\mu_*, Z_1)}{C}$ respectively.

Let  $\overline Z := (\overline X, \overline t) \in \textnormal{spt}(\mu_{**})$. Denoting  any $X \in \overline{\RNp}$   by $X = (x^{\prime}, x_m, y) \in \R^{m - 1} \times \R \times [0, \infty)$, we have   $\overline X = (\overline x^{\prime}, \overline x_m, 0)$. Then by  Proposition \ref{energy_tail_est}  there exists sufficiently  large  $M \in \Zp$ such that for any $r > 0$,    Claim 4  yields 
\begin{align*}
    \Theta^{m + 1}(\mu_*, Z_1) =& \int_{\overline{T_r^+ (\overline t)}} \G_{\overline X, \overline t} \: d\mu_{**} \leq \frac{C}{r^{m + 1}} \mu_{**} \left( P_{2 M r} (\overline Z) \right) + \frac{1}{2} \Theta^{m + 1} (\mu_*, Z_1).
\end{align*}
Combining this with $\eqref{eq:measure_decomposition}$ and replacing $r$ with $\frac{r}{2M}$ yields
\begin{align*}
    \Theta^{m + 1}(\mu_*, Z_1) \leq C \frac{\nu_{**} \left( B_{r}^2(\overline{x}_m, 0) \times (\overline t - r^2, \overline t + r^2) \right)}{r^2}.
\end{align*}
This gives us the required lower bound on the 2-dimensional lower density of $\nu_{**}$ on its support. For the upper bound,  we estimate 
\begin{align*}
\frac{\mu_{**} \left( P_r (\overline X, \overline t) \right) }{r^{m + 1}} &\leq \liminf_{l \to \infty} \frac{1}{r^{m + 1}} \int_{P_r (\overline X, \overline t)} e(V_l) dX dt\\
&\leq \liminf_{l \to \infty} \frac{1}{(r s_{j_l})^{m + 1}} \int_{P_{r s_{j_l}} (X_1 + s_{j_l} \overline X, t_1 + s_{j_l}^2 \overline t)} e(\tilde U_{k_l}) dX dt \leq C.
\end{align*}
Together with \eqref{eq:measure_decomposition}, this gives 
\begin{align*}
\frac{\nu_{**} \left( B_r(\overline x_m, 0) \times (\overline t - r^2, \overline t + r^2) \right) }{r^{2}} \leq& C \frac{\mu_{**} \left( P_r (\overline X, \overline t) \right) }{r^{m + 1}} \leq \frac{C}{\Theta^{m + 1}(\mu_*, Z_1)} \Theta^{m + 1}(\mu_*, Z_1)\\
\leq& C \Theta^{m + 1}(\mu_*, Z_1),
\end{align*}
which  is  the required upper bound.

Now let $\mathcal{S} := \textnormal{spt}(\nu_{**}) \cap \{t < 0\}$. Using the corresponding property for $\mu_{**}$ we have $$\mathcal{S} = \left\{ (\zeta \sqrt{ - t}, t) | \: (\zeta, - 1) \in \mathcal{S}, \, t \in (- \infty, 0) \right\}.$$
Now the density bounds and local finiteness of $\nu_{**}$ implies that $\{ \zeta \in \R^2 | \: (\zeta, - 1) \in \mathcal{S} \}$ is locally finite. This gives us the required structure for $\mathcal{S}$.

Finally, for any $(\overline x_m, 0, \overline t) \in \partial \R^2_+ \times (- \infty, 0)$, $$ \frac{\nu_{**} \left( B_r(\overline x_m, 0) \times (\overline t - r^2, \overline t + r^2) \right) }{r^{2}} \leq C$$ implies that $\nu_{**}$ is absolutely continuous with respect to $\p^2 \lfloor_{\mathcal{S}}$, whence we are done.  

\medskip 
The local finiteness of the slice $\mathcal{S} \cap \{ t = - 1\} = \{\zeta_i\}_{i \in I}$ implies that the parabolic curves $\gamma_i(t) := (\zeta_i \sqrt{-t}, t)$ constituting $\mathcal{S}$ are locally separated away from $t = 0$. Namely, for each $i \in I$ and $t < 0$, the curve $\gamma_i$ admits a space-time neighborhood in $\mathbb{R}^2 \times (-\infty, 0)$ disjoint from all other curves $\gamma_j$ ($j \neq i$).

We fix  $\tilde{\zeta} \in \partial \mathbb{R}^2_+$ such that $$\left\{ ( \tilde{\zeta} \sqrt{-t} , \, t) : t \in (-\infty, 0) \right\} \subseteq \mathcal{S} \quad \text{or, equivalently,} \quad \tilde{\zeta} \in \mathcal{S} \cap \{t = -1\}.$$
We then choose $- \infty < T_0 < T_1 < 0$ such that  $$K := \left\{ ( \tilde{\zeta} \sqrt{ - t}, t)| \: T_0 \leq t \leq T_1 \right\}\subseteq \overline{B_{\frac12}} \times [T_0, T_1].$$

For any $l \in \Zp$, define  $$\psi_l := | \partial_t V_l|^2 + \sum_{k = 1}^{m - 1} | \partial_k V_l|^2.$$
Then by \eqref{eq: time derivative cgs} and Claim 6,  $\psi_l \to 0$ in $L^1_{loc} \left( \overline{\RNp} \times (-\infty, 0) \right)$ as $l \to \infty$. For any $$\bar Z=(\overline X, \overline t) = (\overline x^{\prime}, (\overline x_m, 0), \overline t) \in \R^{m - 1} \times \partial \R^2_+ \times (-\infty, 0) = \partial \RNp \times (-\infty, 0),$$ we set $$\tilde P_r^+ (\bar Z) := \left\{ (x^{\prime}, (x_m, y), t) \in \R^{m - 1} \times \R^2_+ \times (-\infty, 0) | \: (x_m, y) \in B_{r_0}^+ (\overline x_m, 0), \, (x^{\prime}, t) \in P_r (\overline x^{\prime}, \overline t) \right\},$$
where $r_0 \in (0, \frac{1}{2} \sqrt{|T_1|})$ is small enough so that $P_{ 2 r_0} (K) := \bigcup_{Z \in K} P_{2 r_0} (Z)$ does not contain any other \textit{parabolic tentacle} of $\mathcal{S}$ (see the discussion following Claim 8).
Next, for  $l \in \Zp$,  define the localized   maximal function  
\begin{align*}
&M_l (\overline Z) := \sup_{r \in (0, r_0)} r^{- (m + 1)} \int_{\tilde P_r^+ (\overline Z)} \psi_l \: dX dt.
\end{align*}
Consider the compact set $\tilde K := \overline{B_{r_0}} \times K \subseteq \R^{m - 1} \times \partial \R^2_+ \times (- \infty, 0) = \partial \RNp \times (- \infty, 0)$. Standard   weak $L^1$ type estimates for   maximal function yield, for any  $\alpha > 0$, $$\p^{m + 1} \left( \left\{\overline Z \in \tilde K| \: M_l (\overline Z) > \alpha \right\} \right) \leq \frac{C}{\alpha} \int_{ P_{r_0}^+ (\tilde K)} \psi_l \: dX dt.$$
Since $\psi_l \to 0$ in $L^1(P_{r_0}^+(\tilde{K}))$, thanks to  \eqref{eq: time derivative cgs} and Claim 6,   passing to a suitable subsequence  (which we still denote by $(l)$),  we can find  $\tilde Z = (\tilde x^{\prime}, \tilde\zeta \sqrt{- \tilde t}, \tilde t)\in \tilde K$ satisfying 
\begin{equation}\label{eq: maximal function estimate}
\lim_{l \to \infty} M_l (\tilde Z) = 0.
\end{equation}

By Claim 4 and Proposition \ref{energy_tail_est} we get some large enough $M \in \Zp$ such that for any $r > 0$,
\begin{align*}
    \eps_0^2 &\leq \Theta^{m + 1} (\mu_*, Z_1) = \int_{\overline{T_r^+ (\tilde Z)}} \G_{\tilde Z} \: d\mu_{**}\\
    \leq& C \liminf_{l \to \infty} \frac{1}{r^{m + 1}} \int_{B_{Mr} (\tilde x^{\prime}) \times \overline{B_{Mr}^+ (\tilde\zeta\sqrt{- \tilde t})} \times (\tilde t - (M r)^2, \tilde t)} e(V_l) dX dt + \frac{\eps_0^2}{2}.
\end{align*}
Replacing $r$ with $\frac{r}{M}$ and absorbing the constant $M^{m+1}$ into the constant $C$, we obtain for every $r > 0$
\begin{equation}\label{eq: energy concentration}
\frac{\eps_0^2}{2} \leq C_0 \liminf_{l \to \infty} \frac{1}{r^{m + 1}} \int_{B_r (\tilde x^{\prime}) \times \overline{B_r^+ (\tilde\zeta\sqrt{- \tilde t})} \times (\tilde t - r^2, \tilde t)} e(V_l) dX dt.
\end{equation}

Fix a constant  $C_1 > \max \{ 2, 2 C_0\}$, whose precise value will be specified later.
For each $l \in \Zp$ and $\lambda > 0$ define $$Q(l, \lambda) := \max_{Y \in \overline{\partial^0 B_{r_0}^+ (\tilde\zeta \sqrt{- \tilde t})}} \frac{1}{\lambda^{m + 1}} \int_{B_{\lambda}(\tilde x^{\prime}) \times \overline{B_{\lambda}^+ (Y)} \times (\tilde t - \lambda^2, \tilde t)} e(V_l) dX dt.$$
It follows from \eqref{eq: energy concentration} that for every $\lambda > 0$  $$\liminf_{l \to \infty} Q(l, \lambda) > \frac{\eps_0^2}{C_1}.$$ On the other hand, the smoothness of $V_l$ implies that for each fixed $l \in \Zp$, $$\lim_{\lambda \to 0} Q(l, \lambda) = 0.$$
By a standard intermediate value argument, we can extract a subsequence of $(l)$ (still denoted by $(l)$) and a strictly decreasing sequence of radii $\lambda_l \downarrow 0$ such that, for every $l \in \Zp$,$$Q(l, \lambda_l) = \frac{\eps_0^2}{C_1}.$$ Moreover,  by compactness of  $\overline{ \partial^0 B_{r_0}^+ (\tilde{\zeta} \sqrt{- \tilde t})}$ and    smoothness of   $V_l$, the maximum  $Q(l, \lambda_l)$ is attained at some point $Y_l \in \overline{ \partial^0 B_{r_0}^+ (\tilde{\zeta} \sqrt{- \tilde t})}$
\begin{equation}\label{energy concentration equation}
\frac{\eps_0^2}{C_1} = Q(l, \lambda_l) = \frac{1}{\lambda_l^{m + 1}} \int_{B_{\lambda_l}(\tilde x^{\prime}) \times \overline{B_{\lambda_l}^+ (Y_l)} \times (\tilde t - \lambda_l^2, \tilde t)} e(V_l) dX dt.
\end{equation}
Furthermore,  $Y_l \to$ $\tilde\zeta \sqrt{- \tilde t}$ as $l \to \infty$. Indeed, if   $Y_l \to \overline Y \in \overline{\partial^0 B_{r_0}^+ (\tilde\zeta \sqrt{- \tilde t})}$ along a subsequence, then one can show  that $(\tilde x^{\prime}, \overline Y, \tilde t) \in \Sigma_{**}$, whence $(\overline Y, \tilde t) \in \mathcal{S}$. But then the choice of $r_0$ forces $\overline Y = \tilde\zeta \sqrt{- \tilde t}$.

We are now ready to define the aforementioned rescaling of $V_l$ as follows: $$\overline V_l (x^{\prime}, Y, t) := V_l (\tilde x^{\prime} + \lambda_l x^{\prime}, Y_l + \lambda_l Y, \tilde t + \lambda_l^2 t), \, \forall (x^{\prime}, Y) \in \lambda_l^{- 1} \left( B_{r_0} \times \overline{B_{r_0}^+} \right), \forall t \in \lambda_l^{- 2}( -r_0^2, r_0^2).$$ Notice that each $\overline V_l$ satisfies parabolic system in the interior of its domain and satisfies a lateral boundary condition  analogous to  \eqref{eq:approx_ext_+ general target}  with the rescaled parameter  $\overline \eps_l := \frac{\hat \eps_l}{\sqrt{\lambda_l}}$.
Going back to  \eqref{eq: maximal function estimate}, a change of variables yields $$\lim_{l \to \infty} \sup_{r \in (0, \frac{r_0}{\lambda_l})} r^{-(m + 1)} \int_{- r^2}^{r^2} \int_{B_r \times B_{\frac{r_0}{2 \lambda_l}}^+} \left( | \partial_t \overline V_l|^2 + \sum_{k = 1}^{ m - 1} | \partial_k \overline V_l |^2 \right) \: dX dt = 0.$$
In particular,   for every  $k \in \{ 1, \ldots, m - 1\}$ we have  \begin{align}\label{bar-Vl}\partial_k \overline V_l \to 0\quad\text{and }\partial_t \overline V_l \to 0\quad\text{in }L^2_{loc} \left( \overline{\RNp} \times \R \right)\quad\text{as }l\to\infty.\end{align}

We now establish local strong convergence for the sequence of  rescaled maps $(\overline V_l)$  in $\overline{\RNp} \times \R$. While  strong convergence near the origin follows directly from our choice of rescaling, extending this convergence to arbitrary compact subsets  of $ \overline{\RNp} \times \mathbb{R}$  requires suitable control on the energy. 
To this end, for any $l \in \Zp$ define $$\Psi_l (\overline x^{\prime}, \overline Y, \overline t) := \int_{\overline{\RNp} \times \R} e(\overline V_l) (\overline x^{\prime} + x^{\prime}, \overline Y + Y, \overline t + t) \: \phi(x^{\prime}, Y, t) \: d x^{\prime} \: dY dt,$$
where $\phi \in C_c^{\infty} (\R^{m + 1} \times \R)$ satisfies $\phi \equiv 1$ in $B_{\frac12}^{m - 1} \times B_{\frac12}^2 \times (- \frac34, - \frac14) \subseteq \R^{m - 1} \times \R^2 \times \R$, $\textnormal{spt}(\phi) \subseteq B_1^{m - 1} \times B_1^2 \times (- 1, 0)$ and $0 \leq \phi \leq 1$. Here,  $\Psi_l$ is defined for all  values of $(\overline x^{\prime}, \overline Y, \overline t)$ for which the defining expression of $\Psi_l$ makes sense. As $l \to \infty$, these domains form an increasing sequence that exhausts  $\overline{\RNp} \times \R$.

For any $k \in \{ 1, \ldots, m - 1\}$, differentiating $\Psi_l$ with respect to the $k$-th component of $\overline x^{\prime}$,   integrating by parts, and invoking both the equation satisfied by $\overline{V}_l$ and Einstein's summation convention, we obtain
\begin{align*}
\partial_k \Psi_l (\overline x^{\prime}, \overline Y, \overline t) =& \int_{ \overline{\RNp} \times \R} (\partial_i \overline V_l^j) (\partial_i \partial_k \overline V_l^j) (\overline x^{\prime} + x^{\prime}, \overline Y + Y, \overline t + t) \: \phi(x^{\prime}, Y, t) \: dx^{\prime} dY dt\\
&+ \frac{1}{\overline \eps_l^2} \int_{\partial \RNp \times \R} \partial_k (F(\overline V_l)) (\overline x^{\prime} + x^{\prime}, \overline Y + Y, \overline t + t) \: \phi(x^{\prime}, Y, t) \: dx^{\prime} dx_m dt\\
=& - \int_{ \overline{\RNp} \times \R} (\partial_{ii} \overline V_l^j) (\partial_k \overline V_l^j) (\overline x^{\prime} + x^{\prime}, \overline Y + Y, \overline t + t) \: \phi(x^{\prime}, Y, t) \: dx^{\prime} dY dt\\
&- \int_{ \overline{\RNp} \times \R} (\partial_i \overline V_l^j) (\partial_k \overline V_l^j) (\overline x^{\prime} + x^{\prime}, \overline Y + Y, \overline t + t) \: (\partial_i \phi) (x^{\prime}, Y, t) \: dx^{\prime} dY dt\\
&- \int_{\partial \RNp \times \R} (\partial_y \overline V_l^j) (\partial_k \overline V_l^j) (\overline x^{\prime} + x^{\prime}, \overline Y + Y, \overline t + t) \: \phi(x^{\prime}, Y, t) \: dx^{\prime} dx_m dt\\
&+ \frac{1}{\overline \eps_l^2} \int_{\partial \RNp \times \R} \partial_k (F(\overline V_l)) (\overline x^{\prime} + x^{\prime}, \overline Y + Y, \overline t + t) \: \phi(x^{\prime}, Y, t) \: dx^{\prime} dx_m dt\\
=& - \int_{ \overline{\RNp} \times \R} (\partial_t \overline V_l) \cdot (\partial_k \overline V_l) (\overline x^{\prime} + x^{\prime}, \overline Y + Y, \overline t + t) \: \phi(x^{\prime}, Y, t) \: dx^{\prime} dY dt\\
&- \int_{ \overline{\RNp} \times \R} (\partial_i \overline V_l) \cdot (\partial_k \overline V_l) (\overline x^{\prime} + x^{\prime}, \overline Y + Y, \overline t + t) \: (\partial_i \phi) (x^{\prime}, Y, t) \: dx^{\prime} dY dt,
\end{align*} and in a  similar way,  \begin{align*}
\partial_{\overline t} \Psi_l (\overline x^{\prime}, \overline Y, \overline t) =& - \int_{ \overline{\RNp} \times \R} |\partial_t \overline V_l|^2 (\overline x^{\prime} + x^{\prime}, \overline Y + Y, \overline t + t) \: \phi(x^{\prime}, Y, t) \: dx^{\prime} dY dt\\
&- \int_{ \overline{\RNp} \times \R} (\partial_i \overline V_l) \cdot (\partial_t \overline V_l) (\overline x^{\prime} + x^{\prime}, \overline Y + Y, \overline t + t) \: (\partial_i \phi) (x^{\prime}, Y, t) \: dx^{\prime} dY dt. 
\end{align*} It then follows from \eqref{bar-Vl} that  $$\partial_k \Psi_l\to0\quad\text{and }\partial_t \Psi_l\to0\quad\text{in }L^{\infty}_{loc} \left( \overline{\RNp} \times \R \right)\quad\text{as }l\to\infty.$$

Now take any $( \overline X, \overline t) = (\overline x^{\prime}, \overline Y, \overline t) \in \R^{m - 1} \times \partial \R^2_+ \times \R = \partial \RNp \times \R$. Due to our choice of $Y_l$ we know that for all $l \in \Zp$ large enough and for all $Y \in \overline{ \partial^0 B_2^+ (\overline Y)}$, we have $$\Psi_l (0, Y, 0) \leq \frac{\eps_0^2}{C_1}.$$
So using the above derivative bounds for $\Psi_l$ we see that when $l$ is sufficiently large then $$\Psi_l \leq \frac{2 \eps_0^2}{C_1}, \quad \text{ in } \overline{P_2^+ (\overline X, \overline t)}.$$
Together with a covering argument this would imply  that for some dimensional constant $C > 0$, $$\int_{ \overline{ P_2^+ (\overline X, \overline t)}} e(\overline V_l) dX dt \leq \frac{2 C \eps_0^2}{C_1}.$$
Now following Proposition \ref{energy_tail_est}, we can choose $\overline r \in (0, \frac12)$ such that the following estimate holds for all $l \in \Zp$
\begin{align*}
\E_{\overline \eps_l} \left( \overline V_l, (\overline X, \overline t), \overline r \right) < C \int_{ \overline{ P_1^+ (\overline X, \overline t)}} e(\overline V_l) \:dX dt + \frac{\eps_0^2}{2}.
\end{align*}
So choosing $C_1$ large enough, we get
\begin{align*}
\E_{\overline \eps_l} \left( \overline V_l, (\overline X, \overline t), \overline r \right) < C \int_{ \overline{ P_1^+ (\overline X, \overline t)}} e(\overline V_l) dX dt + \frac{\eps_0^2}{2} \leq \frac{C}{C_1} + \frac{\eps_0^2}{2} < \eps_0^2.
\end{align*}
Using this with Lemma \ref{eps_gradient_est} and interior estimate for heat equation we get that up to a subsequence $(\overline V_l)$ converges to some function $\overline V$ in $ C^1_{loc} \left( \overline{\RNp} \times \R \right)$.
Now using \eqref{energy concentration equation}, we get
\begin{equation}\label{eq:energyconc}
    \int_{- 1}^0 \int_{B_1 \times \overline{ B_1^+}} e(\overline V_l) dX dt = \frac{\eps_0^2}{C_1}.
\end{equation}
Since  $\partial_k V_l\to0$ and $\partial_t V_l\to0$   in $L^2_{loc} \left( \overline{\RNp} \times \R \right)$ for all $k \in \{ 1, \ldots, m - 1 \}$,  it follows that  $\partial_k \overline V \equiv 0 \equiv \partial_t \overline V$ for all $k \in \{ 1, \ldots, m - 1\}$.
Since $\overline{V}_l$ satisfies the heat equation, passing to the limit as $l \to \infty$ implies that $\overline{V}$ is a harmonic function.

Suppose that up to a subsequence, $c := \lim \overline{\eps}_l \in [0, \infty]$. Then for every $r > 0$, 
\begin{equation}\label{eq:integrability}
\frac{1}{r^{m - 1}} E^c (\overline V; B_r^+) \leq \lim_{l \to \infty} \frac{1}{r^{m - 1}} E^{\overline{\eps}_l} (\overline{V}_l (\cdot, 0); B_r^+) \leq C \liminf_{l \to \infty} \mathcal {D}_{\overline{\eps}_l} \left( \overline V_l, (0, r^2), r \right) \leq C,
\end{equation}
where the constant $C$ is independent of $r$ and $E^c$ denotes the standard Dirichlet energy when $c = 0$.
Since $\overline V$ is independent of the first $m - 1$ spatial variables, it can be viewed purely as a function from $\overline{\R^2_+}$ into $\R^L$. In light of this, \eqref{eq:integrability} shows  $$ \int_{B_r^+} |\nabla \overline V|^2 \, dX + \frac{1}{c} \int_{\partial^0 B_r^+} F(\overline V) \, dx \leq C,\quad\text{for every }r>0. $$
The rest of the proof now follows by the reasoning identical to that of Claim 3 in the proof of Theorem \ref{main thm_intro}, showing that $\overline V$ is the desired non-trivial harmonic map with free boundary on $N$. We conclude the proof.
\end{proof}

\medskip 

In Theorem \ref{no harmonic S1_intro}, we have shown that non-constant harmonic maps from $\mathbb{R}^2_+$ to $\mathbb{R}^L$ with free boundary on $N$ are the primary obstruction to the strong convergence of spatial derivatives, yielding a non-zero defect measure $\nu$. Combined with Proposition \ref{blowupcndn}, this result explains the measure-theoretic size of the concentration set $\Sigma$, specifically $\mathcal{P}^{m+1}(\Sigma) > 0$. In fact, the existence of such non-constant free-boundary harmonic maps leads to a lower bound on the parabolic Hausdorff dimension of the singular set, namely $\mathcal{P}_{\mathrm{dim}}(\Sigma) \ge m - 1$. Theorem \ref{fine reg parabolic blow up set} addresses this dimensional result via   an approach inspired by \cite[Theorem 4.3]{LW-1}.

We now prove Theorem \ref{fine reg parabolic blow up set}.

\begin{proof}[Proof of Theorem \ref{fine reg parabolic blow up set}]  
    By  Theorem \ref{no harmonic S1_intro},  the proof is immediate if   $\nu \not\equiv 0$. Consequently,  in what follows, we  assume  that $\nu \equiv 0$. 
    
    We now assume that  $\p_{\textnormal{dim}} (\Sigma) > m - 1$. Then there exists $s > m - 1$ such that $\p^s (\Sigma) > 0$ (not necessarily finite). We choose a compact set $\tilde \Sigma \subseteq \Sigma$ such that $\p^s (\tilde \Sigma) \in (0, \infty)$ (see e.g. \cite[Theorem 8.19]{Mattila}). Before going to the blow up argument first consider the function $$Z \mapsto \sup_{r \in (0, 1)} \frac{\p^s (\tilde \Sigma \cap \overline{P_r(Z)})}{r^s} \in [0, \infty].$$
    Notice that $\p^s (\tilde \Sigma) \in (0, \infty)$ implies $$\limsup_{r \searrow 0} \frac{\p^s(\tilde \Sigma \cap \overline{P_r(Z)})}{r^s} < \infty, \text{ for } \p^s \text{-almost every } Z \in \tilde \Sigma,$$
    whence we get that $$\sup_{r \in (0, 1)} \frac{\p^s (\tilde \Sigma \cap \overline{P_r(Z)})}{r^s} < \infty, \text{ for } \p^s \text{-almost every } Z \in \tilde \Sigma.$$
    So for any $M \in \Zp$, if we consider $$A_M := \{ Z \in \tilde \Sigma | \: \sup_{r \in (0, 1)} \frac{\p^s (\tilde \Sigma \cap \overline{P_r(Z)})}{r^s} \leq M \},$$
    then $$\p^s \left( \tilde \Sigma \backslash \bigcup_{M \in \Zp} A_M \right) = 0.$$
    In particular we can choose $M \gg 1$ such that for $A := A_M$ we have $\p^s (A) \in (0, \infty)$. We can now choose $Z_0 = (X_0, t_0) \in A$ and a sequence of positive real numbers $r_i \searrow 0$ such that
    $$\limsup_{r \searrow 0} \frac{\p^s \left( A \cap \overline{P_r (Z_0)} \cap \{ t \leq t_0 \} \right) }{r^s} = \lim_{i \to \infty} \frac{\p^s \left( A \cap \overline{P_{r_i} (Z_0)} \cap \{ t \leq t_0 \} \right) }{r_i^s} \in (0, \infty).$$
    
    \noindent Before proceeding further let us recall that $U$ is the weak limit of $(U_{\eps})$ in $H^1_{loc} \left( \overline{\RNp} \times (0, \infty) \right)$. 
    Using Lemma \ref{eps_gradient_est} it also follows that $U \in C^{\infty} \left( \overline{\RNp} \times (0, \infty) \backslash \Sigma \right)$.
    
    Now for each $i \in \Zp$, let us define 
    $$U_i (X, t) := U(X_0 + r_i X, t_0 + r_i^2 t), \: \forall (X, t) \in \overline{\RNp} \times (- \frac{t_0}{r_i^2}, \infty).$$
    Note that $\Sigma^{(i)} := sing(U_i) = D_{1/r_i} (\Sigma - Z_0)$ (see \eqref{eq:parabolic dilation} for the definition of the parabolic dilation $D$).
    Observe that each $U_i$ can be realized as a weak limit of smooth solutions to \eqref{eq:approx_ext_+ general target} satisfying \eqref{eq:scale_inv_bds}, namely the sequence $(U_{\eps, i})_{\eps}$, given by $$U_{\eps, i} (X, t) := U_{\eps} (X_0 + r_i X, t_0 + r_i^2 t ).$$
    Therefore, an application of Theorem \ref{no harmonic S1_intro} to the sequence $(U_{\eps, i})_{\eps}$ either produces the desired bubble, or yields the following convergence $$\nabla_X U_{\eps,i} \to \nabla_X U_i, \quad \text{strongly in } L^2_{loc} \left( \overline{\RNp} \times \R; \R^L \right).$$
    As we are done in the former case, we shall assume the latter case holds and proceed.
    Therefore,  we can choose subsequences $(\eps_k)$ and $(i_k)$ such that $$ \| U_{\eps_k, i_k} - U_{i_k} \|_{L^2 (P_k^+)} + \| \nabla_X U_{\eps_k, i_k} - \nabla_X U_{i_k} \|_{L^2 (P_k^+)} < \frac{1}{k}.$$
   Lemma \ref{local_energy_ineq1} and   \eqref{eq:rescaledfunction} imply   that $(U_{\eps_k, i_k})_k$ is bounded in $H^1_{loc} \left( \overline{\RNp} \times \R \right)$. Hence up to a subsequence $U_{\eps_k, i_k} \rightharpoonup V_0$, weakly in $H^1_{loc} \left( \overline{\RNp} \times \R \right)$, for some function $V_0$.
    As before, Theorem \ref{no harmonic S1_intro} applied to the sequence $(U_{\eps_k, i_k})_k$ either gives us the desired bubble, or aided with the above mentioned properties, it implies (which is the case that we shall assume)
    \begin{equation}\label{eq:strong grad cgs}
    \nabla_X U_{\eps_k, i_k} \to \nabla_X V_0, \quad \text{strongly in } L^2_{loc} \left( \overline{\RNp} \times \R; \R^L \right),
    \end{equation}
    and hence up to a subsequence $$ \nabla_X U_{i_k} \to \nabla_X V_0, \quad \text{strongly in } L^2_{loc} \left( \overline{\RNp} \times \R; \R^L \right). $$ Without any loss of generality,  we denote all the above subsequences by $(i)$.

    For any $(\overline X, \overline t) \in \partial \RNp \times \R$ and $r > 0$ we define    $$\E (V_0, (\overline X, \overline t), r) := \int_{\overline t - 4 r^2}^{\overline t - r^2} \int_{\RNp} \frac{1}{2} |\nabla_X V_0|^2 \, \G_{\overline X, \overline t} \: dX dt.$$
    We note that the strong convergence in \eqref{eq:strong grad cgs} implies $$ \E (V_0, (\overline X, \overline t), r) = \lim_{k \to \infty} \E_{\eps_k} (U_{\eps_k, i_k}, (\overline X, \overline t), r).$$
    Hence passing the limit $k \to \infty$ in the monotonicity expression given in Lemma \ref{lem-mono} for $U_{\eps_k, i_k}$, it follows that $V_0$ satisfies the following monotonicity inequality: for any $\overline{Z} = (\overline{X}, \overline{t}) \in \partial \RNp \times \R$ and any $0 < r < R < \infty$, 
    \begin{align}\label{eq:rescaled monotonicity}
    \rule{.7cm}{0cm}\E (V_0, \overline Z, R) - \E (V_0, \overline Z, r) \geq \int_r^R \int_{T_s^+(\overline t)} \frac{1}{s} \frac{ \left| \left( (X - \overline X) \cdot \nabla_X \right) V_0 + 2 (t - \overline t) \partial_t V_0 \right|^2}{2 (\overline t - t)} \, \G_{\overline Z} \, dX dt ds.
    \end{align}
    
    Next set $\Sigma_0 := sing(V_0)$. Using small energy regularity (Lemma \ref{eps_gradient_est}) it can be seen that if a sequence $(Z_i)$ converges to $\overline{Z}$ with $Z_i \in \Sigma^{(i)}$ for each $i \in \Zp$, then $\overline{Z} \in \Sigma_0$. This shows that for any $\delta > 0$ there exists $i_0 \in \Zp$ such that for all $i \geq i_0$, $\Sigma^{(i)} \cap \overline{P_1 (0)} \subseteq P_{\delta} ( \Sigma_0)$.
    
    For any $ t \in \R$, define $\Sigma_0^{t} := \{ X \in \overline{\RNp}| \: (X, t) \in \Sigma_0 \}$ and $\Sigma_0^- := \Sigma_0 \cap ( \overline{\RNp} \times (- \infty, 0))$. We have that following claim:  

    \medskip 
    
\noindent\underline{\textbf{Claim 1:}} $\p^s \left( \overline{\Sigma_0^-} \cap \overline{P_2 (0)} \right) > 0$.

\medskip 

Assume by contradiction that  $\p^s \left( \overline{\Sigma_0^-} \cap \overline{P_2 (0)} \right)=0$. Then for any small $\alpha > 0$ and $\delta \in (0, 1)$ we have $\p^s_{\delta} \left( \overline{\Sigma_0^-} \cap \overline{P_2(0)} \right) = 0$. By the compactness of $   \overline{\Sigma_0^-} \cap \overline{P_2 (0)}  $,     there exist a collection of finitely many parabolic cylinders  $\{P_{\rho_j} (Z_j) \}_j$ with $\rho_j \leq \delta$ and $Z_j \in \overline{\Sigma_0^-} \cap \overline{P_2 (0)}$ such that $$\overline{\Sigma_0^-} \cap \overline{P_2 (0)} \subseteq \bigcup_j \overline{P_{\rho_j} (Z_j)}, \quad \text{ and } \sum_j \rho_j^s < \alpha.$$
Set  $\rho := \min \{ \rho_j \}_j > 0$. Notice that  $P_{\rho} \left( \overline{\Sigma_0^-} \cap \overline{P_2 (0)} \right) \subseteq \cup_j \overline{P_{2 \rho_j} (Z_j)}$, and hence there exists $i_0 \in \Zp$ such that for all $i \geq i_0$ we have $$ D_{1/r_i} (A - Z_0) \cap \{ t \leq 0 \} \cap \overline{P_1 (0)} \subseteq \Sigma^{(i)} \cap \{ t \leq 0 \} \cap \overline{P_1 (0)} \subseteq P_{\rho} \left( \overline{\Sigma_0^-} \cap \overline{P_2 (0)} \right) \subseteq \bigcup_j \overline{P_{2 \rho_j} (Z_j)}.$$
For each $i \geq i_0$, we restrict this covering to those cylinders in $\{ \overline{P_{2 \rho_j} (Z_j)} \}_j$ that  intersects $D_{1/r_i} (A - Z_0) \cap \{ t \leq 0 \} \cap \overline{P_1 (0)}$.   Changing the centers of these cylinders  if necessary, we  obtain  $$D_{1/r_i} (A - Z_0) \cap \{ t \leq 0 \} \cap \overline{P_1 (0)} \subseteq \bigcup_j \overline{P_{4 \rho_j} (Z_j)}, \text{ with } Z_j \in D_{1/r_i} (A - Z_0) \cap \{ t \leq 0 \} \cap \overline{P_1 (0)}.$$
Notice that the choice of the cylinders  and  its centers  may depend on $i$. Using the subadditivity and scaling property of $\p^s$, it holds for each $i \geq i_0$  that
\begin{align*}
    \p^s \left( D_{1/r_i} (A - Z_0) \cap \{ t \leq 0 \} \cap \overline{P_1 (0)} \right) \leq& \sum_j \p^s ( D_{1/r_i} (A - Z_0) \cap \overline{P_{4 \rho_j} (Z_j)})\\
    \leq& \sum_j r_i^{-s} \p^s \left( A \cap \overline{P_{4 r_i \rho_j} (Z_0 + D_{r_i} Z_j)} \right)\\
    \leq& \sum_j (4 \rho_j)^s \frac{\p^s \left( A \cap \overline{P_{4 r_i \rho_j} (Z_0 + D_{r_i} Z_j)} \right)}{(4 r_i \rho_j)^s}.
\end{align*}
Choosing  $i > i_0$ sufficiently  large such  that $4 r_i \delta < 1$,   the definition of $A=A_M$ yields  $$\frac{\p^s (A \cap \{ t \leq t_0 \} \cap \overline{P_{r_i} (Z_0)})}{r_i^s} = \p^s \left( D_{1/r_i} (A - Z_0) \cap \{ t \leq 0 \} \cap \overline{P_1 (0)} \right) \leq 4^s M \alpha.$$
Since $\alpha > 0$ is arbitrary,  the above inequality contradicts  to the fact that $$\lim_{i \to \infty} \frac{\p^s (A \cap \overline{P_{r_i} (Z_0)} \cap \{ t \leq t_0 \})}{r_i^s} > 0.$$
This proves Claim 1. 

\medskip 

It is straightforward to see that for every $r > 0$, $\E(V_0, (0, 0), r)$ is independent of $r$. In fact we have $$\E (V_0, (0, 0), r) = \lim_{s \searrow 0} \E (U, Z_0, s) \geq \eps_0^2,$$
This in particular implies  that $V_0$ is non-constant on $\overline{\RNp} \times (- \infty, 0)$ and $(0, 0) \in \Sigma_0$. Moreover, the monotonicity inequality \eqref{eq:rescaled monotonicity} centered at the point $(0, 0)$ shows
\begin{equation}\label{eq: monotonicity derivative zero at zero}
    2 t \partial_t V_0 + (X \cdot \nabla_X) V_0 = 0,\quad \text{ a.e. on } \overline{\RNp} \times (- \infty, 0).
\end{equation}
By Claim 1, we have $\p^s \left( \overline{\Sigma_0^-} \right) > 0$. Repeating the rescaling  procedure that produced $V_0$ from $U$, we can   choose a point $Z_1 = (X_1, t_1) \in \overline{\Sigma_0^-} \backslash \{ (0, 0) \}$, with  properties  analogous  to that of $Z_0$, together with a sequence $s_i \searrow 0$ playing the role of  $(r_i)$. We then define another  rescaling $$ V_0^i (X, t) := V_0 (X_1 + s_i X, t_1 + s_i^2 t),\quad  \text{for } (X, t) \in \overline{\RNp} \times \R. $$
As before, applying Theorem \ref{no harmonic S1_intro}, either we get the desired bubble, or up to a subsequence
\begin{itemize}
    \item $V_0^i \rightharpoonup V_1, \text{ weakly in } H^1_{loc} \left( \overline{\RNp} \times \R; \R^L \right)$, for some limit map $V_1$,
    \item $V_1$ is a weak limit of a  sequence of solutions to \eqref{eq:approx_ext_+ general target} satisfying \eqref{eq:scale_inv_bds}, such that the spatial derivatives converges locally strongly in $L^2$,
    \item $V_1$ satisfies a monotonicity inequality similar to \eqref{eq:rescaled monotonicity},
    \item $\nabla_X V_0^i \to \nabla_X V_1, \text{ strongly in } L^2_{loc} \left( \overline{\RNp} \times \R; \R^L \right)$.
\end{itemize}
Let us proceed with the assumption of the above mentioned properties. Suppose $\Sigma_1 := sing(V_1)$ is the singular set of $V_1$. Following the same arguments as before, it holds that $$ \p^s \left( \overline{\Sigma_1^-} \cap \overline{P_2 (0)} \right) > 0, \quad (0, 0) \in \Sigma_1 ,$$  and  $\E (V_1, (0, 0), r) $ is strictly positive and is independent of $r>0$. Using \eqref{eq:rescaled monotonicity} once more we infer $$ 2 t \partial_t V_1 + (X \cdot \nabla_X) V_1 = 0, \quad \text{ a.e. on } \overline{\RNp} \times (- \infty, 0).$$
Rewriting  \eqref{eq: monotonicity derivative zero at zero} in terms of $V_0^i$ yields  $$2 s_i^{- 2} (t_1 + s_i^2 t) (\partial_t V_0^i) + s_i^{- 1} \left( (X_1 + s_i X) \cdot \nabla_X \right) V_0^i = 0,\quad \text{ in } \overline{\RNp} \times ( - \infty, 0).$$
Passing to the limit as $i \to \infty$ we obtain \begin{itemize}
\item if $t_1 < 0$ then $\partial_t V_1 \equiv 0$ in $\overline{\RNp} \times (- \infty, 0]$,
\item  if $t_1 = 0$ then $(X_1 \cdot \nabla_X) V_1 \equiv 0$ in $\overline{\RNp} \times (- \infty, 0]$, that is,  $V_1$ is independent of the $X_1$ direction in $\overline{\RNp} \times (- \infty, 0)$. In particular, since $(0,0)\in\Sigma_1$, this spatial invariance  yields $(\lambda X_1, 0) \in \Sigma_1$ for   every  $\lambda \in \R$. \end{itemize}  

\medskip 

Depending on the time dependence of $V_1$, we are led to consider the following two cases:

\medskip 

\noindent \underline{\textbf{Case 1:}} $\partial_t V_1 \equiv 0$ in $\overline{\RNp} \times (- \infty, 0)$.

As before, we may choose a suitable point $Z_2 = (X_2, t_2) \in \overline{\Sigma_1^-} \setminus \{ (0, 0) \}$, and perform  rescalling of  $V_1$ around $Z_2$.  Since  $\overline{\Sigma_1^-} = \Sigma_1^0 \times (- \infty, 0]$, in this case we can choose  $Z_2 \in \Sigma_1^0$. Consequently,  the limit $V_2$ of the rescaled sequence is invariant under translation along the $X_2$ direction. Moreover,  $V_2$ satisfies all structural   properties analogous to $V_1$, and additionally   it is independent of $t$ on $\overline{\RNp} \times ( - \infty, 0]$.

Set  $\Sigma_2 := sing(V_2)$. From   $\p^s \left( \overline{\Sigma_2^-} \right) > 0$ and  the product structure $\overline{\Sigma_2^-} = \Sigma_2^0 \times (- \infty, 0]$, we infer  that  $\H^{s - 2} (\Sigma_2^0) > 0$. Since $s > m - 1$,   $\Sigma_2^0$ contains at least $m - 2$  linearly independent vectors. We may therefore select a suitable point $Z_3 = (X_3, 0) \in \Sigma_2$ with $X_3$ linearly independent of $X_2$, and perform a rescaling of $V_2$ around $Z_3$. The resulting limit $V_3$ is then independent of time and invariant under translations along both $X_2$ and $X_3$ directions. We also remark that these directions $X_2, X_3$ lie in $\partial \RNp$ since they are being chosen from the energy concentration set, which is contained in $\partial \RNp \times \R$. 

We keep on iterating this procedure, making the limit function independent of a new spatial direction at each stage, when restricted to $\overline{\RNp} \times (- \infty, 0)$. At any of the intermediate steps, if $\overline V$ is the limit function and $\overline{\Sigma} := sing(\overline{V})$, then arguing as earlier, we conclude that $\overline{\Sigma}^0$ contains at least $m - 2$ linearly independent vectors. We can proceed with the iteration as long as we get some vector $\overline X \in \overline{\Sigma}^0$ along which $\overline V$ is not translation invariant on $\overline{\RNp} \times (-\infty, 0)$. We stop when such a choice becomes impossible. Clearly, such a situation must occur in a finite number of steps due to the finiteness of the number of linearly independent directions. This finally yields a limiting map $V$ that satisfies the following properties:

\begin{itemize}
    \item on $\overline{\RNp} \times (-\infty, 0)$, $V$ is independent of $t$ and $m - 2$   linearly independent spatial directions, whence we may regard $V$ as a function defined on $\overline{\R^3_+}$,

    \item $\Sigma_V := sing(V)$ contains $(0, 0)$, hence in particular $V$ is non-constant,

    \item $V$ is 0-homogeneous, that is,  $(X \cdot \nabla_X) V \equiv 0$ on $\R^3_+$, which follows from the analogue of \eqref{eq: monotonicity derivative zero at zero} for $V$,

    \item $V$ is a harmonic map with free boundary on $N$.
\end{itemize}
Set  $$\S^2_+ := \{ (x, y, z) \in \R^3 | \, x^2 + y^2 + z^2 = 1 \text{ and } z > 0 \},$$
and consider the restricted map  $ \tilde V := V |_{\overline{\S^2_+}}: \overline{\S^2_+} \to \R^L$.  It then follows from the  harmonicity and $0$-homogeneity of $V$  that $\tilde V$ is a non-constant harmonic map  from $\overline{\S^2_+}$ into $\R^L$ with its free boundary on $N$. By conformal equivalence, this yields a non-constant free-boundary harmonic map from $\mathbb{R}^2_+$ into $\mathbb{R}^L$,    with free boundary   on $N$. This completes the proof in this case.

\medskip 

\noindent\underline{\textbf{Case 2:}}  $\partial_t V_1\not\equiv0$  in $\overline{\RNp} \times (- \infty, 0)$.

In this case we must have $t_1 = 0$,  and on $\overline{\RNp} \times (-\infty, 0)$, $V_1$ is invariant under translations along the $X_1$ direction. We will prove that after finite number of suitable rescalings, we must fall back to Case 1.

As in Case 1, we start by choosing a suitable point $Z_2 $ around which we perform a second rescaling. Recall that  $\p^s \left( \overline{\Sigma_1^-} \right) > 0$. If $\p^s \left( \Sigma_1^- \right) > 0$ then we can   choose the point $Z_2$ from $\Sigma_1^-$. Rescalings around $Z_2$ yields a  limit function $V_2$, which  is independent of $t$, thereby falling back to Case 1, whence we are done. So we consider the situation when $\p^s \left( \Sigma_1^- \right) = 0$. Then necessarily  $\p^s \left( \Sigma_1^0 \right) > 0$. Since $s > m - 1$, the set $\Sigma_1^0$ contains    at least $m$ linearly independent vectors. Continuing this rescaling procedure inductively, we either fall back to Case 1 at an intermediate step or, after finitely many steps, arrive at a non-trivial limit map $V$ satisfying the following properties: 
\begin{itemize}
    \item on $\overline{\RNp} \times (-\infty, 0)$, $V$ is translation invariant in $m$ linearly independent spatial directions, allowing us to view it simply as a function on  $(0, \infty) \times \R$,

    \item $\Sigma_V := sing(V)$ contains $(0, 0)$, and  in particular, $V$ is non-constant on $ \{ t < 0 \}$,

    \item $\E (V, (0, 0), r)$ is independent of $r > 0$,

    \item $V$ is self-similar, that is,  $$V(y, t) = \psi \left( \frac{y}{\sqrt{- t}} \right), \quad  \forall (y, t) \in (0, \infty) \times (- \infty, 0),$$ which follows from the analogue  of \eqref{eq: monotonicity derivative zero at zero} for $V$,

    \item $V$ satisfies the heat equation  $$\partial_t V-\partial_{yy}V=0,\quad\text{in }(0,\infty)\times(-\infty,0).$$
\end{itemize}
By standard interior regularity,   $\psi$ is smooth in $(0, \infty)$.  Moreover, $\psi$ satisfies $$\psi^{\prime \prime} (y) - \frac{y}{2} \psi^{\prime} (y) = 0, \quad \text{ in } (0, \infty).$$ Since $\psi$ is bounded, it must be constant, a contradiction to the non-constancy of $V$. 
\end{proof}

\appendix 

\section{}\label{Appendix A} 

A related Liouville-type result for sphere targets appears in \cite[Remark 5.3]{MS}.  The proof presented below for general targets follows the same spirit.

\begin{prop}\label{Proposition A.2} Let $F$ be given by  \eqref{eq:defn of perturbation}. Let  $V \in L^{\infty} \cap C^1 \left( \overline{\R^2_+}; \R^L \right) $ be a solution of
\begin{equation}\left\{
    \begin{array}{lll}
        \Delta V &= 0, & \text{ in } \R^2_+,  \\
        \rule{0cm}{0.5cm} \partial_y V &= \frac{1}{c^2} (\nabla F) (V), & \text{ on } \partial \R^2_+,
    \end{array}\right.
\end{equation}
for some $c \in (0, \infty)$.
If $V$ satisfies \begin{equation}\label{eq:finite GL energy}
    \int_{\R^2_+} |\nabla V|^2 \, dX + \int_{\partial \R^2_+} F(V) \, dx < \infty,
\end{equation}then $V$ is constant and $F(V) \equiv 0$.
\end{prop}

\begin{proof}
    Consider any $\phi \in C_c^{\infty} (\R^{2})$ with $\phi \equiv 1$ on $ B_1(0)$ and $\textnormal{spt}(\phi) \subseteq B_2(0)$. For each $r > 0$, define $$ \phi_r(X) := \phi \left( \frac X r \right).$$ Using $ \phi_r (X \cdot \nabla) V$ as a test function in the equation satisfied by $V$ and integrating by parts (and using the Einstein summation convention) we get
\begin{align*}
	0 =& - \int_{\R^2_+} \phi_r \: \Delta V \cdot \left( (X \cdot \nabla) V \right) = - \int_{\R^2_+} \phi_r \: \partial_{kk} V^i \:  X_j \: \partial_j V^i\\
=& \int_{\R^2_+} \partial_k V^i \: \left( X_j \: \partial_k \phi_r \: \partial_j V^i + \delta_{jk} \: \phi_r \: \partial_j V^i + \phi_r \: X_j \: \partial_k \partial_j V^i \right) - \int_{\partial \R^2_+} \phi_r \: \partial_k V^i \: ( - \delta_{k2} ) X_j \: \partial_j V^i\\
=& \int_{\R^2_+} X_j \: \partial_k \phi_r \: \partial_k V^i \: \partial_j V^i + \int_{\R^2_+} \phi_r \: \partial_k V^i \: \partial_k V^i + \frac 1 2 \int_{\R^2_+} X_j \: \phi_r \: \partial_j \left( |\nabla V|^2 \right) + \int_{\partial \R^2_+} X_j \: \phi_r \: \partial_j V^i \: \partial_2 V^i\\
=& \int_{\R^2_+} X_j \: \partial_k \phi_r \: \partial_k V^i \: \partial_j V^i + \int_{\R^2_+} \phi_r \: |\nabla V|^2 - \frac 1 2 \int_{\R^2_+} (2 \phi_r + X_j \: \partial_j \phi_r) |\nabla V|^2\\
& + \frac 1 2 \int_{\partial \R^2_+} X_j \: \phi_r |\nabla V|^2 (- \delta_{j 2}) + \frac{1}{c^2} \int_{\partial \R^2_+} X_1 \: \phi_r \: \partial_1 V^i (\partial_i F) (V)\\
=& \int_{\R^2_+} X_j \: \partial_k \phi_r \: \partial_k V^i \: \partial_j V^i - \frac 1 2 \int_{\R^2_+} X_j \: \partial_j \phi_r \: |\nabla V|^2 + \frac{1}{c^2} \int_{\partial \R^2_+} X_1 \: \phi_r \: \partial_1 \left( F(V) \right)\\
=& \int_{\R^2_+} X_j \: \partial_k \phi_r \: \partial_k V^i \: \partial_j V^i - \frac 1 2 \int_{\R^2_+} X_j \: \partial_j \phi_r \: |\nabla V|^2 - \frac{1}{c^2} \int_{\partial \R^2_+} ( \phi_r + X_1 \: \partial_1 \phi_r) F(V).
\end{align*}
Letting $r \to \infty$ and using \eqref{eq:finite GL energy}, we get that $\int_{\partial \R^2_+} F(V) = 0$, i.e. $F(V) \equiv 0$. Hence $V$ satisfies
\begin{equation}\left\{
\begin{array}{ll}
\D V = 0, &\text{ in } \R_+^{2},\\
\rule{0cm}{0.5cm} \partial_2 V = 0, &\text{ on } \partial \R^2_+.
\end{array}\right.
\end{equation}
Let $\overline V$ denote an extension of $V$ to the whole of $\R^2$ using an even reflection around $\partial \R^2_+$. It can then be checked that $\overline V$ is a harmonic function from $\R^2$ to $\R^L$. As $V$ is bounded, so is $\overline V$ and hence it's constant, completing the proof.
\end{proof}

 Although the following proposition is standard, we include a proof for completeness as we could not find a precise reference in the literature.

\begin{prop}\label{Propsition A.1}
Let $(f_n)$ be a sequence of   non-negative functions on $\R^{m+1}\times(-\infty,0)$. Suppose that $(f_n)$ is locally bounded in  $ L^1( \R^{m+1}\times(-\infty,0))$, and   
\begin{align}
f_n\,dX \,dt  \overset{*}{\rightharpoonup} \mu
\end{align}
as Radon measures on $\mathbb{R}^{m+1}\times(-\infty,0)$. Assume that, for every
$\phi\in C_c^\infty\bigl(\mathbb{R}^{m+1}\times(-\infty,0)\bigr)$
and every $j=1,\ldots,m-1$,
\begin{align}
\lim_{n\to\infty}
\int_{\mathbb{R}^{m+1}\times(-\infty,0)}
\partial_{X_j}\phi\, f_n\,dX\,dt=0.
\end{align}
Then there exists a non-negative Radon measure $\nu$ on $  \mathbb{R}^2\times(-\infty,0)$  such that \begin{align} \mu = \H^{m-1} \times \nu. \end{align} Equivalently, for every $\phi\in C_c(\mathbb{R}^{m+1}\times(-\infty,0))$, \begin{align} \int_{\mathbb{R}^{m+1}\times(-\infty,0)} \phi\,d\mu = \int_{\mathbb{R}^2\times(-\infty,0)} \int_{\mathbb{R}^{m-1}} \phi(X_1,\ldots,X_{m+1},t)
\,dX_1\cdots dX_{m-1}\,d\nu(X_m,X_{m+1},t). \end{align}
\end{prop}

\begin{proof}
First, by the weak* convergence of $f_n\,dX\,dt$ to $\mu$, for every
$\phi\in C_c^\infty\bigl(\mathbb{R}^{m+1}\times(-\infty,0)\bigr)$
and every $j=1,\ldots,m-1$,
\begin{align}
\int_{\mathbb{R}^{m+1}\times(-\infty,0)}
\partial_{X_j}\phi\,d\mu
&=
\lim_{n\to\infty}
\int_{\mathbb{R}^{m+1}\times(-\infty,0)}
\partial_{X_j}\phi\,f_n\,dX\,dt
=0.
\end{align}
Hence
\begin{align}
\partial_{X_j}\mu=0,
\qquad\text{in }\mathcal{D}'
\bigl(\mathbb{R}^{m+1}\times(-\infty,0)\bigr),
\qquad j=1,\ldots,m-1.
\end{align}
Write $x'=(X_1,\ldots,X_{m-1})$ and  $y=(X_m,X_{m+1})$, so that   $(X,t)=(x',y, t)$. 
For every  non-negative $ 
\eta\in C_c^\infty\bigl(\mathbb{R}^2\times(-\infty,0)\bigr),$ we define a non-negative Radon measure $\mu_\eta$ on $\mathbb{R}^{m-1}$ by
\begin{align}
\int_{\mathbb{R}^{m-1}}\psi(x')\,d\mu_\eta(x') := \int_{\mathbb{R}^{m+1}\times(-\infty,0)} \psi(x')\eta(y, t)\,d\mu(x',y, t),\quad \text{for every } \psi\in C_c(\mathbb{R}^{m-1}).\end{align}   We claim that \begin{align} \partial_{X_j}\mu_\eta=0, \qquad\text{in }\mathcal{D}'(\mathbb{R}^{m-1}), \qquad j=1,\ldots,m-1. \end{align} Indeed, for every
$\psi\in C_c^\infty(\mathbb{R}^{m-1})$, \begin{align} \int_{\mathbb{R}^{m-1}}
\partial_{X_j}\psi(x')\,d\mu_\eta(x') &= \int_{\mathbb{R}^{m+1}\times(-\infty,0)}
\partial_{X_j}\psi(x')\eta(y, t)\,d\mu(x',y, t) =0, \end{align} where the last equality follows from $\partial_{X_j}\mu=0$.

Since $\mu_\eta$ is a non-negative Radon measure on $\mathbb{R}^{m-1}$ whose distributional derivatives all vanish, it must be a constant multiple of Lebesgue measure. Hence there exists a constant $c_\eta\geq0$ such that $  \mu_\eta=c_\eta\,\H^{m-1}.$ In other words, \begin{align} \int_{\mathbb{R}^{m+1}\times(-\infty,0)} \psi(x') \, \eta(y, t)\,d\mu(x',y, t) = c_\eta \int_{\mathbb{R}^{m-1}}\psi(x')\,dx', \quad \forall \psi\in C_c^\infty(\mathbb{R}^{m-1}). \end{align}
Notice that the constant  $c_\eta$ depends only on $\eta$ and $\mu$, and the map
\begin{align} \eta\longmapsto c_\eta \end{align}
is a positive linear functional on
$C_c^\infty(\mathbb{R}^2\times(-\infty,0))$. Therefore, by the Riesz representation theorem, there exists a unique non-negative Radon measure $\nu$ on $\mathbb{R}^2\times(-\infty,0)$  such that \begin{align} c_\eta = \int_{\mathbb{R}^2\times(-\infty,0)}
\eta(y, t)\,d\nu(y, t). \end{align} Consequently,  for every $\psi\in C_c^\infty(\mathbb{R}^{m-1})$ and $\eta\in C_c^\infty(\mathbb{R}^2\times(-\infty,0))$, \begin{align} \int_{\mathbb{R}^{m+1}\times(-\infty,0)} \psi(x')\eta(y, t)\,d\mu(x',y, t) &= \left( \int_{\mathbb{R}^{m-1}}\psi(x')\,dx'
\right) \left( \int_{\mathbb{R}^2\times(-\infty,0)} \eta(y, t)\,d\nu(y, t) \right)\\&=\int_{\mathbb{R}^{m+1}\times(-\infty,0)}
\psi(x')\eta(y, t)\,
d\bigl(\H^{m-1}\times\nu\bigr)
\end{align}
 Since finite linear combinations of such product functions are dense in $C_c\bigl(\mathbb{R}^{m+1}\times(-\infty,0)\bigr)$, we conclude that  $\mu=\H^{m-1}\times\nu$.
\end{proof}

\addtocontents{toc}{\protect\setcounter{tocdepth}{1}}
\subsection*{Preliminaries on Approximate Continuity}
In this subsection we briefly collect the definition and key properties of approximate continuity relevant to  our analysis. For a comprehensive discussion in a broader setting, see \cite[2.9.12]{F}.
\begin{defn}[Approximate Continuity]\label{defn_appx_cty}
    Let $M \in \Zp$ and $\sigma$ be a Radon measure on $\R^M$. A Borel measurable function $\phi : \R^M \to \R$ is said to be $\sigma$-approximately continuous at $x_0 \in \R^M$ if for every $\eps > 0$ $$\lim_{r \to 0} \frac{\sigma \left( \{ x \in \overline{B_r (x_0)} | \: |\phi(x) -\phi(x_0)|\geq \eps \} \right) }{ \sigma \left( \overline{B_r(x_0)} \right) } = 0.$$
\end{defn}
Proof of the following   result    can be found in \cite[2.9.13]{F}.
\begin{prop}\label{appx_cty_iff}
    Let $\sigma$ be a Radon measure on $\R^M$. Then every real valued Borel measurable function on $\R^M$ is $\sigma$-approximately continuous at $\sigma$-almost every point of $\R^M$.
\end{prop}
\addtocontents{toc}{\protect\setcounter{tocdepth}{4}}

We now prove a variant of the geometric lemma \cite[Lemma 2.4]{Lin-annals} in our framework.
 
\begin{lem}\label{geometric lemma}
Let $\sigma$ be a non-negative Radon measure on $\R^m$ such that for some integer $n \in \{1, \dots,m-1\}$, $$\H^n (\spt(\sigma)) < \infty,$$ and
\begin{equation}\label{eq:bdd_disc_up_density}
\sigma \left( \overline{B_r(x)} \right)\leq \alpha r^{n}, \quad \forall x\in B_1, \, 0<r\leq1.
\end{equation}
Set $$ \theta(x) := \limsup_{r \searrow 0} r^{-n} \sigma \left( \overline{B_r(x)} \right).$$
Assume that $x_0\in B_1$ is such that $\theta(x_0)>0$, $\theta$ is $\H^n \lfloor_{\spt(\sigma)}$-approximately continuous at $x_0$ and 
\begin{equation}\label{eq:bdd_up_density}
\limsup_{r \searrow 0} r^{- n} \H^n \lfloor_{\spt(\sigma)} \left( \overline{B_r (x_0)} \right) < \infty.
\end{equation}
Then there exists a constant $\delta = \delta(n, x_0, \alpha) \in (0, \frac 12)$, a sequence of real numbers $s_j \searrow 0$ and $n$-many sequences $(x_1^j)_j, \ldots, (x_n^j)_j$ in $\spt(\sigma)$ such that
\begin{enumerate}
    \item $\delta s_j \leq |x^j_k - x_0| \leq s_j, \forall j \in \Zp, \forall k \in \{1, \ldots, n\}$,

    \item $\textnormal{dist} \left(x^j_k - x_0, \textnormal{span}\{ x^j_1 - x_0, \ldots, x^j_{k - 1} - x_0 \} \right) \geq \delta s_j, \forall j \in \Zp, \forall k \in \{ 2, \ldots, n\},$

    \item $\theta(x^j_k) \geq \theta(x_0) - \eps_j, \forall j \in \Zp, \forall k \in \{ 1, \ldots, n\}, \text{ and some sequence } \eps_j \searrow 0$. 
\end{enumerate}
\end{lem}

\begin{proof}
    Let us choose a sequence $s_j \searrow 0$ in such a way that $$\theta(x_0) = \lim_{j \to \infty} s_j^{- n} \sigma \left( \overline{B_{s_j}(x_0)} \right) > 0.$$
    Denote $S := \spt(\sigma)$. The $\H^n \lfloor_S$-approximate continuity of $\theta$ at $x_0$ and \eqref{eq:bdd_up_density} imply  that for each $\eps > 0$, $$\lim_{r \to 0} r^{- n} \H^n \lfloor_S \left( \{ x \in \overline{B_r(x_0)} | \: |\theta(x) - \theta(x_0)| > \eps \} \right) = 0.$$
    This allows us to choose a sequence $\eps_j \searrow 0$ so that possibly after passing to a subsequence
    \begin{equation}\label{eq:approx_cts}
    s_j^{- n} \H^n \lfloor_S \left( \{ x \in \overline{B_{s_j} (x_0)} | \: |\theta(x) - \theta(x_0)| \geq \eps_j \} \right) \leq \delta,
    \end{equation}
    for some fixed $\delta \in (0, \frac{1}{2})$ whose precise value is to be determined later.

    Now for each $j \in \Zp$, we want to choose $n$ number of points $x^j_1, \ldots, x^j_n \in \overline{B_{s_j}(x_0)}$ that satisfy the above mentioned relations $(1)\text{-}(3)$. For the sake of contradiction suppose that this is not possible for every $j \in \Zp$. It is then straightforward to see that for each $j \in \Zp$, there exists an $(n - 1)$-dimensional affine subspace $P_j \subseteq \R^m$ which passes through $x_0$ and satisfy $$\{ x \in S \cap \overline{B_{s_j} (x_0)} | \: |\theta(x) - \theta(x_0)| < \eps_j \} \subseteq B_{\delta s_j} ( P_j).$$
    Therefore for each $j \in \Zp$, $$S \cap \overline{B_{s_j} (x_0)} \subseteq \left\{ x \in S \cap \overline{B_{s_j} (x_0)} | \: |\theta(x) - \theta(x_0)| \geq \eps_j \right\} \cup B_{\delta s_j} \left( P_j \cap \overline{B_{s_j} (x_0)} \right).$$
    Let us estimate the $\sigma$ measure of both the sets on the right-hand side. By standard geometric measure theoretic results, together with \eqref{eq:bdd_disc_up_density} and  \eqref{eq:approx_cts} we deduce that
    \begin{align*}
        &\sigma \left( \{ x \in S \cap \overline{B_{s_j} (x_0)} | \: |\theta(x) - \theta(x_0)| \geq \eps_j \} \right)\\
        &\leq C \alpha \H^n \left( \{ x \in S \cap \overline{B_{s_j} (x_0)}| \: |\theta(x) - \theta(x_0)| \geq \eps_j \} \right)\\
        &\leq C \alpha \delta s_j^n.
    \end{align*}

    \noindent Next note that $B_{\delta s_j} \left( P_j \cap \overline{B_{s_j} (x_0)} \right)$ can be covered by $N$-many balls $\left\{ \overline B_i \right\}_{i = 1}^N$ of radius $\delta s_j$, with $N \leq C \delta^{1 - n}$. Aided with \eqref{eq:bdd_disc_up_density} we compute
    \begin{align*}
        \sigma \left( B_{\delta s_j} \left(P_j \cap \overline{B_{s_j}(x_0)} \right) \right) \leq& \sum_{i = 1}^N \sigma \left( \overline B_i \right) \leq C \alpha (\delta s_j)^n \delta^{1 - n} = C \alpha \delta s_j^n.
    \end{align*}
Combining these  estimates we conclude
    \begin{align*}
        \sigma \left( \overline{B_{s_j} (x_0)} \right) \leq& \sigma \left( \{ x \in S \cap \overline{B_{s_j}(x_0)} | \: |\theta(x) - \theta(x_0)| \geq \eps_j \}\right) + \sigma \left( B_{\delta s_j} \left( P_j \cap \overline{B_{s_j} (x_0)} \right) \right)\\
        \leq& C \alpha \delta s_j^n < \frac{1}{2} \theta(x_0) s_j^n,\quad \forall j \gg 1,
    \end{align*}
    when $\delta$ is chosen sufficiently small. But this is a contradiction to the choice of $(s_j)$. This finishes the proof.
\end{proof}

\section*{Acknowledgments}
A.D. and A.H. acknowledge the support of the Department of Atomic Energy, Government of India, under
Project Identification No. RTI 4014.  Y.S. is partially supported by NSF DMS Grant $2154219$, " Regularity {\sl vs} singularity formation in elliptic and parabolic equations".

\bibliographystyle{alpha}
\bibliography{wed_hmhf,wed_hmhf2,biblio}

\end{document}